\documentclass{amsart}

   \renewcommand{\footnote}[1]{
\textsuperscript{
\addtocounter{footnote}{1}
(\thefootnote)
}
\footnotetext{#1}
}

\usepackage{url}
\usepackage{bm}
\usepackage{amsmath}
\usepackage{amsthm}
\usepackage{trfsigns}
\usepackage{wasysym}
\usepackage{amssymb}
\usepackage{amsfonts}
\usepackage{graphicx}  
\usepackage[frenchb,english]{babel}
\usepackage[utf8]{inputenc}
\usepackage[T1]{fontenc}
\usepackage{mathrsfs}
\usepackage{pdfsync}
\usepackage{enumerate}
\usepackage{version}
\usepackage{calc}
\usepackage{subfigure}

 \usepackage{pstricks,pstricks-add,pst-math,pst-xkey}
 \usepackage{float,graphics,yhmath}
 \usepackage{indentfirst}
 \usepackage[normalem]{ulem}
 \usepackage[colorlinks=true,linkcolor=blue,urlcolor=blue,citecolor=blue]{hyperref}
 
\newtheorem{prop}{Proposition}

\newtheorem{thm}{Theorem}

\newtheorem{lem}[prop]{Lemma}
\newtheorem{cor}[prop]{Corollary}

\theoremstyle{remark}
\newtheorem{remark}[prop]{Remark}
\newtheorem{notation}[prop]{Notation}

\newtheorem{defi}[prop]{Definition}

\def\mint{\mathop{\,\,\rlap{--}\!\!\!\int}\nolimits}

\def\v{\varepsilon}

\def\W{\mathcal{W}}
\def\O{\Omega}

\def\p{\partial}

\def\o{\omega}
\def\R{\mathbb{R}}
\def\C{\mathbb{C}}
\def\S{\mathbb{S}}
\def\Div{{\rm  div}}
\def\deg{ {\rm deg}}
\def\n{\nabla}
\def\dist{{\rm dist}}
\def\tr{{\rm tr}}

\def\E{\mathcal{E}}

\def\N{\mathbb{N}}
\def\D{\mathbb{D}}
\def\Z{\mathbb{Z}}
\def\J{\mathcal{J}}
\def\K{ \wideparen{\mathcal{K}}}
\def\Krit{\mathscr{K}}
\def\I{\mathcal{I}}
\def\M{\mathcal{M}}

\def\1{\textrm{1\kern-0.25emI}}
\def\num{P}
\def\Ring{\mathscr{R}}
\def\H{\mathscr{H}}

\def\z{\zeta}
\def\dom{\mathcal{D}}

\def\di{\displaystyle}

\def\B{\mathcal{B}}
\def\A{\mathcal{A}}
\def\h{h_{\rm ex}}
\def\rot{{\rm curl}}
\def\F{\mathcal{F}}

\def\Card{{\rm Card}}

\def\eqj{\stackrel{{\rm gauge}}{\sim}}

\def\Haus{\mathcal{H}}

\def\rad{r}
\def\Rad{{\tilde r}}
\def\zd{{\bf (z,d)}}

\def\wstar{w^\zd_\star}
\def\wr{w^\zd_\Rad}
\def\k{{\bf k}}
\def\Pstar{\Phi^\zd_\star}
\def\Cr{\mathscr{C}}
\def\SO{\mathcal{S}_\O}
\def\mod{{\chi}}
\def\tae{\hbar}
\def\Ostar{(\O^N)^*} 

\def\LamN{\Lambda_{d}}
\def\pD{{({\bf p},{\bf D})}}

\def\ab{s_b}

\def\ad{{({\bf a},{\bf d})}}
\def\borneh{C_0}

\def\Zz{\mathcal{Z}}
\def\Wmin{\overline{\W}}
\def\Pic{M_\O}
\def\L{\mathscr{L}}
\def\Jo{{\bf J_0}}
\def\Critical{{\tt K}^{\tt(I)}}
\def\Criticalbis{{\tt K}^{\tt(II)}}
\def\vtest{v^{0}}
\def\Atest{A^{0}}
\def\DLU{\Delta^{(1)}}
\def\DLD{\Delta^{(2)}}
\def\dtest{d_{0}}

\def\HoC{H^0_{c_1} }
\def\U{\mathscr{U}}
\def\V{\mathscr{V}}
\date{}

\title[Pinned magnetic Ginzburg-Landau energy]{Magnetic Ginzburg-Landau energy with a periodic rapidly oscillating and diluted pinning term}
\author{Mickaël Dos Santos}
\thanks{The author would like to thank Lia Bronsard, Vincent Millot, Petru Mironescu and Etienne Sandier for fruitful discussions.
}

\address{M. Dos Santos, 
              Universit\'e Paris Est-Cr\'eteil, 61 avenue du G\'en\'eral de Gaulle, 94010 Cr\'eteil Cedex}
 \email{mickael.dos-santos@u-pec.fr} 
\begin{document}
\maketitle
\begin{abstract}
We study the $2D$ full Ginzburg-Landau energy with a periodic rapidly oscillating, discontinuous and [strongly] diluted pinning term using  a perturbative argument. This energy models the state of an heterogeneous type II superconductor submitted to a magnetic field. We calculate the value of the first critical field which links the presence of vorticity defects with the intensity of  the applied magnetic field. Then we prove a standard dependance of the quantized vorticity defects with the intensity of the applied field. Our study includes the case of a London solution having several {\it minima}. The macroscopic location of the vorticity defects is understood with the famous Bethuel-Brezis-Hélein renormalized energy. The mesoscopic location, {\it i.e.}, the arrangement of the vorticity defects around the {\it minima} of the London solution, is the same than in the homogenous case. The microscopic  location is exactly the same than in the heterogeneous case without magnetic field. We also compute the value of secondary critical fields that increment the quantized vorticity.
\end{abstract}


\section{Introdution}
This article studies the pinning phenomenon in type-II superconducting composites.

Superconductivity is a property that appears in certain materials cooled below a critical temperature. These materials are called superconductors. Superconductivity is characterized by a total absence of electrical resistance and a perfect diamagnetism. Unfortunately, when the imposed conditions are too intense, superconductivity is destroyed in  certain areas of the material called {\it vorticity defects}.

We are interested in type II superconductors which are characterized by the fact that the vorticity defects first appear in small areas.  Their number increases with the intensity of the conditions imposed until filling the material. For example, when the intensity $\h$ of an applied magnetic field exceeds a first threshold, the first vorticity defects appear: the magnetic field begins to penetrate the superconductor. The penetration is done along thin wires and may move resulting an energy dissipation. These motions may be limited by trapping the vorticity defects in small areas. 

The behavior of a superconductor is modeled by minimizers of a Ginzburg-Landau type energy. In order to study the presence of traps for the vorticity defects we consider an energy   including a pinning term that models impurities in the superconductor. These impurities would play the role of  traps for the vorticity defects. We are thus lead to the subject of this article: the type-II superconducting composites with impurities.\\

The case of an infinite long homogenous type II superconducting cylinder was intensively studied in mathematics by various authors since the 90's [see \cite{SS1} for a guide to the litterature]. Namely, the present work deals with a cylindrical superconductor $\mathcal{S}=\O\times\R$ [whose section is $\O\subset\R^2$] submitted to a vertical magnetic field $(0,0,\h)$. Under these considerations, the vorticity defects are thin vertical cylinder. Thus their study may be done {\it via}  a 2D problem formulated on $\O\subset\R^2$. Following the works of various authors [see \cite{Rub}, \cite{ASS1}, \cite{K1}], for a small parameter $\v>0$ [$\v\to0$ in this article] and $\h=\h(\v)\geq0$, we are interested in the description of the [global] minimizers of the  functional 
\[
\begin{array}{cccc}
 \E_{\v,\h}:&\H&\to&\R^+\\&(u,A)&\mapsto&\di\dfrac{1}{2}\int_\O|\n u-\imath Au|^2+\dfrac{1}{2\v^2}(a_\v^2-|u|^2)^2+|\rot(A)-\h|^2
\end{array},
\]
where [see Section \ref{SecNotation} for more detailed notation]
\begin{itemize}
\item $\O\subset\R^2$ is a smooth bounded simply connected open set,
\item $\H:=H^1(\O,\C)\times H^1(\O,\R^2)$,
\item $a_\v:\O\to\{1,b\}$ [$b\in(0,1)$ is independent of $\v$] is a periodic diluted pinning term [see Figure \ref{Intro.FigureTermeChevillage} and Section \ref{SecConstructionPinningTerm} for  a  construction of $a_\v$]. The impurities are the connected components of $\o_\v:=a_\v^{-1}(\{b\})$. In the definition of $a_\v$,  $\delta=\delta(\v)\underset{\v\to0}{\to}0$ is the parameter of period, $\lambda=\lambda(\v)\underset{\v\to0}{\to}0$ is the parameter of dilution and $0\in\o\subset\R^2$ is a smooth bounded simply connected open set which gives the form of the impurities. 
\end{itemize}
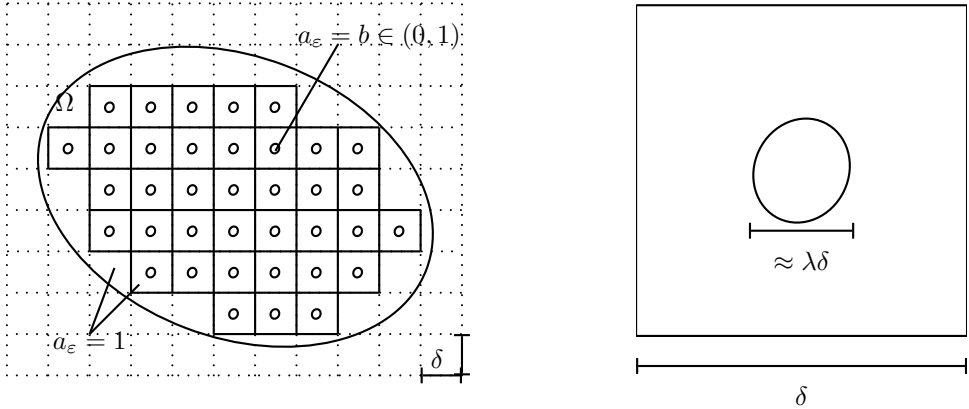
\begin{figure}[h]
\subfigure[The pining term is periodic on a $\delta\times\delta$-grid]{
\psset{xunit=0.11cm,yunit=0.11cm,algebraic=true,dotstyle=*,dotsize=3pt 0,linewidth=0.8pt,arrowsize=3pt 2,arrowinset=0.25}
\begin{pspicture*}(-20,-14.3)(36,40)
\rput{64.54}(2.39,2.43){\psellipse(0,0)(0.72,0.64)}
\psline(12.45,17.15)(20,30)
\rput(25,31){$a_\v=b\in(0,1)$}
\psline(-7,3)(-10,-5)
\rput(-10,-6){$a_\v=1$}
\psline(-4,1)(-10,-5)

\rput{64.54}(7.39,2.43){\psellipse(0,0)(0.72,0.64)}
\rput{64.54}(12.39,2.43){\psellipse(0,0)(0.72,0.64)}
\rput{64.54}(2.39,7.43){\psellipse(0,0)(0.72,0.64)}
\rput{64.54}(2.39,12.43){\psellipse(0,0)(0.72,0.64)}
\rput{64.54}(2.39,17.43){\psellipse(0,0)(0.72,0.64)}
\rput{64.54}(7.39,17.43){\psellipse(0,0)(0.72,0.64)}
\rput{64.54}(12.39,17.43){\psellipse(0,0)(0.72,0.64)}
\rput{64.54}(17.39,17.43){\psellipse(0,0)(0.72,0.64)}
\rput{64.54}(17.39,12.43){\psellipse(0,0)(0.72,0.64)}
\rput{64.54}(7.39,12.43){\psellipse(0,0)(0.72,0.64)}
\rput{64.54}(7.39,7.43){\psellipse(0,0)(0.72,0.64)}
\rput{64.54}(12.39,12.43){\psellipse(0,0)(0.72,0.64)}
\rput{64.54}(12.39,7.43){\psellipse(0,0)(0.72,0.64)}
\rput{64.54}(17.39,7.43){\psellipse(0,0)(0.72,0.64)}
\rput{64.54}(17.39,2.43){\psellipse(0,0)(0.72,0.64)}
\rput{-20.94}(7.62,11.61){\psellipse(0,0)(24.83,17.1)}

\rput{64.54}(-2.61,7.43){\psellipse(0,0)(0.72,0.64)}
\rput{64.54}(-7.61,7.43){\psellipse(0,0)(0.72,0.64)}
\rput{64.54}(-2.61,2.43){\psellipse(0,0)(0.72,0.64)}
\rput{64.54}(-2.61,12.43){\psellipse(0,0)(0.72,0.64)}
\rput{64.54}(-2.61,17.43){\psellipse(0,0)(0.72,0.64)}
\rput{64.54}(-2.61,22.43){\psellipse(0,0)(0.72,0.64)}
\rput{64.54}(-7.61,22.43){\psellipse(0,0)(0.72,0.64)}
\rput{64.54}(-7.61,17.43){\psellipse(0,0)(0.72,0.64)}
\rput{64.54}(-7.61,12.43){\psellipse(0,0)(0.72,0.64)}
\rput{64.54}(-12.61,17.43){\psellipse(0,0)(0.72,0.64)}
\rput{64.54}(2.39,22.43){\psellipse(0,0)(0.72,0.64)}
\rput{64.54}(7.39,22.43){\psellipse(0,0)(0.72,0.64)}
\rput{64.54}(12.39,22.43){\psellipse(0,0)(0.72,0.64)}
\rput{64.54}(22.39,12.43){\psellipse(0,0)(0.72,0.64)}
\rput{64.54}(22.39,7.43){\psellipse(0,0)(0.72,0.64)}
\rput{64.54}(27.39,7.43){\psellipse(0,0)(0.72,0.64)}
\rput{64.54}(22.39,2.43){\psellipse(0,0)(0.72,0.64)}
\rput{64.54}(7.39,-2.57){\psellipse(0,0)(0.72,0.64)}
\rput{64.54}(12.39,-2.57){\psellipse(0,0)(0.72,0.64)}
\rput{64.54}(17.39,-2.57){\psellipse(0,0)(0.72,0.64)}
\psline[linestyle=dotted,dash=18pt 18pt](-20,30)(35,30)
\psline[linestyle=dotted,dash=18pt 18pt](35,25)(-20,25)
\psline[linestyle=dotted,dash=18pt 18pt](-20,20)(35,20)
\psline[linestyle=dotted,dash=18pt 18pt](35,15)(-20,15)
\psline[linestyle=dotted,dash=18pt 18pt](35,10)(-20,10)
\psline[linestyle=dotted,dash=18pt 18pt](35,0)(-20,0)
\psline[linestyle=dotted,dash=18pt 18pt](35,-5)(-20,-5)
\psline[linestyle=dotted,dash=18pt 18pt](35,-10)(-20,-10)
\psline[linestyle=dotted,dash=18pt 18pt](35,35)(35,-10)
\psline[linestyle=dotted,dash=18pt 18pt](30,35)(30,-10)
\psline[linestyle=dotted,dash=18pt 18pt](25,35)(25,-10)
\psline[linestyle=dotted,dash=18pt 18pt](20,35)(20,-10)
\psline[linestyle=dotted,dash=18pt 18pt](15,35)(15,-10)
\psline[linestyle=dotted,dash=18pt 18pt](10,35)(10,-10)
\psline[linestyle=dotted,dash=18pt 18pt](5,35)(5,-10)
\psline[linestyle=dotted,dash=18pt 18pt](0,35)(0,-10)
\psline[linestyle=dotted,dash=18pt 18pt](-5,35)(-5,-10)
\psline[linestyle=dotted,dash=18pt 18pt](-10,35)(-10,-10)
\psline[linestyle=dotted,dash=18pt 18pt](-15,35)(-15,-10)
\psline[linestyle=dotted,dash=18pt 18pt](-20,35)(-20,-10)
\psline[linestyle=dotted,dash=18pt 18pt](-20,35)(35,35)
\psline[linestyle=dotted,dash=18pt 18pt](-20,5)(35,5)
\psline(5,-5)(20,-5)
\psline(25,0)(-5,0)
\psline(-10,5)(30,5)
\psline(-10,10)(30,10)
\psline(25,15)(-15,15)
\psline(-15,20)(25,20)
\psline(-10,25)(15,25)

\psline(-15,20)(-15,15)
\psline(-10,5)(-10,25)
\psline(-5,0)(-5,25)
\psline(0,0)(0,25)
\psline(5,-5)(5,25)
\psline(10,25)(10,-5)
\psline(15,25)(15,-5)
\psline(20,-5)(20,20)
\psline(25,0)(25,20)
\psline(30,5)(30,10)

\rput{64.54}(22.39,17.43){\psellipse(0,0)(0.72,0.64)}
\rput(32,-8){$\delta$}
\psline{|-|}(30,-10)(35,-10)
\psline{|-|}(35,-10)(35,-5)
\rput(-13,23){$\O$}
\end{pspicture*}
}\hfill
 \subfigure[The parameter $\lambda$ controls the size of an inclusion in the cell]
 {
\psset{xunit=0.20cm,yunit=0.20cm,algebraic=true,dotstyle=*,dotsize=3pt 0,linewidth=0.8pt,arrowsize=3pt 2,arrowinset=0.25}
\begin{pspicture*}(38,4)(60,35)
\psline(38,31)(60,31)
\psline(38,31)(38,9)
\psline(38,9)(60,9)
\psline(60,9)(60,31)
\rput{64.54}(49,20){\psellipse(0,0)(3.6,3.2)}
\psline{|-|}(38,7)(60,7)
\rput(49,5){$\delta$}
\psline{|-|}(45.5,16)(52.5,16)
\rput(49,14){$\approx\lambda\delta$}
\end{pspicture*}
}
\caption{The periodic pinning term}\label{Intro.FigureTermeChevillage}
\end{figure} 

 We focus on a strongly diluted case [$\lambda^{1/4}|\ln\v|\to0$] with  not too small connected components of $\o_\v$ [$|\ln(\lambda\delta)|=\mathcal{O}(\ln|\ln\v|)$] but with a sufficiently small parameter of the period [see \eqref{PutaindHypTech}]. \\

Under these considerations, if $(u_\v,A_\v)$ minimizes $\E_{\v,\h}$, then  the vorticity defects may be interpreted as the set $\{|u_\v|< b/2\}$. 

As said above, our study takes place in the extrem type II case $\v\to0$ and we also assume a divergent upper bound for  $\h$. 
Vorticity defects appear for minimizers above a critical valued $H_{c_1}=[{b^2|\ln\v|+(1-b^2)|\ln(\lambda\delta)|}]/({2\|\xi_0\|_{L^\infty(\O)}})+\mathcal{O}(1)$ [see Corollary \ref{Cor.ExactEnergyExpPreCritField}  and \eqref{DefH0c1}]. Here $\xi_0\in H^1_0\cap H^2$ is called the {\it London solution} and is the unique solution of the {\it London equation}
\begin{equation}\label{LondonEq1}
\begin{cases}
-\Delta^2\xi_0+\Delta\xi_0=0&\text{ in }\O
\\
\Delta\xi_0=1&\text{ on }\p\O
\\
\xi_0=0&\text{ on }\p\O
\end{cases}.
\end{equation}

 The value $H_{c_1}$ is calculated by a standard balance of the energetic costs of a configuration without vorticity defects [$|u|\geq b/2$] with well prepared competitors having an arbitrary number of quantized vorticity defects. Here quantization as to be interpreted by the degree  of $u$ around a vorticity defect. It is an observable quantity related with the circulation of the superconducting current.
 
In order to lead the study, the set $\Lambda:=\{z\in\O\,|\,\xi_0(z)=\min\xi_0\}\subset\O$ is of major interest [it is standard to prove that, in $\O$, $-1<\xi_0<0$]. From Lemma 4.4 in \cite{S1} and Lemma 4 in \cite{SS2} we have the following :
\begin{lem}\label{Lem.DescriptionLambda}
The set $\Lambda$ is finite. Moreover there exist $\eta>0$ and $M\geq1$ s.t. for $a\in\O$ we have $\xi_0(a)\geq \min\xi_0+\eta\dist(a,\Lambda)^M$ \footnote{In Lemma 4 in \cite{SS2}, $M$ is just a positive number, but  $\xi\in C^0(\overline\O)$, and then, up to consider $\eta>0$ sufficiently small, we may assume $M\geq 1$.}.
\end{lem}
We write $N_0:={\rm Card}(\Lambda)$ and $\Lambda=\{p_1,...,p_{N_0}\}$.\\

We may give a simple picture of the emergence of the vorticity defects. The first vorticity defects appear close to  $H_{c_1}$. If $N_0=1$ then there is first a unique vorticity defect and it is close to $\Lambda$. If $N_0\geq2$ the situation is less clear: we first have $d_1^\star\in\{1,...,N_0\}$ vorticity defect and each of them is located close to $d_1^\star$ elements of $\Lambda$. By increasing the intensity of the applied field $\h$ by a bounded quantity we increment the number of vorticity defects until filling $\Lambda$. 

Once each elements of $\Lambda$ is close to a vorticity defect, then by increasing $\h$ of a  $\mathcal{O}(\ln|\ln\v|)$, additional defects appear one by one.\\


We may now state the main theorems of the present work. For simplicity of the presentation the theorems are not stated on their most general form [see Theorem \ref{THM}].

These main results are obtained assuming that $\lambda,\delta$ and $\h$ satisfy
\begin{equation}\label{CondOnLambdaDelta}
\lambda^{1/4}|\ln\v|\to0\text{ and }|\ln(\lambda\delta)|=\mathcal{O}(\ln|\ln\v|),
\end{equation}
\begin{equation}\label{BorneKMagn}
\text{There is $K\geq 1$ s.t. $\h\leq\dfrac{b^2|\ln\v|}{2\|\xi_0\|_{L^\infty(\O)}}+K\ln|\ln\v|$}
\end{equation}
and when $\h\to\infty$ we need
\begin{equation}\label{PutaindHypTech}
\dfrac{\ln(\delta\sqrt{\h})}{\ln(\ln\h)}\to-\infty.
\end{equation}
Namely, in order to meet Hypothesis \eqref{CondOnLambdaDelta}, \eqref{BorneKMagn} and \eqref{PutaindHypTech}, we may think $\lambda\simeq|\ln\v|^{-s},\delta\simeq|\ln\v|^{-t}$ with $s>4$ and $t>1/2$.
 
We need also assume that 
\begin{equation}\label{NonDegHyp}\text{the minimal points of  $\xi_0$, $\Lambda=\{p_1,...,p_{N_0}\}$, are non degenerate critical points}
\end{equation}
in the sense that for $p\in\Lambda$, letting ${\rm Hess}_{\xi_0}(p)$ be the Hessian matrix of $\xi_0$ at $p$, the quadratic form $Q_p(z)=z\cdot {\rm Hess}_{\xi_0}(p)z$ is a definite positive quadratic form. Note that if \eqref{NonDegHyp} holds then we may take $M=2$ in Lemma \ref{Lem.DescriptionLambda}.\\

The strategy of this work is based on a perturbative argument. This argument applies for families of {\it quasi-minimizers} of the energy with some regularity assumptions [see Theorem \ref{THM}]. In particular, we cannot have a sharp profil near a zero of a quasi-minimizer since such profil does not make any sense for quasi-minimizer. Therefore we cannot speak about an {\it ad-hoc} notion of {\it vortices} s.t. "isolated zeros". However with a natural $L^\infty$-bound on the gradient of quasi-minimizers, the notion of vorticity defects is sufficiently robust to give them a nice description.\\

For simplicity of the presentation we first state the main results for a family $\{(u_\v,A_\v)\,|\,0<\v<1\}\subset\H$ s.t.
\begin{equation}\label{MinimalityAssumption(u,A)}
\text{$(u_\v,A_\v)$ minimizes $\E_{\v,\h}$ in $\H$.}
\end{equation}
\begin{thm}\label{THM-A}
Assume that \eqref{NonDegHyp} holds and $\lambda,\delta,\h,K$ satisfy \eqref{CondOnLambdaDelta}, \eqref{BorneKMagn} and \eqref{PutaindHypTech}. There exists $\mathcal{D}_{K,b}>1$  s.t. for  $\{(u_\v,A_\v)\,|\,0<\v<1\}\subset\H$ satisfying \eqref{MinimalityAssumption(u,A)}, for sufficiently small $\v$, there exits $d_\v\in\N$ s.t. if $d_\v=0$ then $|u_\v|\geq b/2$ in $\O$, and if $d_\v\in\N^*$ then there exists a set of $d_\v$ points, $\Zz_\v=\{z^\v_1,...,z^\v_{d_\v}\}\subset\O$, s.t. for $\mu>0$ sufficiently small and independent of $\v$ we have:
\begin{enumerate}
\item $d_\v\leq\mathcal{D}_{K,b}$
\item $\{|u_\v|< b/2\}\subset\cup B(z^\v_i,\v^{\mu})\subset\O$,
\item $|z^\v_i-z^\v_j|\geq \h^{-1}\ln\h$ for $i\neq j$,
\item  $\dist( z^\v_i,\Lambda)\leq\h^{-1/2}{\ln\h}{}$ for all $i$,
\item $\deg_{\p B(z^\v_i,\v^{\mu})}(u_\v)=1$ for all $i$.
\end{enumerate}
Moreover:
\begin{enumerate}
\item There is $\eta_{\o,b}>0$ depending only on $\o$ and $b$ s.t. for all $i$ we have $B(z^\v_i,\eta_{\o,b}\lambda\delta)\subset\o_\v$.
\item If for a sequence $\v=\v_n\downarrow0$ we have $\h=\mathcal{O}(1)$ then $d_{\v}=0$ for small $\v$.
\end{enumerate}
\end{thm}
From Theorem \ref{THM-A} we know that, for small $\v$, if $\{|u_\v|< b/2\}\neq\emptyset$, then  the vorticity defects are contained in small  disks which are well separated, trapped by the impurities and located near $\Lambda$. The second theorem gives sharper informations related with the location of these disks. We divide the second theorem in three parts: 
\begin{itemize}
\item Macroscopic location: We know that the disks are near $\Lambda$, for some $p\in\Lambda$, how many disks are near $p$ ? 
\item Mesoscopic location: For $p\in\Lambda$, how the disks near $p$ are they organized ? What is their inter-distance ?
\item Microscopic location: We know that the disks are trapped by the inclusion $\o_\v$, what is their location inside $\o_\v$.
\end{itemize}
These questions are related with the crucial notion of {\it renormalized energy} [see Section \ref{Sec.RenEn}]. 
\begin{thm}\label{THM-B}
{\bf [Direct part]}\\
Assume that \eqref{NonDegHyp} holds and $\lambda,\delta,\h,K$ satisfy \eqref{CondOnLambdaDelta}, \eqref{BorneKMagn} and \eqref{PutaindHypTech}. Assume also $\h\to\infty$.

Let $\{(u_\v,A_\v)\,|\,0<\v<1\}\subset\H$ satisfying \eqref{MinimalityAssumption(u,A)} and let $\v=\v_n\downarrow0$ be a sequence. Since $d=d_\v\leq\mathcal{D}_{K,b}$, up to pass to a subsequence, we may assume that $d$ is independent of $\v$. Assume $d>0$.

\noindent{\bf Macroscopic location.} Recall that $\Lambda=\{p_1,...,p_{N_0}\}$ and for $k\in\{1,...,N_0\}$ we let $D_k:=\deg_{\p B(p_k,2\ln(\h)/\sqrt{\h})}(u_\v)$. Write ${\bf D}=(D_1,...,D_{N_0})$. Up to pass to a subsequence we may assume that ${\bf D}$ is independent of $\v$. We then have:
\begin{itemize}
\item The distribution of the disks $B(z^\v_i,\v^{\mu})$ around the elements of $\Lambda$ is the most homogenous possible : 
\[
{\bf D}\in \LamN:=\left.\left\{{\bf D}'\in\left\{\left\lceil\dfrac{d}{{N_0}}\right\rceil;\left\lfloor\dfrac{d}{{N_0}}\right\rfloor\right\}^{N_0}\,\right|\,\sum_{k=1}^{N_0} D_k'=d\right\}.
\]
Here, for $x\in\R$,  we wrote $\lceil x\rceil$ for the ceiling of $x$ and $\lfloor x\rfloor$ for the floor of $x$.
\item  There exists a renormalized energy $\W_d:\LamN\to\R$ [see \eqref{DefWdOpD}] s.t. ${\bf D}$ minimizes $\W_d$. 

\end{itemize}
\noindent{\bf Mesoscopic location.} The mesoscopic location is the same than in the homogenous case. Namely, for $p\in\Lambda$ s.t. $\deg_{\p B(p,2\ln(\h)/\sqrt{\h})}(u_\v)=D>0$, there exists a renormalized energy [see Section \ref{SecMesoRenEn}] 
\[
W^{\rm meso}_{p,D}:\{(a_1,...,a_D)\in (\R^2)^D\,|\,a_i\neq a_j\text{ for }i\neq j\}\to\R
\]
s.t., denoting $\ell:=\sqrt{\dfrac{D}{\h}}$ and for $z_i^\v\in B(p,2\ln(\h)/\sqrt{\h})$ letting $\breve z_i^\v:=\dfrac{z_i^\v-p}{\ell}$, we have $\breve{\bf z}^\v=(\breve z_1^\v,...,\breve z_D^\v)$ [assuming $z_i^\v\in B(p,2\ln(\h)/\sqrt{\h})\Leftrightarrow i\in\{1,...,D\}$] which converges to a minimizer of $W^{\rm meso}_{p,D}$. In particular $\ell$ is the typical interdistance between two close $z^\v_i,z^\v_j$. 

\noindent{\bf Microscopic location.} We know that, for $i\in\{1,...,d\}$, $B(z^\v_i,\eta_{\o,b}\lambda\delta)\subset\o_\v$. Moreover for $i\neq j$ we have $|z^\v_i-z^\v_j|\geq \ln(\h)\h^{-1}\gg\lambda\delta$. Then each connected component of $\o_\v$ contains at most one disk $B(z^\v_i,\v^{\mu})$. 

There exists a renormalized energy $W^{\rm micro}:\o\to\R$ [see Section \ref{SecMicrRenEn}] s.t. for $i\in\{1,...,d\}$, letting $y^\v_i\in\delta\cdot\Z^2$ be s.t. $B(z^\v_i,\eta_{\o,b}\lambda\delta)\subset y^\v_i+\lambda\delta\o$ and $\hat z^\v_i:=\dfrac{z^\v_i-y^\v_i}{\lambda\delta}\in\o$ we have 
\begin{itemize}
\item $W^{\rm micro}(\hat z^\v_i)\to\di\min_\o W^{\rm micro}$,
\item Up to pass to a subsequence, there is $a_i\in\o$ s.t. $\hat z^\v_i\to a_i$ and $a_i$ minimizes $W^{\rm micro}$.\footnote{For example if $\o$ is a disk then $a_i$ is the center of the disk \cite{dos2015microscopic} .}
\end{itemize}

{\bf [Optimality of the renormalized energies]}\\
Consider a sequence $\v=\v_n\downarrow0$ previously fixed [in order to have ${\bf D}$  independent of $\v$] and assume $d\neq0$. We let
\begin{itemize}
\item  ${\bf D}'\in \LamN$ be  a minimizer of $\W_d$, 
\item for  $k\in\{1,...,N_0\}$ s.t. $D_k'\geq1$,  ${\bf a}_k'$ be a minimizer of $W^{\rm meso}_{p_k,D_k'}$,
\item $a_0$ be a minimizer of $W^{\rm micro}$.
\end{itemize}
Then, for $\v=\v_n$, there exist $(u_\v',A_\v')\in\H$ and $d$ distinct points of $\O$, $\{z_1',...,z_d'\}=\{{z^\v_1}',...,{z^\v_d}'\}\subset\o_{\v}$, s.t. 
\begin{itemize}
\item $\E_{\v,\h}(u_\v',A_\v')\leq\inf_\H\E_{\v,\h}+o(1)$,
\item $\{|u_\v'|\leq b/2\}\subset \cup B(z_i',\sqrt\v)\subset\cup_{p\in\Lambda}B(p,{\ln(\h)}/\sqrt{\h})$, 
\item for $k\in\{1,...,N_0\}$, $D_k'=\deg_{\p B(p_k,2\ln(\h)/\sqrt{\h})}(u_\v')$, 
\item $\deg_{\p B(z_i',\sqrt\v)}(u_\v')=1$ for all $i$,
\item  writing for $p_k\in\Lambda$ [s.t. $D'_k\geq1$]  and $z'_i\in B(p_k,\ln(\h)/\sqrt{\h})$, $\breve z_i':=(z_i-p_k)/\sqrt{D_k/\h}$ and $\breve{\bf z}_{p_k}':=\{\breve z'_i\,|\,z_i'\to p_k\}$\footnote{We used a little abuse of notation for the simplicity of the presentation.},  we have $\breve{\bf z}_{p_k}'\to{\bf a}_k'$,
\item For $i\in\{1,...,d\}$, letting $y^\v_i\in\delta\cdot\Z^2$ be s.t. $z'_i\in y^\v_i+\lambda\delta\cdot\o$ and $\hat z'_i:=\dfrac{z'_i-y^\v_i}{\lambda\delta}\in\o$ we have $\hat z_i'\to a_0$.
\end{itemize}
\end{thm}

The third  theorem underline the link between the number $d$ and  $\h$. In this theorem we write, for $x\in\R$,  $[x]^+=\max(x,0)$ and $[x]^-=\min(x,0)$.
\begin{thm}\label{THM-C}
Assume that $\O$ satisfies \eqref{NonDegHyp}, $\lambda,\delta,\h,K$ satisfy \eqref{CondOnLambdaDelta}, \eqref{BorneKMagn} and \eqref{PutaindHypTech}.

There are integers $L\in\{1,...,N_0\}$, $0=d_0^\star<d_1^\star<\cdots<d_L^\star=N_0$ [$d_k^\star\in\N$ is independent of $\v$] and critical fields [depending on $\v$] $\Critical_1<\cdots<\Critical_L<\Criticalbis_1<\Criticalbis_2<\cdots$ [see  \eqref{TheExprionnI} and \eqref{TheExprionnII} for the expressions of $\Critical_k$ and $\Criticalbis_k$] s.t. for  $\{(u_\v,A_\v)\,|\,0<\v<1\}\subset\H$ a family satisfying \eqref{MinimalityAssumption(u,A)} and  for a sequence $\v=\v_n\downarrow0$:
 \begin{itemize}
\item If $d_\v=0$ for  small $\v$, then $[\h-\Critical_1]^+\to0$.
\item If $d_\v>0$ for  small $\v$, then $[\h-\Critical_1]^-\to0$.
\item Assume $L\geq2$. For $k\in\{1,...,L-1\}$, if  for small $\v$ we have $d_{k-1}^\star<d_\v\leq d_k^\star$, then
\[
\left[\h-\Critical_k\right]^-\to0\text{ and }\left[\h-\Critical_{k+1}\right]^+\to0.
\]
\item  For $L\geq1$, if for small $\v$ we have $d_{L-1}^\star< d_\v\leq d_L^\star=N_0$, then
\[
\left[\h-\Critical_L\right]^-\to0\text{ and }\left[\h-\Criticalbis_1\right]^+\to0.
\]
\item Let $l\in\N^*$. If for small $\v$  we have $d_\v=N_0+l$, then
\[
\left[\h-\Criticalbis_{l}\right]^-\to0\text{ and }\left[\h-\Criticalbis_{l+1}\right]^+\to0.
\]
\end{itemize}
\begin{remark}
A more complete statement for $d_\v\in\{0,...,N_0\}$ may be found in Proposition \ref{Prop.SHarperdescriptionNonSatured}.
\end{remark}
\end{thm}
 \section{Notation}\label{SecNotation}
\subsection{Sets, vectors and numbers}
\begin{itemize}
\item We identify the real plan $\R^2$ with $\C$ and we denote by $\S^1$ the unit circle in $\C$.
\item For $\U \subset\R^2$, $N\in\N\setminus\{0;1\}$, $(\U ^N)^*:=\{(z_1,...,z_N)\in \U ^N\,|\,z_i\neq z_j\text{ for }i\neq j \}$.
\item For $k\in\{1;2\}$, $\Haus^k$ is the $k$-dimensional Hausdorff measure.
\item If $(a_1,a_2),(b_1,b_2)\in\R^2$, then $|(a_{1},a_{2})|=\sqrt{a_{1}^2+a_{2}^2}$, $(a_1,a_2)^\bot=(-a_2,a_1)$, $(a_1,a_2)\cdot(b_1,b_2)=a_1b_1+a_2b_2$ and $(a_1,a_2)\wedge(b_1,b_2)=a_1b_2-a_2b_1$.
\item For $\U \subset\R^2$, $\overline{\U}$ is the closure of $\U$ w.r.t. $|\cdot|$
\item For $\emptyset\neq\U,\V\subset\R^2$ and $x_0\in\R^2$ we write $\dist(\U,\V):=\inf\{|x-y|\,|\,x\in\U,\,y\in\V\}$ and $\dist(x_0,\V):=\dist(\{x_0\},\V)$.
\item For $\Gamma\subset\R^2$ a Jordan curve we let: 
\begin{itemize}
\item ${\rm int}(\Gamma)$, the interior of $\Gamma$, be  the bounded open set $\U \subset\R^2$ s.t. $\Gamma=\p \U $ where $\p \U $ is the boundary of $\U $.
\item $\nu$ be the outward normal unit vector of ${\rm int}(\Gamma)$
\item $\tau$ be the direct unit tangent vector of $\Gamma$ ($\tau=\nu^\bot$)
\end{itemize}
\item If $S$ is a finite set then ${\rm Card}(S)$ is the cardinal of $S$.
\end{itemize}
\begin{itemize}
\item If $x\in\R$, then we write $\lceil x\rceil:=\min\{m\in\Z\,|\,m\geq x\}$, the ceiling of $x$, and $\lfloor x\rfloor:=\max\{m\in\Z\,|\,m\leq x\}$, the floor of $x$.
\item If $x\in\R$, then we write $[x]^+=\max(x,0)$ and $[x]^-=\min(x,0)$.
\end{itemize}
\subsection{Functions}
\begin{itemize}
\item For $\U\subset\R^2$  a smooth open set and $K\subset\C$, $H^1(\U,K)=\{u\in H^1(\U,\C)\,|\,u(x)\in K\text{ for a.e. }x\in \U\}$ where $H^1(\U,\C)$ is the Classical Sobolev space of the first order modeled on the Lebesgue space $L^2$. 

For $k\in\N^*$ and $p\in [1;\infty]$ we use the standard notation for the higher order Sobolev space $H^k(\U,K)$ modeled on $L^2$ and $W^{k,p}(\U,K)$ for the Sobolev space of order $k$ modeled on $L^p$.
\item We use the standard notation for the differential operators: ''$\n$'' for the gradient, ''$\rot$'' for the curl, ''$\Div$'' for the divergence, "$\p_\tau=\tau\cdot\n$" for the tangential derivative, "$\p_\nu=\nu\cdot\n$" for the normal derivative...
\item  For $\U\subset\R^2$  a smooth bounded open set we let $\tr_{\p\U}:H^1(\U,\C)\to H^{1/2}(\p \U,\C)$ be the [surjective] trace operator. For $\Gamma$ a connected component of $\p \U$ and $u\in H^1(\U,\C)$, we let $\tr_\Gamma(u)$ be the restriction of $\tr_{\p \U}(u)$ to $\Gamma$.

We write $H_0^1(\U,\C):=\{u\in H^1(\U,\C)\,|\,\tr_{\p\U}(u)=0\}$.
\item For $u:\O\to\C$ a function we let $\underline u:=\begin{cases}u&\text{if }|u|\leq1\\u/|u|&\text{if }|u|>1\end{cases}$.
\item For $\Gamma\subset\R^2$ a Jordan curve and $g\in H^{1/2}(\Gamma,\S^1)$, the degree of $g$ is defined as
\begin{equation}\nonumber
\deg_{\Gamma}(g):=\frac{1}{2\pi}\int_{\Gamma}g\wedge\p_\tau g\in\Z.
\end{equation}
For  a  smooth and bounded open set $\U\subset\R^2$, $\Gamma$ a connected component of $\p \U$  and $u\in H^1(\U,\C)$, if  there exists $\eta>0$ s.t. $g:=\tr_\Gamma(u)$ satisfies $|g|\geq\eta$ , then $g/|g|\in H^{1/2}(\Gamma,\S^1)$ and we write $\deg_\Gamma(u):=\deg_\Gamma(g/|g|)$.
 
When $\U,\V\subset\R^2$ are smooth bounded simply connected open sets s.t. $\overline{\V}\subset \U$ and $u\in H^1(\U\setminus\overline{\V},\S^1)$, then we write [without ambiguity] $\deg(u)$ instead of $\deg_\Gamma(u)$ for any Jordan curve $\Gamma\subset \overline\U\setminus{\V}$ s.t. $\V\subset{\rm int}(\Gamma)$.

\end{itemize}
\subsection{Construction of the pinning term}\label{SecConstructionPinningTerm}
Let
\begin{enumerate}[$\bullet$]\item $\delta=\delta(\v)\in(0,1),\,\lambda=\lambda(\v)\in(0,1)$;
 \item $\o\subset\R^2$ be a smooth bounded and simply connected open set s.t. $(0,0)\in\o$ and $\overline{\o}\subset Y:=(-1/2,1/2)^2$.
\end{enumerate} For $m\in\Z^2$ we denote $Y_{m}^\delta:=\delta m+\delta\cdot Y$ and 
$\displaystyle\o_\v=\bigcup_{\substack{m\in\Z^2\text{ s.t.}\\Y_{m}^\delta\subset\O}}[\delta m+\lambda\delta\cdot\o]$.
For $b\in(0,1)$ we define 
\[
\begin{array}{cccc}
a_\v:&\R^2&\to&\{b,1\}\\&x&\mapsto&\begin{cases}b&\text{if }x\in\o_\v\\1&\text{otherwise}\end{cases}
\end{array}.
\]
\subsection{Asymptotic}
 \begin{itemize}
\item[$\bullet$] In this article $\v\in(0;1)$ is a  small number. We are essentially interested in the asymptotic  $\v\to0$.
\item[$\bullet$] The notation $o(1)$ means a quantity depending on $\v$ which tends to $0$ when $\v\to0$.
\item[$\bullet$] The notation $o[f(\v)]$ means a quantity $g(\v)$ s.t. $\dfrac{g(\v)}{f(\v)}=o(1)$.
\item[$\bullet$] The notation $\mathcal{O}[f(\v)]$ means a quantity $g(\v)$ s.t. $\dfrac{g(\v)}{f(\v)}$ is bounded  for small $\v$.
\end{itemize}

\section{Classical facts and the strongest theorem}
{\bf Gauge invariance and Coulomb Gauge}

It is standard to quote the {\it gauge invariance} of the energy $\E_{\v,\h}$. Namely, two configurations $(u,A),(u',A')\in\H$ are gauge equivalent, denoted by $(u,A)\eqj(u',A')$, if there exists a gauge transformation from $(u,A)$ to $(u',A')$:
\[
(u,A)\eqj(u',A')\Longleftrightarrow\begin{cases} \exists\,\varphi\in H^2(\O,\R)\text{ s.t.}\\u'=u\e^{\imath\varphi}\text{ and }A'=A+\n\varphi\end{cases}.
\]

Two gauge equivalent configurations describe the same physical state. Then, physical quantities are those which are gauge invariant. For example, if $(u,A)\in\H$, then $|u|$, $|\n u-\imath Au|$, $\rot(A)$ and then $ \E_{\v,\h}(u,A)$, $\{|u|<b/2\}$ also are gauge invariants.

In the context the Ginzburg-Landau energy, a classical choice of gauge is the {\it Coulomb gauge}. We say that  $(u,A)$ is in the  Coulomb gauge if
\begin{equation}\label{JaugeCoulomb}
\begin{cases}\Div(A)=0&\text{in }\O\\A\cdot\nu=0&\text{on }\p\O \end{cases}.
\end{equation}
One may prove [see Proposition 3.2 in \cite{SS1}] that, for $(u,A)\in\H$, there exists  $\varphi\in H^2(\O,\R)$ s.t. $A':=A+\n\varphi$ satisfies \eqref{JaugeCoulomb}. Then, letting $u'=u\e^{\imath\varphi}$, we have $(u',A')$ which is in the Coulomb gauge and $(u,A)\eqj(u',A')$.

One of the main motivations in using the Coulomb gauge comes from the fact that  $\|\rot(A)\|_{L^2}$ controls $\|A\|_{H^1}$. Namely there exists $C\geq 1$ [which depends only on $\O$] s.t. if $A$ satisfies \eqref{JaugeCoulomb}  then [see Proposition 3.3 in \cite{SS1}]
\begin{equation}\label{CoulombH1}
\|A\|_{H^1(\O,\R^2)}\leq C\|\rot(A)\|_{L^2(\O)}
\end{equation}
and
\begin{equation}\label{CoulombH2}
\|A\|_{H^2(\O,\R^2)}\leq C\|\rot(A)\|_{H^1(\O)}.
\end{equation}
Moreover we have an easy representation of  $A\in H^1(\O,\R^2)$ satisfying \eqref{JaugeCoulomb}
\begin{equation}\label{RepresentCoulomGauge}
\text{$A\in H^1(\O,\R^2)$ is a solution of \eqref{JaugeCoulomb}}\Longleftrightarrow\,\exists\, \xi\in H^1_0\cap H^2(\O,\R)\text{ s.t. }A=\n^\bot\xi.
\end{equation}

{\bf Basic description of a minimizer}

We first note that, by direct minimization, for  all $a_\v\in L^\infty(\O,[b;1])$, $\v,\h>0$, the minimization problem of  $ \E_{\v,\h}$ in $\H$ admits [at least] a solution $(u_\v,A_\v)\in\H$. 

Writing $h_\v:=\rot(A_\v)$, it is standard to check that a such minimizer solves:
\begin{equation}\label{FullGLuAEq}
\begin{cases}
-(\n-\imath A_\v)^2u_\v=\dfrac{u_\v}{\v^2}(a_\v^2-|u_\v|^2)^2&\text{in }\O
\\(\n-\imath A)u_\v\cdot\nu=0&\text{on }\O
\\
-\n^\bot h_\v=u_\v\wedge(\n-\imath A_\v)u_\v&\text{in }\O
\\h_\v=\h&\text{on }\p\O
\end{cases}.
\end{equation}
Using a maximum principle, we may get the following proposition:
\begin{prop}\label{Prop.ModuLeq1}
Let $\v,\h>0$ and $a\in L^\infty(\O,[b,1])$. If $(u_\v,A_\v)$ is a minimizer of  $ \E(u,A)=\di\dfrac{1}{2}\int_\O|\n u-\imath Au|^2+\dfrac{1}{2\v^2}(a^2-|u|^2)^2+|\rot(A)-\h|^2$ in $\H$ then $|u_\v|\leq1$ in $\O$.
\end{prop}
On the other hand, if $(u_\v,A_\v)$ is a minimizer of $\E_{\v,\h}$ in the Coulomb gauge, then it solves
\begin{equation}\label{FullGLuEq}
\begin{cases}
-\Delta u_\v=\dfrac{u_\v}{\v^2}(a_\v^2-|u_\v|^2)^2-2\imath(A_\v u_\v\cdot\n u_\v)-|A_\v|^2u_\v&\text{in }\O
\\\p_\nu u_\v=0&\text{on }\O
\end{cases}.
\end{equation}
A fundamental bound in the study concerns $\|\n u_\v\|_{L^\infty(\O)}$. We have the following lemma which is a Gagliardo-Nirenberg type inequality with homogenous Neumann boundary condition. 
\begin{lem}\label{LemGNEst}\footnote{The proof of Lemma \ref{LemGNEst} is done by first using $\Phi:\D\to\O$, a conformal representation of $\O$ on the unit disk $\D$. Then we extend $\tilde u:=u\circ\Phi$ in the disk $B(0,2)$ by letting $u'(x)=\tilde u(x/|x|)$ for $x\in B(0,2)\setminus\D$. By using the boundary condition we have $u'\in H^2(B(0,2),\C)$. And finally one may conclude by using an interior version of Lemma \ref{LemGNEst} [Lemma A.1 in \cite{BBH1}].}\label{NumFootNoteConformal}
Let $\O\subset\R^2$ be a smooth bounded simply connected open set. There exists $C_\O\geq1$ s.t. if $u\in H^2(\O)$ is s.t. $\p_\nu u=0$ on $\p\O$ then
\[
\|\n u\|_{L^\infty(\O)}^2\leq C_\O\left(\|\Delta u\|_{L^\infty(\O)}+\|u\|_{L^\infty(\O)}\right)\|u\|_{L^\infty(\O)}.
\]
\end{lem}

Consequently, with Lemma \ref{LemGNEst} [up to change the value of  $C_\O$], for $\v,\h>0$ and $a_\v\in L^\infty(\O,[b^2,1])$, if $(u_{\v},A_{\v})\in\H$ minimizes $\E_{\v,\h}$ is in the Coulomb gauge and is s.t. $\|A_\v\|_{L^\infty(\O)}\leq 1/\v$ [which is the case in the present work] then
\begin{equation}\label{CrucialLipEst}
\|\n u_\v\|_{L^\infty(\O)}\leq \dfrac{C_\O}{\v}.
\end{equation}
In the homogenous case as well as in the case without magnetic field, Estimate \eqref{CrucialLipEst} is crucial to describe vorticity defects. It is the same in the present work. More precisely, the main result [Theorem \ref{THM}] states that  the three above theorems are true  replacing $(u_\v,A_\v)$ that minimizes $\E_{\v,\h}$ in $\H$ by any configuration $(\tilde u_\v,\tilde A_\v)$ s.t. $E_\v(\tilde u_\v,\tilde A_\v)=\inf_\H\E_{\v,\h}+o(1)$ with two extra hypotheses on $|\tilde u_\v|$ :  $\|\n  |\tilde u_\v|\|_{L^\infty(\O)}=\mathcal{O}(\v^{-1})$ and $|\tilde u_\v|\in W^{2,1}(\O)$ [see \eqref{HypGlobalSurQuasiMin}]\\

{\bf Lassoued-Mironescu decoupling}

In order to study  pinned Ginzburg-Landau type energies, a nice trick was initiated by Lassoued and Mironescu in \cite{LM1}. Before explaining this trick we have to do a direct calculation for $(u,A)\in \H$:
\begin{equation}\label{DevCarre}
 \E_{\v,\h}(u,A)
= E_\v(u)+\dfrac{1}{2}\int_\O-2(u\wedge\n u)\cdot A+|u|^2|A|^2+|\rot(A)-\h|^2
\end{equation}
with 
\[
 E_\v(u)=\dfrac{1}{2}\int_\O|\n u|^2+\dfrac{1}{2\v^2}(a_\v^2-|u|^2)^2.
\]

The Lassoued-Mironescu decoupling is obtained by first minimizing  $ E_\v$ in $H^1(\O,\C)$. It is clear that $E_\v$ admits minimizers and if $U$ minimizes $E_\v$ then it satisfies
\begin{equation}\label{EqForU}
\begin{cases}
-\Delta U=\dfrac{U}{\v^2}(a_\v^2-|U|^2)&\text{in }\O\\\p_\nu U=0&\text{on }\p\O
\end{cases}.
\end{equation}

By an energetic argument it is easy to prove that, if $U$ minimizes $ E_\v$ in $H^1(\O,\C)$, then $b\leq|U|\leq1$. Moreover from \eqref{EqForU}, $U\wedge\n U=0$, {\it i.e.} $U=|U|\e^{\imath\theta}$ with $\theta\in\R$.

Then one may consider a scalar minimizer $U_\v:\O\to[b,1]$. This scalar minimizer may be seen as a regularization of  $a_\v$ [see Proposition \ref{Prop.RegularizationLMSol}].

Using this scalar minimizer one may get the well known Lassoued-Mironescu decoupling: for $v\in H^1(\O,\R)$ we have
\begin{equation}\label{DecouplageLM}
 E_\v(U_\v v)= E_\v(U_\v)+F_\v(v)
\end{equation}
with
\[
F_\v(v):=\dfrac{1}{2}\int_\O U_\v^2|\n v|^2+\dfrac{U_\v^4}{2\v^2}(1-|v|^2)^2.
\]
Using this decoupling, one may prove that, for $\v>0$, there exists a unique positive minimizer $U_\v:\O\to[b,1]$ of $E_\v$ in $H^1(\O,\R)$.

On the other hand, from \eqref{DevCarre} and \eqref{DecouplageLM}, for $(u,A)\in \H$ and $v=u/U_\v$ we have:
\begin{eqnarray*}
\F_{\v,\h}(v,A)&:=&\E_{\v,\h}(U_\v v,A)- E_\v(U_\v)
\\&=&\dfrac{1}{2}\int_\O U_\v^2|\n v-\imath A v|^2+\dfrac{U_\v^4}{2\v^2}(1-|v|^2)^2+|\rot(A)-\h|^2.
\end{eqnarray*}

It is easy to check that  $\F_{\v,\h}(v,A)$ is gauge invariant. This functional is of major interest in the study since $(v,A)$ minimizes $\F_{\v,\h}$ in $\H$ if and only if $(U_\v v,A)$ minimizes $\E_{\v,\h}$ in $\H$.

An easy comparaison argument implies that if $(v_\v,A_\v)$ minimizes $\F_{\v,\h}$ then $\|v_\v\|_{L^\infty(\O)}\leq1$.

 From now on we focus on the study of the minimizer of $\F_{\v,\h}$. Namely we have the following theorem.

\begin{thm}\label{THM} 
Assume that \eqref{NonDegHyp} holds and $\lambda,\delta,\h,K$ satisfy \eqref{CondOnLambdaDelta}, \eqref{BorneKMagn} and \eqref{PutaindHypTech}.

Let $\{(v_\v,A_\v)\,|\,0<\v<1\}\subset\H$ be s.t.  $\F(v_\v,A_\v)\leq\inf_\H\F+o(1)$. Assume also that 
\begin{equation}\label{HypGlobalSurQuasiMin}
\begin{cases}|v_\v|\in W^{2,1}(\O,\C)
\\\|\n| v_\v|\|_{L^\infty(\O)}=\mathcal{O}(\v^{-1})
\end{cases}.
\end{equation}

Then Theorems \ref{THM-A}, \ref{THM-B} and \ref{THM-C} hold for $u_\v=U_\v v_\v$.
\end{thm}
\begin{remark}\label{THMRmark} 
Theorem \ref{THM} may be rephrased in term of $U_\v$.  Let $(\h)_{0<\v<1}\subset(0,\infty)$, $\{(u_\v,A_\v)\,|\,0<\v<1\}\subset\H$ and let $v_\v:=u_\v/U_\v\in H^1(\O,\C)$. On the one hand,  from the decoupling \eqref{DecouplageLM}, we have $\{(u_\v,A_\v)\,|\,0<\v<1\}\subset\H$ is s.t. $\E_{\v,\h}(u_\v,A_\v)\leq\inf_\H\E_{\v,\h}+o(1)$ if and only $\{(v_\v,A_\v)\,|\,0<\v<1\}$ is s.t. $\F_{\v,\h}(v_\v,A_\v)\leq\inf_\H\F_{\v,\h}+o(1)$. On the other hand $v_\v$ satisfies \eqref{HypGlobalSurQuasiMin} if and only if we have $|u_\v|\in W^{2,1}(\O,\C)$ and $\|\n| u_\v|\|_{L^\infty(\O)}=\mathcal{O}(\v^{-1})$.

\end{remark}

\section{Plan of the article and proof of Theorem \ref{THM}}
The proof of Theorem \ref{THM} is done in several steps. It is based on a perturbative argument by replacing the energy $\F_{\v,\h}$ with an energy $\tilde\F_{\v,\h}$. This step is called the energetic cleaning [Section \ref{Sec.Clean}]. The functional $\tilde\F_{\v,\h}$ is a perturbation of $\F_{\v,\h}$: for $(v_\v,A_\v)\in\H$ which is in the Coulomb gauge and s.t. $\F_{\v,\h}(v_\v,A_\v)=\mathcal{O}(\h^2)$ we have $\tilde\F_{\v,\h}(v_\v,A_\v)-\F_{\v,\h}(v_\v,A_\v)=o(1)$ [see Proposition \ref{Prop.Nettoyage}].  In particular we have $\F_{\v,\h}(v_\v,A_\v)\leq\inf_\H\F_{\v,\h}+o(1)$ if and only if $\tilde\F_{\v,\h}(v_\v,A_\v)\leq\inf_\H\tilde\F_{\v,\h}+o(1)$. \\

In section \ref{SectionBoundVorticity} we apply a vortex ball construction of Sandier-Serfaty [Proposition \ref{Prop.BorneInfLocaliseeSandSerf}] and  we follow the strategy of Sandier-Serfaty developed in \cite{SS2} to prove that the vorticity of a reasonable configuration is bounded [see Theorem \ref{ThmBorneDegréMinGlob}]. 

Once the bound on the vorticity yields, we adapt a result of  Serfaty \cite{S1} which gives a decomposition of  $\tilde{\F}_{\v,\h}(v_\v,A_\v)$ in term of $F_\v(v_\v)$ and the location of the vorticity defects [Proposition \ref{Docmpen}]. \\

The decomposition obtained in Proposition \ref{Docmpen} allows to focus the study on the energy $F_\v$ which ignores the magnetic field. From this point, the study of a configuration $(v_\v,A_\v)$ is done for a major part {\it via} classical results based on the case without magnetic field [as in \cite{BBH}]. To this end we adapt to our case some standard estimates ignoring the magnetic field, in particular the crucial notion of Renormalized energies is presented Section \ref{Sec.RenEn}.\\

With these preliminary results, in Section \ref{SecUpperBound}, for $d\in\N^*$, we  construct competitors $(v_\v,A_\v)\in\H$ with $d$  quantized vorticity defects and then we get a sharp upper bound [see Proposition \ref{Prop.BorneSupSimple}]:
\[
\inf_\H\F_{\v,\h}\leq \h^2 \Jo+d\Pic\left[-\h+\HoC\right]+\L_1(d)\ln\h+\L_2(d)+o(1).
\]
Here $\Jo\&\Pic$ are independent of $\v$ and $d$, $\L_1(d)\&\L_2(d)$ are independent of $\v$ and $\HoC$ is the leading term in the expression of the first critical field.\\

With the above upper bound for the minimal energy, the heart of the work consists in getting lower bounds for quasi-minimizers. Before getting such lowers bounds we adapt to our case some tools in Section \ref{Sec.ToolBox}: an $\eta$-ellipticity result is proved [Proposition \ref{Prop.EtaEllpProp}], a construction of {\it ad-hoc} bad-discs is done [Proposition \ref{Prop.ConstrEpsMauvDisk}] and the strong effect of the dilution is expressed by various result in Section \ref{Sec.StrongEffectDilution}. \\

In Section \ref{Sect.ShapInfo} we begin the proof of the theorems.  The part of Theorem \ref{THM} related with Theorem \ref{THM-A} is a direct consequence of Propositions  \ref{PropToutLesDegEg1}, \ref{PropVortexProcheLambda}, \ref{Prop.BonEcartement} and \ref{Prop.PinningComplet} [and also Corollary \ref{CorDefPremierChampsCrit}].

The part of Theorem \ref{THM} related with Theorem \ref{THM-B} is given by Corollary \ref{Cor.ExactEnergyExp}  and Proposition \ref{Prop.BorneSupSimple}.

The part of Theorem \ref{THM} related with Theorem \ref{THM-C} is a direct consequence of Corollary \ref{CorDefPremierChampsCrit} and Propositions \ref{Prop.SHarperdescriptionNonSatured}$\&$\ref{Prop.SHarperdescriptionNonSaturedII}.

\section{Some preliminaries}

\subsection{Energetic cleaning}\label{Sec.Clean}

In order to do the cleaning step, we have to get some estimates. Our goal is to study {\it quasi-minimizer} of $\F_{\v,\h}$. To keep a simple presentation, we write $\F$ instead of $\F_{\v,\h}$ and $F$ instead of $F_\v$ when there is no ambiguity.\\

From \eqref{CoulombH1}, \eqref{CoulombH2} and classical elliptic regularity arguments we have the following proposition.
\begin{prop}\label{Prop.BornesSups1}Let $\{(v_\v,A_\v)\,|\,0<\v<1\}\subset\H$ be  a family of configuration in the Coulomb gauge. Then there is $\xi_\v\in H^1_0\cap H^2(\O,\R)$ s.t. $A_\v=\n^\bot\xi_\v$. Moreover, if for some $\h=\h(\v)$ we have 
\begin{equation}\label{BorneFh^2}
\text{$\F(v_\v,A_\v)=\mathcal{O}(\h^2)$,}
\end{equation}
then there exists ${C}$ [independent of $\v$] s.t.
\begin{eqnarray}\label{FirstUpBoundXi}
\|\xi_\v\|_{H^2(\O)}&\leq& {C}\h.
\end{eqnarray}
Consequently, for $p\in[1,\infty)$, there exists $C_p>1$ [independent of $\v$] s.t.
\begin{equation}\label{EstLpA}
\|\n \xi_\v\|_{L^p(\O)}=\|A_\v\|_{L^p(\O)}\leq C_p\h.
\end{equation}
Moreover, up to increase the value of $C>1$ [independently of $\v$], we have
\begin{equation}\label{EstGradMinL2}
\|\n v_\v\|_{L^2(\O)}\leq C\h.
\end{equation}
And if $\rot(A_\v)\in H^1(\O)$ then
\begin{equation}\label{EstH3}
\|\xi_\v\|_{H^3(\O)}\leq C\|\rot(A_\v)\|_{H^1(\O)}.
\end{equation}
In particular, for further use,  note that if $\rot(A_\v)\in H^1(\O)$ then $\xi_\v\in H^1_0\cap H^2\cap W^{1,\infty}(\O)$ and 
\begin{equation}\label{EstH4}
\|\n\xi_\v\|_{L^{\infty}(\O)}\leq C\|\rot(A_\v)\|_{H^1(\O)}.
\end{equation}
\end{prop}
In order to do the cleaning step we need to underline the fact that $U_\v$ may be seen as a regularization of $a_\v$ in $W^{1,\infty}$ with estimates that become bad when approaching $\p\o_\v$.
\begin{prop}\label{Prop.RegularizationLMSol}
There exist $C_b,\ab>0$ depending only on $b$ and $\O$ s.t. for $\v,r>0$ we have:
\begin{equation}\label{EstGlobGradU}
\|\n U_\v\|_{L^\infty(\O)}\leq\dfrac{C_b}{\v},
\end{equation}
\begin{equation}\label{EstLoinInterfaceU}
|U_\v-a_\v|\leq{C_b}\e^{-\frac{\ab r}{\v}}\text{ in }\{x\in\O\,|\,\dist(x,\p\o_\v)\geq r\},
\end{equation}
\begin{equation}\label{EstLoinInterfaceGradU}
|\n U_\v|\leq\dfrac{C_b\e^{-\frac{\ab r}{\v}}}{\v}\text{ in }\{x\in\O\,|\,\dist(x,\p\o_\v)\geq r\}.
\end{equation}
\end{prop}
\begin{proof}
Estimate \eqref{EstGlobGradU} is a consequence of Lemma \ref{LemGNEst}. The proof of \eqref{EstLoinInterfaceU} is the same than Proposition 2 in \cite{Publi3}. Estimate \eqref{EstLoinInterfaceGradU} is proved in Appendix \ref{ProofVotation}. 
\end{proof}
Since the 2-dimensional Hausdorff measure of $\o_\v$ satisfies $\Haus^2(\o_\v)= \mathcal{O}(\lambda^2)$, from \eqref{EstLoinInterfaceU}, for $p\in[1,\infty[$, we have the following crucial estimate
\begin{equation}\label{LpEstmU}
\|U_\v^2-1\|_{L^p(\O)}=\mathcal{O}(\lambda^{2/p}).
\end{equation}

We are now in position to do the cleaning step. We assume that $\{(v_\v,A_\v)\,|\,0<\v<1\}\subset\H$ is  a family of configuration in the Coulomb gauge which satisfies \eqref{BorneFh^2}. 
We denote $\alpha_\v=U_\v^2$ and $\rho_\v=|v_\v|$. From direct computations, by splitting the integrals with the identity $\alpha_\v=(\alpha_\v-1)+1$ and using $(1-\rho_\v)^4\leq(1-\rho_\v^2)^2$, we have the existence of $C\geq1$ [independent of $\v$] s.t.
\begin{equation}\label{Lem.Nettoyage1}
\left|\int_\O\alpha_\v(v_\v\wedge\n v_\v)\cdot A_\v-\int_\O(v_\v\wedge\n v_\v)\cdot A_\v\right|\leq\dfrac{C}{2}\left[\sqrt\lambda\h^2+\lambda^{1/4}\h^3\v\right]\leq  C\sqrt\lambda\h^2
\end{equation}
and
\begin{equation}\label{Lem.Nettoyage2}
\left|\int_\O \alpha_\v \rho_\v^2|A_\v|^2-\int_\O  |A_\v|^2\right|\leq C\h^2(\v\h+\lambda).
\end{equation}

By combining \eqref{Lem.Nettoyage1} and \eqref{Lem.Nettoyage2} we immediately get the following proposition.
\begin{prop}\label{Prop.Nettoyage}
If $(v_\v,A_\v)$ is in the Coulomb gauge and satisfies \eqref{BorneFh^2} then
\[
|\tilde\F(v_\v,A_\v)-\F(v_\v,A_\v)|\leq C\h^2(\v\h+\sqrt\lambda)
\] 
with $C$ which is independent of $\v$ and 
\begin{equation}\label{EGALITEdenettoyga}
\tilde\F( v, A)=\tilde\F_{\v,\h}( v, A):=F( v)+\dfrac{1}{2}\int_\O-2( v\wedge\n  v)\cdot  A+| A|^2+|\rot( A)-\h|^2.
\end{equation}

\end{prop}
\begin{remark}
\begin{enumerate}
\item One may claim that  $\tilde\F$ is not gauge invariant if $\alpha_\v\not\equiv1$.
\item Note that if $\lambda^{1/4}|\ln\v|\to0$ and if $\h=\mathcal{O}(|\ln\v|)$ then for $(v_\v,A_\v)\in\H$ which  is in the Coulomb gauge and satisfies \eqref{BorneFh^2} we have $\tilde\F(v_\v,A_\v)-\F(v_\v,A_\v)=o(1)$ without hypothesis on $\delta\in(0;1)$.
\end{enumerate}
\end{remark}

\subsection{Bound on the vorticity and energetic decomposition}\label{SectionBoundVorticity}
By applying Proposition 1 in \cite{SS2} with $U_\v\geq b$ we immediately get the following proposition which does not need any assumption for $\lambda,\delta\in(0;1)$.
\begin{prop}\label{Prop.BorneInfLocaliseeSandSerf}
Assume $\h\leq \borneh |\ln\v|$ with $\borneh \geq1$ which is independent of $\v$. Let $\{(v_\v,A_\v)\,|\,0<\v<1\}$ be a family s.t. $\F(v_\v,A_\v)\leq \borneh |\ln\v|^2$. 

Then there exist $C,\v_0>0$ {[depending only on $\O$, $b$ and $\borneh$]} s.t. for $\v<\v_0$ we have either $|v_\v|\geq1-|\ln\v|^{-2}$ in $\O$ or there exists  a finite family of disjoint disks $\{B_i\,|\,i\in \J\}$ with $\J\subset\N^*$ [$\J$ depends on $\v$] and $B_i:=B(a_i,r_i)$ satisfying :
\begin{enumerate}
\item $\{|v_\v|<1-|\ln\v|^{-2}\}\subset\cup B_i$
\item $\sum r_i<|\ln\v|^{-10}$,
\item writing $h_\v=\rot(A_\v)$, $\rho_\v=|v_\v|$ and $v_\v=\rho_\v\e^{\imath\varphi_\v}$ [$\varphi_\v$ is locally defined] we have
\begin{equation}\label{EstimateSS3ball}
\dfrac{1}{2}\int_{B_i}\rho^2|\n \varphi_\v-A_\v|^2+| h_\v-\h|^2\geq\pi|d_i|(|\ln\v|-C\ln|\ln\v|),
\end{equation}
with $d_i=\deg_{\p B_i}(v)$ if $B_i\subset\O$ and $0$ otherwise.
\end{enumerate}
\end{prop}

By following the argument of Sandier and Serfaty \cite{SS2}, we get the main result of this section.

\begin{thm}\label{ThmBorneDegréMinGlob}
Assume that $\lambda,\delta$ satisfy \eqref{CondOnLambdaDelta} and $\delta^2|\ln\v|\leq1$. Assume also Hypothesis \eqref{BorneKMagn} holds for $\h$ with some $K\geq1$.   

Then there exist $\v_K>0$ and $\M_K\geq1$ [independent of $\v$] s.t. if  $\{(v_\v,A_\v)\,|\,0<\v<1\}\subset\H$ is a family in the Coulomb gauge satisfying $\F(v_\v,A_\v)\leq\inf_{\H}\F+K\ln|\ln\v|$ then for $0<\v<\v_K$ we have
\begin{equation}\label{CrucialBoundedkjqbsdfbn}
\dfrac{1}{2}\int_\O|\n v_\v|^2+\dfrac{1}{2\v^2}(1-|v_\v|^2)^2\leq\M_K|\ln\v|.
\end{equation}
Moreover, if $|v_\v|\not>1-|\ln\v|^{-2}$ in $\O$, then letting $\{B_i\,|\,i\in\J\}$ be a family of disks given by Proposition \ref{Prop.BorneInfLocaliseeSandSerf}, for $0<\v<\v_K$, we have $d_i\geq0$ for all $i\in\J$ and there is $s_0>0$ [depending only on $\O$] s.t. if $i\in\J$ is s.t. $d_i\neq0$ then $\dist(B_i,\Lambda)\leq \M_K|\ln\v|^{-s_0}$.

\end{thm}
The proof of this theorem is postponed in Appendix \ref{SectionProofAppSSBound}.

 We let 
\begin{equation}\label{DefJ0}
\Jo:=\tilde{\F}_{1,1}(1,\n^\bot\xi_0)=\dfrac{\tilde{\F}_{\v,\h}(1,\h\n^\bot\xi_0)}{\h^2}.
\end{equation}
Note that if $\{(v_\v,A_\v)\,|\,0<\v<1\}$ is a family of quasi-minimizers then 
\[
\F_{\v,\h}(v_\v,A_\v)\leq\F_{\v,\h}(1,\n^\bot\xi_0)+o(1)=\h^2\Jo+o(1)=\mathcal{O}(\h^2).
\]
The discs given by Proposition \ref{Prop.BorneInfLocaliseeSandSerf} are "too large" for our strategy. Indeed one of the main argument is a construction of {\it bad discs} in the spirit of \cite{BBH} which  links $x_\v\in\{|v_\v|\leq1/2\}$ with the energetic cost in a ball $B(x_\v,\v^\mu)$  with small $\mu>0$. Namely if $x_\v\in\{|v_\v|<1-|\ln\v|^{-2}\}\subset\cup B_i$ then the  energetic cost in a ball $B(x_\v,\v^\mu)$ is not sufficiently large comparing to our error term.

In the next proposition we present the good framework of vortex balls required in the study. The first step in the study is an   energetic decomposition valid under some assumptions [no assumption on $\delta\in(0;1)$ is required].
\begin{prop}\label{Docmpen} Let $\borneh>1$, $(v_\v)_{0<\v<1}\subset H^1(\O,\C)$ and $\h>0$ be s.t.
 \begin{equation}\label{AbsNatBorneuh}
F(v_\v)\leq \borneh|\ln\v|^2,\,\h\leq\borneh|\ln\v|.
\end{equation}

Assume furthermore that $\lambda^{1/4}|\ln\v|\to0$ and, for $\v\in(0;1)$, either $|v_\v|>1/2$ in $\O$ or $v_\v$ admits a family of  valued disks  $\{(B(a_i,r_i),d_i)\,|\,i\in \J\}$ [$ \J$ is finite] s.t. :
\begin{itemize}
\item[$\bullet$] the disks $B_i=B(a_i,r_i)$ are pairwise disjoint 
\item[$\bullet$] $\{|v_\v|\leq1/2\}\subset\cup_{i\in \J} B_i$
\item[$\bullet$] $\sum_{i\in \J} r_i<|\ln\v|^{-10}$
\item[$\bullet$] For $i\in \J$, letting $d_i=\begin{cases}\deg_{\p B_i}(v)&\text{if }B_i\subset\O\\0&\text{otherwise}\end{cases}$, we assume  $\sum_{i\in\J}|d_i|\leq\borneh$.
\end{itemize}
Then, if $(\xi_\v)_\v\subset H^1_0\cap H^2\cap W^{1,\infty}(\O,\R)$ is s.t.
\begin{equation}\label{BorneXiPourLaDec}
\|\n\xi_\v\|_{L^{\infty}(\O)}\leq\borneh|\ln\v|,
\end{equation}
writing $\zeta_\v:=\xi_\v-\h\xi_0$  we have in the case $|v_\v|\not>1/2$ in $\O$:
 \begin{equation}\label{FullDecDiscqVal0}
\F(v_\v,\n^\bot\xi_\v)-\h^2\Jo=F(v_\v)+2\pi\h\sum_{i\in \J} d_i\xi_0(a_i)+\tilde{V}_\ad(\zeta_\v)+o(1)
\end{equation}
where for $\zeta\in H^1_0\cap H^2(\O)$ we denoted
 \begin{equation}\label{FullDecDiscqValAlt}
\tilde{V}_\ad(\zeta):=2\pi\sum_{i\in \J}d_i\zeta(a_i)+\dfrac{1}{2}\int_\O(\Delta\zeta)^2+|\n\zeta|^2.
\end{equation}
And if $|v|>1/2$ in $\O$ then
 \begin{equation}\label{FullDecDiscqVal0Bisso}
\F(v_\v,\n^\bot\xi_\v)-\h^2\Jo=F(v_\v)+\dfrac{1}{2}\int_\O(\Delta\zeta_\v)^2+|\n\zeta_\v|^2+o(1)
\end{equation}

\end{prop}
The proof of Proposition \ref{Docmpen} is an adaptation of an argument of Serfaty \cite{S1} [section 4]. The proof is presented Appendix \ref{Sec.PreuveDocmpen}

Before going further, we state a result which will be useful in this article and whose proof is left to the reader.
\begin{lem}\label{LemAuxConstructMagnPot}
For $v\in H^1(\O,\C)$, $0<\v<1$ and $\h>0$, there exists a unique potential $A_{v,\v,\h}=A_v\in H^1(\O,\R^2)$ s.t. $(v,A_v)$ is in the Coulomb gauge and satisfies
\begin{equation}\label{MagnetiqueEq}
\begin{cases}-{\n^\bot \rot(A_v)}{}=\alpha(\imath v)\cdot(\n v-\imath A_v v)&\text{in }\O\\\rot(A_v)=\h&\text{on }\p\O\end{cases}.
\end{equation}
Moreover $A_v$ is the unique solution of the minimization problem
 \begin{equation}\label{Eq.MinPb.Pot}
 \inf_{A\text{ satisfies \eqref{JaugeCoulomb}}}\F_{\v,\h}(v,A)
\end{equation}
and from \eqref{CoulombH2} and \eqref{RepresentCoulomGauge} we have $A_v=\n^\bot\xi_v$ with $\xi_v\in H^1_0\cap H^2\cap W^{1,\infty}(\O,\R)$.
\end{lem}
\begin{remark}\label{RemCOntrolTrainRER}
Assume $\lambda,\delta$ satisfy \eqref{CondOnLambdaDelta}, $\delta^2|\ln\v|\leq1$ and
Hypothesis \eqref{BorneKMagn} holds. Consider  $\{(v_\v,A_\v)\,|\,0<\v<1\}\subset\H$  a family  in the Coulomb gauge satisfying $\F(v_\v,A_\v)\leq\inf_{\H}\F+\mathcal{O}(\ln|\ln\v|)$. 
\begin{itemize}
\item From Theorem \ref{ThmBorneDegréMinGlob}, either $|v_\v|>1-|\ln\v|^{-2}$ in $\O$ or the family of disjoint disks given by Proposition \ref{Prop.BorneInfLocaliseeSandSerf} satisfies the properties of the family of discs used in Proposition \ref{Docmpen}.
\item Let $A_{v_\v}=\n^\bot\xi_{v_\v}\in H^1(\O,\R^2)$ be given by Lemma \ref{LemAuxConstructMagnPot}. Then with \eqref{CoulombH2}$\&$\eqref{MagnetiqueEq} we have $A_{v_\v}\in L^\infty(\O)$ and $\|A_{v_\v}\|_{L^\infty(\O)}\leq C |\ln\v|$ where $C$ depends only on $\O$.
\end{itemize}
\end{remark}

As noted by  Serfaty \cite{S1}, with the help of the decomposition given by Proposition \ref{Docmpen}, we may prove that $\h^2 \Jo$ is almost the minimal energy of a vortex less configuration.
\begin{cor}\label{CorEtudeSansVortex}
Let $\H^0:=\left\{(\rho\e^{\imath\varphi},A)\,|\,\rho\in H^1(\O,[0,\infty)),\,\varphi\in H^1(\O,\R)\text{ and }A\in H^1(\O,\R^2)\right\}$. Note that $\H^0$ is gauge invariant. Assume $\lambda^{1/4}|\ln\v|\to0$.
\begin{enumerate}
\item\label{CorEtudeSansVortex1} Let $\v=\v_n\downarrow0$. Assume $\h=\mathcal{O}(|\ln\v|)$ and  for each $\v$ let $(v_\v,\n^\bot\xi_\v)\in\H^0$ be s.t. $\xi_\v\in H^1_0\cap H^2\cap W^{1,\infty}(\O,\R)$ with $\|\n\xi_\v\|_{L^\infty(\O)}=\mathcal{O}(|\ln\v|)$. Writing  $\zeta_\v:=\xi_\v-\h\xi_0$ we have:
 \begin{equation}\label{FullDecDiscqValH0}
\F(v_\v,\n^\bot\xi_\v)=\h^2\Jo+F(v_\v)+\dfrac{1}{2}\int_\O(\Delta\zeta_\v)^2+|\n\zeta_\v|^2+o(1).
\end{equation}
Thus, if $\F(v_\v,\n^\bot\xi_\v)\leq\h^2\Jo+o(1)$ then $\zeta_\v\to0$ in $H^2(\O)$, $|v_\v|\to1$ in $H^1(\O)$ and, up to pass to a subsequence, there exists ${\tt v}\in\S^1$ s.t. $v_\v\to{\tt v}$ in $H^1(\O)$.
\item\label{CorEtudeSansVortex2} We have $\inf_{\H^0}\F=\h^2\Jo+o(1)$.
\end{enumerate}
\end{cor}
\begin{proof}
We prove the first assertion. Estimate \eqref{FullDecDiscqValH0} is a direct consequence of  Proposition \ref{Docmpen}.

For sake of simplicity of the presentation we drop the subscript $\v$. If $\F(v,\n^\bot\xi)\leq\h^2\Jo+o(1)$, then $F(v)+\|\zeta\|_{H^2(\O)}=o(1)$ and then $\zeta\to0$ in $H^2(\O)$, $|v|\to1$ in $H^1(\O)$. Moreover $\|\n v\|_{L^2(\O)}=o(1)$ and $\|v\|_{L^2(\O)}=\mathcal{O}(1)$. This clearly implies the remaining part of the assertion.\\

We prove the second assertion. We first claim, by the definition of $\Jo$, that using the configuration $(1,\h\n^\bot\xi_0)\in\H^0$ we have $ \inf_{\H^0}\F\leq\h^2\Jo+o(1)$. 

By the gauge invariance of $\H^0$ we may consider a family of quasi-minimizer $\{(v_\v,A_\v)\,|\,0<\v<1\}\subset\H^0$ which is in the Coulomb gauge. We write $(v_\v,A_\v)=(v,A)$. Let $(\tilde v,\tilde A)\in\H^0$ be defined by  $\tilde{v}=\underline{v}$ and $\tilde{A}$ is the unique solution of \eqref{Eq.MinPb.Pot} associated to $\tilde v$.

By direct calculations we have: $\F(\tilde v,\tilde A)\leq\F(\tilde v, A)\leq \F( v, A)\leq\h^2 \Jo+o(1)$.

Moreover, by denoting $h:=\rot(\tilde A)$, we have $\n h=\alpha\tilde v\wedge(\n^\bot\tilde v-\tilde A^\bot\tilde v)$ in $\O$ and $h=\h$ on $\p\O$. Then $\|h\|_{H^1(\O)}=\mathcal{O}(|\ln\v|)$ and using \eqref{EstH3} we get $\|\tilde A\|_{H^2(\O)}=\mathcal{O}(|\ln\v|)$.

We are then able to apply the first assertion to get $\F(\tilde v,\tilde A)\geq \h^2\Jo+o(1)$.

\end{proof}
\subsection{Pseudo vortex structure}
We assume $\lambda^{1/4}|\ln\v|\to0$. Let $\{(v_\v,A_\v)\,|\,0<\v<1\}\subset\H$ be a family of configurations in the Coulomb gauge satisfying \eqref{AbsNatBorneuh}. We  assume that $|v_\v|\not>1/2$ in $\O$ and that there exists $\{(B(a_i,r_i),d_i)\,|\,i\in \J\}$ as in  Proposition \ref{Docmpen}. Then Proposition \ref{Docmpen} gives a decomposition of $\F(v,A)$. Except in the crucial hypothesis $\sum r_i<|\ln\v|^{-10}$, the radii  $r_i$ do not play any role as well as the disks "$B(a_i,r_i)$" associated to a zero degree. We thus introduce an {\it ad-hoc} notion of  {\it pseudo vortex}.

\begin{defi}\label{def.PseudoVortex}
We assume that we have either $\v=\v_n\downarrow0$ or $0<\v<1$. We consider $(v_\v)_\v\subset H^1(\O,\C)$, $(\h)_\v\subset(1,\infty)$ satisfying \eqref{AbsNatBorneuh}.
 
 Let $\{B_i=B(a_i,r_i)\,|\,i\in \J\}$ be a family of disks as in Proposition \ref{Docmpen} and let $d_i=d^{(\v)}_i\in\Z$ be the associated "degrees" defined in Proposition \ref{Docmpen}. We denote $\J'=\J'_\v:=\{i\in\J\,|\,d_i\neq0\}$ [note that we have ${\rm Card}(\J'_\v)\leq\sum|d_i|=\mathcal{O}(1)$]. 

 If $\J'\neq\emptyset$, then we say that $\{\ad\}=\{(a_i,d_i)\,|\,i\in\J'\}$ is a set of {\it pseudo vortices} of $v_\v$.

\end{defi}

For a fixed configuration $\ad$ of pseudo vortices, Serfaty studied in \cite{S1} the minimization problem of $\tilde V_\ad$ [defined in \eqref{FullDecDiscqValAlt}]. We have the following result [Proposition 4.2  in \cite{S1}].
\begin{prop}\label{PropPartieMinimalSandH0}
Let $\ad=\{(a_i,d_i)\,|\,i\in\J'\}\subset\O\times\Z^*$ be a configuration s.t. $1\leq\Card(\J')<\infty$ and $a_i\neq a_j$ for $i\neq j$. Then $\tilde V_\ad(\zeta)$ is minimal for $\zeta=\zeta_\ad$ which satisfies
\begin{equation}\label{LondonEqModifie}
\begin{cases}
-\Delta^2\zeta_\ad+\Delta\zeta_\ad=2\pi\sum_{i\in\J'}d_i\delta_{a_i}&\text{in }\O
\\
\zeta_\ad=\Delta\zeta_\ad=0&\text{on }\p\O
\end{cases}.
\end{equation}
[Here $\delta_a$ is the Dirac mass at $a\in\R^2$]

And we have 
$\tilde{V}[\zeta_\ad]=\pi\sum_{i\in\J'}d_i\zeta_\ad(a_i)$.
\end{prop}
In order to prove the above  proposition, Serfaty introduced for $a\in\O$ the function $\zeta^a\in H^1_0\cap H^2(\O)$ which is the unique solution of 
\begin{equation}\nonumber
\begin{cases}
-\Delta^2\zeta^a+\Delta\zeta^a=2\pi\delta_{a}&\text{in }\O
\\
\zeta^a=\Delta\zeta^a=0&\text{on }\p\O
\end{cases}.
\end{equation}
In particular we have $\zeta^a\leq0$ in $\O$. It is easy to see that
$\zeta_\ad=\sum_{i\in\J'}d_i\zeta^{a_i}$  is the unique solution of \eqref{LondonEqModifie}. 

Lemma 4.6 in \cite{S1} gives important properties related with $\zeta^a$ and $\zeta_\ad$:
\begin{prop}\label{Prop.Information.Zeta-a}
For  $s\in(0,1)$, there exists $C_s>0$ s.t. for $a,b\in\O$
\[
\|\zeta^a\|_{L^\infty(\O)}\leq C_s\dist(a,\p\O)^s
\]
and
\[
\|\zeta^a-\zeta^b\|_{H^2(\O)}\leq C_s|a-b|^s.
\]

Consequently there exists $C>0$ depending only on $\O$ s.t., if $\zeta_\ad$ is the unique solution of \eqref{LondonEqModifie}, then
\[
\tilde{V}[\zeta_\ad]=\pi\sum_{i,j\in\J'}d_id_j\zeta^{a_i}(a_j)\leq C\left(\sum_{i\in\J'}|d_i|\right)^2.
\]
\end{prop}
For a further use we need the following lemma.
\begin{lem}\label{Rk.RegularityLondonModified}
Let $\ad$ as in Proposition \ref{PropPartieMinimalSandH0} then $\zeta_\ad\in H^1_0\cap H^2\cap W^{1,\infty}(\O,\R)$ and  there is $C\geq1$ depending only on $\O$ s.t.
\[
\|\n\zeta_\ad\|_{L^\infty(\O)}\leq\dfrac{C\sum|d_i|}{\min \dist(a_i,\p\O)}.
\]
\end{lem}
\begin{proof}
Let $\ad$ be as in Proposition \ref{PropPartieMinimalSandH0}, with Proposition \ref{Prop.Information.Zeta-a} we have $\zeta_\ad=\sum d_i\zeta^{a_i}\in H^1_0\cap H^2$ and $\|\zeta_\ad\|_{H^2(\O)}\leq C\sum_i|d_i|$ where $C$ depends only on $\O$.

Moreover, for $a\in\O$, from \eqref{LondonEqModifie}, we have $\Delta\zeta_\ad=\zeta_\ad-\sum d_i\ln|x-a_i|-R_\ad$ where $R_\ad$ is the harmonic extension of $\tr_{\p\O}(-\sum d_i\ln|x-a_i|)$ in $\O$. 

Consequently there exists $C\geq1$ depending only on $\O$ s.t. 
\[
\|\Delta\zeta_\ad\|_{L^3(\O)}\leq\dfrac{C\sum|d_i|}{\min \dist(a_i,\p\O)}
\]
 and therefore by elliptic regularity and a Sobolev embedding we get the result.
\end{proof}
Until now, the only way to get a nice magnetic potential associated to a function $v$ was to consider  $A_v=A_{v,\v,\alpha}\in H^2(\O,\R^2)$, the unique solution of \eqref{Eq.MinPb.Pot}. The previous results give that, after the cleaning step, we can do asymptotically as well by using a magnetic potential depending on a pseudo vortices structure of $v$ instead of $v$ itself [see Remark \ref{Prop.BornéPourPotenteifjfjf}]. 

\begin{defi}\label{DefA_ad}
Let $N\geq1$ and $\ad\in\Ostar\times(\Z^*)^N$, $\h>0$. Then we define $A_\ad:=\h\n^\bot\xi_0+\n^\bot\zeta_\ad$ where $\zeta_\ad$  is the unique solution of $\eqref{LondonEqModifie}$, {\it the potential associated to $\ad$}.
\end{defi}
\begin{remark}\label{Prop.BornéPourPotenteifjfjf}
Let $\borneh>1$ and $(v_\v)_{0<\v<1}\subset H^1(\O,\C)$, $\h>0$  satisfying \eqref{AbsNatBorneuh} be s.t. $(v_\v)_{0<\v<1}$ admits a set of pseudo vortices $(\ad_\v)_{0<\v<1}$ with $\sum |d_i|\leq \borneh$. We write $v\&\ad$ instead of $v_\v\&\ad_\v$. 

Assume $\min\dist(a_i,\p\O)>|\ln\v|^{-1}$ in order to have $\|\n\zeta_{\ad}\|_{L^\infty(\O)}=\mathcal{O}(|\ln\v|)$ [with Lemma \ref{Rk.RegularityLondonModified}] and  $\lambda^{1/4}|\ln\v|\to0$. 

For $0<\v<1$, let $A_{v}\in H^1(\O,\R^2)$ be the unique solution of \eqref{Eq.MinPb.Pot} and $A_{\ad}$ be defined in Definition \ref{DefA_ad}. Then we have $A_{\ad}=\n^\bot\xi_{\ad}$ and  $A_{v}=\n^\bot\xi_{v}$ where $\xi_{\ad},\xi_{v}\in H^1_0\cap H^2\cap W^{1,\infty}(\O,\R)$  satisfy the hypotheses of Proposition \ref{Docmpen} [here we used \eqref{CoulombH2}$\&$\eqref{MagnetiqueEq}]. Therefore we have the following inequalities
\[
\F(v,0)\geq\F(v,A_{v})=\tilde{\F}(v,A_{v})+o(1)\geq\tilde{\F}(v,A_{\ad})+o(1),
\]
\[
\F(v,A_{v})\leq{\F}(v,A_\ad)=\tilde{\F}(v,A_\ad)+o(1).
\]
In particular we have $\F(v,A_{v})=\mathcal{O}(|\ln\v|^2)$ and $\F(v,A_{\ad})=\mathcal{O}(|\ln\v|^2)$.
\end{remark}
\subsection{Cluster of pseudo vortices}
From a standard result for the homogenous case, it is expected that, for a reasonable magnetic field, the asymptotic location of pseudo vortices of a studied configuration is a subset of $\Lambda$. This problem is related to the {\it macroscopic location} of the pseudo vortices. To treat this problem we use an {\it ad-hoc} notion of {\it cluster of pseudo vortices}.
\begin{defi}
Let $N,\tilde N_0\in\N^*$, $\tilde N_0\leq N$, $\pD\in(\overline{\O}^{\tilde N_0})^*\times\Z^{\tilde N_0}$, $\v=\v_n\downarrow0$ and $\ad_\v\in\Ostar\times\Z^N$ s.t. ${\bf d}$ is independent of $\v$. We say that $(\ad_\v)_\v$ admits a {\it cluster structure} on $\pD$ if
\begin{itemize}
\item for $i\in\{1,...,N\}$, $\lim a_i$ exists, $\lim a_i\in\{p_1,...,p_{\tilde N_0}\}$ and we write for  $k\in\{1,...,\tilde N_0\}$, $S_k:=\{i\in\{1,...,N\}\,|\,a_i\to p_k\}$
\item for $k\in\{1,...,\tilde N_0\}$ $S_k\neq\emptyset$,
\item for  $k\in\{1,...,\tilde N_0\}$, $D_k=\sum_{i\in S_k}d_i$.
\end{itemize}
\end{defi}
\begin{remark}
In this article we will use the notion of cluster structure with $\ad$ as in Proposition \ref{Docmpen} and ${\bf p}\subset\Lambda$.
\end{remark}
\begin{prop}\label{PropClusterI}Let $ N\geq1$, $\v=\v_n\downarrow0$, $\ad_\v\in\Ostar\times\Z^N$ s.t. $\sum|d_i|$ is bounded independently of $\v$.
\begin{enumerate}
\item\label{PropClusterI1} If $(\ad_\v)_\v$ admits a cluster structure on $\pD$ [and then ${\bf d}$ is independent of $\v$] then $\pD$ is unique [up to change the order]. We say that $\pD$ is the cluster of $(\ad_\v)_\v$.
\item\label{PropClusterI2} Up to pass to a subsequence, there exist $1\leq \tilde N_0\leq N$ and $\pD\in(\overline{\O}^{\tilde N_0})^*\times\Z^{\tilde N_0}$ s.t. $\pD$ is the cluster of $(\ad_\v)_\v$.
\item\label{PropClusterI3} If $\pD$ is the cluster of $(\ad_\v)_\v$ then, denoting $\chi:=\max_k\max_{i\in S_k}|a^\v_i-p_k|$, we have
\begin{equation}\label{SplitClustXi0}
\left|\sum_{k=1}^{\tilde N_0}\sum_{i\in S_k}|d_i||\xi_0(a^\v_i)-\xi_0(p_k)|\right|\leq C\chi
\end{equation}
and
\begin{equation}\label{SplitClustTildeV}
\left|\tilde V[\zeta_{\ad_\v}]-\tilde V[\zeta_\pD]\right|\leq C\sqrt\chi
\end{equation}
where $C$ depends only on $N$, $\sum|d_i|$ and $\O$.
\end{enumerate}
\end{prop}
\begin{proof}
The two first assertions are obvious. Estimate \eqref{SplitClustXi0} is direct by noting that $\xi_0$ a Lipschitzian function in $\O$. Estimate \eqref{SplitClustTildeV} is a direct consequence of Proposition \ref{Prop.Information.Zeta-a}.
\end{proof}

We then have:
\begin{cor}\label{Cor.DecompPourCluster}Assume that $\lambda,\delta,\h$ satisfy \eqref{CondOnLambdaDelta} and \eqref{BorneKMagn} for some $K\geq0$  independent of $\v$. Assume also  $\delta^2|\ln\v|\leq1$.

Let $\{(v_\v,A_\v)\,|\,0<\v<1\}\subset\H$ be a family s.t. $\F(v_\v,A_\v)\leq\inf_\H\F+K\ln|\ln\v|$ which is in the  Coulomb gauge and let  $\{({\bf a}_\v,{\bf d}_\v)=\ad\,|\,0<\v<1\}$ be a family of pseudo vortices associated to $\{(v_\v,A_\v)\,|\,0<\v<1\}$ [indexed on $\J=\J_\v$ possibly empty].

\begin{enumerate}
\item Letting $A_{v_\v}\in H^1(\O,\R^2)$ be defined by Lemma \ref{LemAuxConstructMagnPot} we have
\begin{equation}\label{NiceDecSharp}
\F(v_\v,A_\v)\geq\F(v_\v,A_{v_\v})\geq\h^2\Jo+2\pi\h\sum_{i\in\J}d_i\xi_0(a_i)+F(v_\v)+\tilde{V}[\zeta_\ad]+o(1).
\end{equation}
And then
\begin{equation}\label{NiceDec}
\F(v_\v,A_\v)\geq \h^2\Jo+2\pi\h\sum_{i\in\J}d_i\xi_0(a_i)+F(v_\v)+\mathcal{O}(1).
\end{equation}
\item Assume furthermore that  $\ad$ admits a cluster structure on $\pD$. Then we have
\begin{equation}\label{NiceDecSharpSplitTildeV}
\F(v_\v,A_\v)\geq\h^2\Jo+2\pi\h\sum_{i\in\J}d_i\xi_0(a_i)+F(v_\v)+\tilde{V}[\zeta_\pD]+o(1).
\end{equation}
\end{enumerate}

\end{cor}
\begin{proof}
The lower bounds \eqref{NiceDecSharp} and \eqref{NiceDec} are direct consequences of Theorem \ref{ThmBorneDegréMinGlob}, Lemma \ref{LemAuxConstructMagnPot}, Remark \ref{RemCOntrolTrainRER} and Propositions \ref{Prop.BornesSups1}$\&$\ref{Docmpen}$\&$\ref{PropPartieMinimalSandH0}.

Estimate \eqref{NiceDecSharpSplitTildeV} is a direct consequence of Proposition \ref{PropClusterI} and \eqref{NiceDecSharp}.
\end{proof}
We then have the following corollary.
\begin{cor}\label{CorGonzo}Assume that $\lambda,\delta,\h$ satisfy \eqref{CondOnLambdaDelta} and \eqref{BorneKMagn}. Assume also  $\delta^2|\ln\v|\leq1$.

Let $(v_\v)_{0<\v<1}\subset H^1(\O,\C)$ be s.t. $|v_\v|\not>1/2$ in $\O$ and assume the existence of  $(B_\v)_{0<\v<1}\subset H^1(\O,\R^2)$ s.t. $(v_\v,B_\v)$ is in the Coulomb gauge and $\F(v_\v,B_\v)\leq\inf_\H\F+\mathcal{O}(\ln|\ln\v|)$. Assume also that $({\bf a}_\v,{\bf d}_\v)=({\bf a},{\bf d})$ are pseudo-vortices as in Definition \ref{def.PseudoVortex} for $v_\v$ [note that we thus have $\sum |d_i|=\mathcal{O}(1)$], then 
\begin{equation}\label{NiceDecSharpBorneSup}
\F(v_\v,A_{({\bf a},{\bf d})})=\h^2\Jo+2\pi\h\sum_{}d_i\xi_0(a_i)+F(v_\v)+\tilde{V}[\zeta_{({\bf a},{\bf d})}]+o(1).
\end{equation}
where $A_{({\bf a},{\bf d})}:=\h\n^\bot\xi_0+\n^\bot\zeta_{({\bf a},{\bf d})}$.

Consequently we get
\begin{equation}\label{NiceDecSharpBorneSupBisso}
F(v_\v)\leq2\pi\h\sum_{}d_i|\xi_0(a_i)|+\mathcal{O}(\ln|\ln\v|)\leq \pi b^2\sum_{}|d_i||\ln\v|+\mathcal{O}(\ln|\ln\v|).
\end{equation}
\end{cor}
\begin{proof}
Corollary \ref{CorGonzo} is a direct consequence of $\inf_\H\F\leq \h^2\Jo$, Corollary \ref{Cor.DecompPourCluster} and Propositions \ref{Docmpen}$\&$\ref{Prop.Information.Zeta-a}.
\end{proof}
\begin{remark}
We may state an analog of Corollary \ref{CorGonzo} if $\ad$ admits a structure of cluster.
\end{remark}

\section{Renormalized energies}\label{Sec.RenEn}

\subsection{Macroscopic renormalized energy [at scale $1$]}\label{SecMacroRenEn}
We consider in this section:
\begin{itemize}
\item[$\bullet$] $N\in\N^*$, ${\bf z}={\bf z}^{(n)}\in\Ostar:=\{(z_1,...,z_N)\subset\O\,|\,z_i\neq z_j\text{ pour }i\neq j \}$,
\item[$\bullet$] ${\bf d}=(d_1,...,d_N)\in\Z^N$.
\item[$\bullet$] $\tae=\tae({\bf z}):=\min_i\dist(z_i,\p\O)$
\end{itemize}
We are going to deal with functions defined in the  set $\O$ perforated by disks with radius $\Rad=\Rad_n\downarrow0$:
\[
\O_\Rad=\O_\Rad({\bf z}):=\O\setminus\cup_i\overline{B(z_i,\Rad)}.
\]

We assume
\begin{equation}\label{HypRayClass}
\Rad<\dfrac{1}{8}\min\left\{\min_{i\neq j} |z_i-z_j|\,;\,\tae\right\}.
\end{equation}
For a radius $\Rad>0$ s.t. \eqref{HypRayClass} is satisfied, we consider the set of functions
\[
\I^{\rm deg}_\Rad:=\left\{w\in H^1(\O_\Rad,\S^1)\,|\,\deg_{\p B(z_i,\Rad)}(w)=d_i\text{ for }i\in\{1,...,N\}\right\}
\]
and
\[
\I^{\rm Dir}_\Rad:=\left\{w\in H^1(\O_\Rad,\S^1)\,\left|\,\begin{array}{c}w(z_i+\Rad\e^{\imath\theta})={ C}_i\e^{\imath d_i\theta}\text{ for }i\in\{1,...,N\},\\\,({C}_1,...,{ C}_N)\in(\S^1)^N\end{array}\right.\right\}.
\]
In this section we are interested in the minimization of the Dirichlet functional in $\I^{\rm deg}_\Rad$ and $\I^{\rm Dir}_\Rad$.

Before beginning we state an easy result proved by direct minimization [the proof is left to the reader, see \cite{BBH}].
\begin{prop}
For $N\geq 1$, $\zd\in\Ostar\times\Z^N$ and $\Rad>0$ s.t. \eqref{HypRayClass} is satisfied, the following minimization problems admit solutions:
\begin{equation}\label{MinPropDeg}
I_\Rad^{\rm deg}=I_\Rad^{\rm deg}\zd:=\inf_{w\in\I_\Rad^{\rm deg}}\dfrac{1}{2}\int_{\O_\Rad}|\n w|^2
\end{equation}
and
\begin{equation}\label{MinPropDir}
I_\Rad^{\rm Dir}=I_\Rad^{\rm Dir}\zd:=\inf_{w\in\I_\Rad^{\rm Dir}}\dfrac{1}{2}\int_{\O_\Rad}|\n w|^2.
\end{equation}
Moreover, these solutions are unique up to the multiplication by an $\S^1$ constant.
\end{prop}
\subsubsection{Study of $I^{\rm deg}_\Rad$ and $I^{\rm Dir}_\Rad$}\label{UnimSection}
Following \cite{BBH}, it is standard to define the {\it canonical harmonic map associated to $\zd$}.
\begin{defi}\label{DefApplican}Let $N\in\N^*$ and $\zd\in\Ostar\times\Z^N$. A function $\wstar\in \cap_{0<p<2}W^{1,p}(\O,\S^1)\cap C^\infty(\O\setminus\{z_1,...,z_N\},\S^1)$ is the {\it canonical harmonic map associated to the singularities $\zd$} if
\begin{equation}\label{ApplicanAssocieSingDeg}
\wstar(z)={\rm e}^{\imath\varphi_\star(z)}\prod_{i=1}^N\left(\dfrac{z-z_i}{|z-z_i|}\right)^{d_i}\text{ with }\begin{cases}\text{$\varphi_\star$ is harmonic in $\O$}\\\p_\nu\wstar=0\text{ on }\p\O,\,\di\int_{\p\O}\varphi_\star=0\end{cases}.
\end{equation}
\end{defi}
\begin{remark}\label{Remark.DefConjuHarmPhase}
In this framework, it is classic to define $\Pstar$ [with the notation of Definition \ref{DefApplican}], the unique solution of 
\[
\begin{cases}
\Delta\Pstar=2\pi\sum_{i=1}^Nd_i\delta_{z_i}&\text{in }\O
\\
\Pstar=0&\text{on }\p\O
\end{cases}.
\]
This function satisfies $\n^\bot\Pstar=\wstar\wedge\n\wstar$. Moreover, by denoting $R_\zd$ the unique solution of
\[
\begin{cases}
\Delta R_\zd=0&\text{in }\O
\\
R_\zd(z)=-\sum_i d_i\ln|z-z_i|&\text{on }\p\O
\end{cases},
\]
we have $\Pstar(z)=\sum_i d_i\ln|z-z_i|+R_\zd(z)$.
\end{remark}
We first study the asymptotic behavior of minimizers of $I^{\rm deg}_\Rad\zd$ when $\Rad\to0$.
\begin{prop}\label{MinimalMapHomo}
Let $N\in\N^*$, $\zd=\zd^{(n)}\subset\Ostar\times\Z^N$ and $\tae:=\min_i\dist(z_i,\p\O)$. We assume that $\sum_i|d_i|=\mathcal{O}(1)$. 

For $\Rad>0$ s.t. \eqref{HypRayClass} is satisfied, we may consider $\wr$, the unique solution of the problem
\begin{equation}\label{ShriHoleBIDeg1}
I^{\rm deg}_\Rad\zd:=\inf_{w\in\I_\Rad^{\rm deg}}\dfrac{1}{2}\int_{\O_\Rad}|\n w|^2,
\end{equation}
of the form
\begin{equation}\label{ExprMinShrSol}
\wr(z)={\rm e}^{\imath\varphi_\Rad(z)}\prod_{i=1}^N\left(\dfrac{z-z_i}{|z-z_i|}\right)^{d_i}\text{ with }\begin{cases}\varphi_\Rad\in H^1\cap C^\infty(\O_\Rad,\R)\\\di\int_{\p\O}\varphi_\Rad=0\end{cases}.
\end{equation}
We thus have the existence of $C>0$ [depending only on $\O,N$ and the bound of $\sum_i|d_i|$] s.t.
\begin{equation}\label{BorneGradWstar}
\|\n \wstar\|_{L^\infty(\O_\Rad)}\leq\dfrac{C(1+|\ln\Rad|)}{\Rad}.
\end{equation}
We denote
\begin{equation}\label{DefX}
X:=\begin{cases}
\dfrac{\Rad(1+|\ln(\tae)|)}{\tae}\left(1+ \dfrac{\Rad(1+|\ln(\tae)|)}{\tae}\right)&\text{if }N=1
\\
\left(\dfrac{\Rad}{\min_{i\neq{j}}|z_{i}-z_j|}+\dfrac{\Rad(1+|\ln(\tae)|)}{\tae}\right)\left(1+ \dfrac{\Rad(1+|\ln(\tae)|)}{\tae}\right)&\text{if }N\geq2
\end{cases}
\end{equation}
and we have
\begin{equation}\label{ConvergenceH1ShrHolklgBorne}
\|\varphi_\Rad-\varphi_\star\|^2_{H^1(\O_\Rad)}\leq C X,
\end{equation}
\begin{equation}\label{ConvergenceShrHolklgBorne}
0\leq \dfrac{1}{2}\int_{\O_\Rad}|\n\wstar|^2-\inf_{w\in\I_\Rad^{\rm deg}}\dfrac{1}{2}\int_{\O_\Rad}|\n w|^2\leq CX.
\end{equation}
Moreover, if there exists $\eta>0$ [independent of $n$] s.t. $\tae>\eta$ then \eqref{BorneGradWstar} may be refined into
\begin{equation}\label{BorneGradWstarSpeciale}
\|\n \wstar\|_{L^\infty(\O_\Rad)}\leq\dfrac{C}{\Rad}.
\end{equation}
\end{prop}
The proof of Proposition \ref{MinimalMapHomo} is in Appendix \ref{PreuvePropUniModComp}.

By adapting the proof of Proposition 5.1 in \cite{S1} we have
\begin{prop}\label{Prop.EnergieRenDef}
For $N\geq 1$, there exists an application $W^{\tt macro}_N=W^{\tt macro}:\Ostar\times\Z^N\to\R$ s.t. for sequences $\zd=\zd^{(n)}\in\Ostar\times\Z^N$ and $\Rad=\Rad_n\to0$ satisfying \eqref{HypRayClass} and s.t. ${\bf d}$ is independent of  $n$, there exists $C\geq1$ [depending only on $N$, $\sum|d_i|$ and $\O$] s.t.
\[
\left|\dfrac{1}{2}\int_{\O_\Rad}|\n \wstar|^2-\pi\sum_id_i^2|\ln\Rad|-W^{\tt macro}\zd\right|\leq CX
\]
with
\[
W^{\tt macro}\zd=-\pi\sum_{i\neq j}d_id_j\ln|z_i-z_j|-\pi\sum_id_iR_\zd(z_i),
\]
\[
R_\zd\in C^\infty(\O,\R)\text{ satisfies }\|R_\zd\|_{L^\infty(\O)}\leq C(1+|\ln\tae|).
\]
\end{prop}
Proposition \ref{Prop.EnergieRenDef} is proved in \ref{PreuvelammeShrinkSerfaty}. We immediately obtain from Proposition \ref{Prop.EnergieRenDef} the following corollary.
\begin{cor}\label{CorBorneGrossEneStar}
Under the hypotheses of Proposition \ref{Prop.EnergieRenDef} and assuming that there exists $C_1>0$ [independent of $\rad$] s.t. $\dfrac{\Rad(1+|\ln\tae|)}{\tae}\leq C_1$, there is $C>1$ [depending only on $\O$, $N$, $\sum_i |d_i|$ and $C_1$] s.t. $\displaystyle\int_{\O_\Rad}|\n \wstar|^2\leq C|\ln\Rad|$.

\end{cor}
We end this section by linking $I^{\rm deg}_\Rad$ and $I^{\rm Dir}_\Rad$.
\begin{prop}\label{Prop.ConditionDirEnergieRen}
Let $N\geq1$, ${\bf z}\in\Ostar$ and $\Rad=\Rad_n\downarrow0$ satisfying \eqref{HypRayClass}. Assume $\dfrac{\Rad}{\tae}\to0$  and   if $N\geq2$, we also assume $\dfrac{\Rad}{\min_{i\neq j}|z_i-z_j|}\to0$.

Let
\[
\eta:=\begin{cases}10^{-1}\tae&\text{if }N=1\\10^{-1}\min\{\tae\,;\,\min_{i\neq j}|z_i-z_j|\}&\text{if }N\geq2\end{cases}.
\]
Assume furthermore
\[
Z:=\dfrac{1}{\ln(\eta/\Rad)}\left[\dfrac{\eta(1+|\ln(\tae)|)}{\tae}+1\right]^2\to0.
\]
Then for ${\bf d}\in\Z^N$ [independent of $n$], there exists $C>1$ [depending only on $\O$, $N$ and $\sum|d_i|$] s.t. 
\[
0\leq\inf_{w\in\I_\Rad^{\rm Dir}}\dfrac{1}{2}\int_{\O_\Rad}|\n w|^2-\inf_{w\in\I_\Rad^{\rm deg}}\dfrac{1}{2}\int_{\O_\Rad}|\n w|^2\leq C(X+Z).
\]
\end{prop}
Proposition \ref{Prop.ConditionDirEnergieRen} is proved Appendix \ref{Sec.PreuvelammeShrinkSerfatyDir}.
\subsubsection{Macroscopic renormalized energy and cluster of vortices}
We first state an easy lemma.
\begin{lem}\label{LemPseudoCondRzd}
\begin{enumerate}
\item Let $N\in\N^*$ and ${\bf d}\in\Z^N$. Let $\chi>0$ and ${\bf z},{\bf z}'\in\Ostar$ be s.t. for $i\in\{1,...,N\}$ we have $|z_i-z_i'|\leq\chi$. Then we have
\[
\|R_\zd-R_{({\bf z}',{\bf d})}\|_{L^\infty(\O)}\leq \sum_i|d_i|\dfrac{\chi}{\max\{\tae({\bf z}),\tae({\bf z'})\}}.
\]
\item Let $1\leq \tilde N_0\leq N$, ${\bf p}\in(\O^{\tilde N_0})^*$, $\zd=\zd^{(n)}\in\Ostar\times\Z^N$ be s.t. ${\bf d}$ is independent of $n$ and for  $i\in\{1,...,N\}$ there exists $k\in\{1,...,\tilde N_0\}$ s.t. $z_i\to p_k$. We let $\chi:=\max_i\dist(z_i,\{p_1,...,p_{\tilde N_0}\})$.

For $k\in\{1,...,\tilde N_0\}$ we let $D_k:=\di\sum_{z_i\to p_k}d_i$ and ${\bf D}=(D_1,...,D_{\tilde N_0})$. Then we have
\[
\|R_\zd-R_{\pD}\|_{L^\infty(\O)}\leq \sum_i|d_i|\dfrac{\chi}{\tae({\bf p})}.
\]
\end{enumerate}
\end{lem}
\begin{proof}The first assertion is obtained with the help of the maximum principle and the bound  $|R_\zd-R_{({\bf z}',{\bf d})}|\leq\sum_i|d_i|\dfrac{\chi}{\max\{\tae({\bf z}),\tae({\bf z}')\}}$ on $\p\O$. The second assertion assertion follows by the same way.
\end{proof}
With Lemma \ref{LemPseudoCondRzd} we may exploit a structure of cluster for $W^{\rm macro}$.
\begin{prop}\label{Prop.RenEnergieCluster}
Let $1\leq \tilde N_0\leq N$, ${\bf p}\in(\O^{\tilde N_0})^*$ [independent of $n$] and write
\[
\gamma_{\bf p}:=\begin{cases}1&\text{if }\tilde N_0=1\\{\min_{k\neq l}|p_k-p_l|}&\text{otherwise}\end{cases}.
\]
Let $\zd=\zd^{(n)}\in\Ostar\times\Z^N$ be s.t. ${\bf d}$ is independent of $n$ and for $i\in\{1,...,N\}$ there exists $k\in\{1,...,\tilde N_0\}$ s.t. $z_i\to p_k$. We denote $\chi:=\max_i\dist(z_i,\{p_1,...,p_{\tilde N_0}\})$.


For $k\in\{1,...,\tilde N_0\}$ we denote $D_k:=\di\sum_{z_i\to p_k}d_i$ and ${\bf D}=(D_1,...,D_{\tilde N_0})$. Then there exists $C\geq1$ [depending only on $\O,N$ and $\sum|d_i|$] s.t.
\begin{eqnarray*}
&&\left|W_N^{\rm macro}\zd-\left(W_{\tilde N_0}^{\rm macro}\pD-\pi\sum_{k=1}^{\tilde N_0}\sum_{\substack{z_i,z_j\to p_k\\i\neq j}}d_id_j\ln|z_i-z_j|\right)\right|
\\&\leq& C\chi\left(\dfrac{1+|\ln[\tae({\bf p})]|}{\tae({\bf p})}+\dfrac{1}{\gamma_{\bf p}}\right).
\end{eqnarray*}
\end{prop}
\begin{proof}
We have
\begin{equation}\nonumber
W^{\tt macro}\zd
=-\pi\sum_{k=1}^{\tilde N_0}\sum_{\substack{z_i,z_j\to p_k\\i\neq j}}d_id_j\ln|z_i-z_j|-\pi\sum_{\substack{z_i\to p_k\\z_j\to p_l\\k\neq l}}d_id_j\ln|z_i-z_j|-\pi\sum_id_iR_\zd(z_i).
\end{equation}
It is easy to check that
\begin{eqnarray}\label{ClusterEst1}
\sum_{\substack{z_i\to p_k\\z_j\to p_l\\k\neq l}}d_id_j\ln|z_i-z_j|
&=&\sum_{k\neq l}D_kD_l\ln|p_k-p_l|+H
\end{eqnarray}
with $H\leq 4\left(\sum_i|d_i|\right)^2\dfrac{\chi}{\gamma_{\bf p}}$ for sufficiently large $n$.

On the other hand, from Lemma \ref{LemPseudoCondRzd} [second assertion], we have $\|R_\zd-R_{\pD}\|_{L^\infty(\O)}\leq \sum_i|d_i|\dfrac{\chi}{\max\{\tae({\bf z}),\tae({\bf p})\}}$. From standard pointwise estimates for the gradient of harmonic functions [see \eqref{lkjbblkjn0}] there exists $C\geq1$ depending only on $\O$, $\sum|D_k|$ and $N$ [here we used $1\leq \tilde N_0\leq N$] s.t. for $z_i\to p_k$ we have $\left|R_{\pD}(z_i)-R_{\pD}(p_k)\right|\leq C\chi\dfrac{1+|\ln[\tae({\bf p})]|}{\tae({\bf p})}$.

Then, up to change the value of $C$, we have
\begin{equation}\label{ClusterEst2}
\left|\sum_id_iR_\zd(z_i)-\sum_kD_kR_{\pD}(p_k)\right|\leq C\chi\dfrac{1+|\ln[\tae({\bf p})]|}{\tae({\bf p})}.
\end{equation}
By combining \eqref{ClusterEst1} and \eqref{ClusterEst2} we get the result.
\end{proof}

\subsection{Mesoscopic renormalized energy [at scale $\h^{-1/2}$]}\label{SecMesoRenEn}
From the work of Sandier and Serfaty we may obtain mesoscopic informations. To this end we need to assume a non degeneracy assumption for minimal points of  $\xi_0$. So we assume in this section that Hypothesis \eqref{NonDegHyp} holds.

Let 
\begin{equation}\label{DefEtaO}
\eta_\O:=\begin{cases}10^{-3}\min\{1;\dist(\Lambda,\p\O)\}&\text{if }N_0=1\\10^{-3}\min\{1;\dist(\Lambda,\p\O);\min_{k\neq l}|p_k-p_l|\}&\text{if }N_0\geq 2\end{cases}.
\end{equation}
For $p\in\Lambda$, by applying Lemma 11.1 in \cite{SS1}  in the disk $B(p,\eta_\O)$, we get the following proposition.
\begin{prop}\label{EnergieRenMeso}
Assume that Hypothesis \eqref{NonDegHyp} holds. Let $D\in\N^*$ and $\h\uparrow\infty$ when $\v\to0$. Then for $p\in\Lambda$ and $R=R(\v)\to0$ s.t. $R\sqrt{\h}\to\infty$ we have
\begin{eqnarray}\nonumber
&&\inf_{{\bf z}\in[B(p,R)^D]^*}\left\{-\pi\sum_{i\neq j}\ln|z_i-z_j|+2\pi\h\sum_i[\xi_0(z_i)-\xi_0(p)]\right\}
\\\label{DevMesoscopicDef}&=&\dfrac{\pi}{2}(D^2-D)\ln\left(\dfrac{\h}{D}\right)+C_{p,D}+o(1)
\end{eqnarray}
with
\begin{equation}\label{DefCpD}
C_{p,D}:=\min_{[(\R^2)^D]^*}W^{\rm meso}_{p,D}
\end{equation}
and
\begin{equation}\label{DefEnergyRenMeso}
\begin{array}{cccc}
W^{\rm meso}_{p,D}:&[(\R^2)^D]^*&\to&\R\\&{\bf x}=(x_1,...,x_D)&\mapsto&\di-\pi\sum_{i\neq j}\ln|x_i-x_j|+\pi D\sum_{i=1}^D Q_p(x_i).
\end{array}
\end{equation}
where $Q_p(x):=x\cdot {\rm Hess}_{\xi_0}(p)x$,  ${\rm Hess}_{\xi_0}(p)$ is the Hessian matrix of $\xi_0$ at $p$.

Moreover the infimum in \eqref{DevMesoscopicDef} is reached and if ${\bf z}^\v\in [B(p,R)^D]^*$ is s.t.
\[
-\pi\sum_{i\neq j}\ln|z^\v_i-z^\v_j|+2\pi\h\sum_i[\xi_0(z^\v_i)-\xi_0(p)]=\dfrac{\pi}{2}(D^2-D)\ln\left(\dfrac{\h}{D}\right)+C_{p,D}+o(1)
\]
then for all sequence $\v=\v_n\downarrow0$, up to pass to a subsequence, denoting $\ell=\sqrt{\dfrac{D}{\h}}$ and $\breve z_i^\v=\dfrac{z_i^\v-p}{\ell}$, we have $\breve{\bf z}^\v=(\breve z_1^\v,...,\breve z_D^\v)$ which converges to a minimizer of $W^{\rm meso}_{p,D}$. In particular $|\breve z_i^\v|\leq C_{\O,D}$ with $C_{\O,D}>0$ which depends only on $\O$ and $D$.
\end{prop}

\subsection{Microscopic renormalized energy [at scale $\lambda\delta$]}\label{SecMicrRenEn}
The location of the vorticity defects at scale $\lambda\delta$ [inside a connected component of $\o_\v$] is given by the microscopic renormalized energy exactly as in the case without magnetic field. In order to define the microscopic renormalized energy we need some notation. Recall that the pinning term $a_\v:\O\to\{b,1\}$ is obtained [see Section \ref{SecConstructionPinningTerm}] from  a  smooth bounded simply connected set $\o$ s.t. $0\in\o\subset\overline{\o}\subset Y:=(-1/2,1/2)^2$. The construction of the pinning term uses two parameters $\delta=\delta(\v)$ [the parameter of period] and $\lambda=\lambda(\v)$ [the parameter of dilution]. For $x_0\in\o$ and a sequence $\v=\v_n\downarrow0$, we consider $\hat{x}_\v\in\o$ s.t. $\hat x_\v\to x_0\in\o$. 

Let $m_\v\in\Z^2$ be s.t. the cell $Y_\v=\delta(m_\v+Y)$ satisfies $\overline{Y}_\v\subset\O$. We then denote $z_\v=\delta[m_\v+\lambda\hat{x}_\v]$. It is proved in \cite{dos2015microscopic} [see Estimates (9) and (10)] that for  $ R= R_\v\gg\lambda\delta$ and $r=r_\v\ll\lambda\delta$, denoting $\hat R=R/(\lambda\delta)$, $\hat r=r/(\lambda\delta)$, $\dom_\v=B(\delta m_\v,R)\setminus\overline{B(z_\v,r)}$, $\hat\dom_\v=B(0,\hat R)\setminus\overline{B(\hat x_\v,\hat r)}$ and $\hat\dom=B(0,\hat R)\setminus\overline{B(x_0,\hat r)}$:
\begin{eqnarray}\label{MicroRenoExpressionNonRescalDir}
\inf_{\substack{w\in H^1(\dom_\v,\S^1)\\\deg(w)=1}}\frac{1}{2}\int_{\dom_\v} U_\v^2|\n w|^2&=&\inf_{\substack{w\in H^1(\dom_\v,\S^1)\\w(z_\v+R{\rm e}^{\imath\theta})={\rm e}^{\imath\theta}\\w(x_\v+r{\rm e}^{\imath\theta})={\rm Cst}\,{\rm e}^{\imath\theta}}}\frac{1}{2}\int_{\dom_\v} U_\v^2|\n w|^2+o_\v(1)
\\\label{DefRenMicroEn1}&=&\inf_{\substack{ \hat w\in H^1(\hat\dom_\v,\S^1)\\\deg(w)=1}}\frac{1}{2}\int_{\hat\dom_\v} a^2|\n\hat  w|^2+o_\v(1).
\end{eqnarray}

Moreover from the main result in \cite{Dos-MicroRenoEN}, we have the existence of an application $\tilde W^{\rm micro}:\o\to\R$ [depending only on $\o$ and $b$] s.t. 
\begin{equation}\label{DefRenMicroEn2}
\inf_{\substack{ \hat w\in H^1(\hat\dom_\v,\S^1)\\\deg(w)=1}}\frac{1}{2}\int_{\hat\dom_\v} a^2|\n\hat  w|^2=f_\o(\hat R)+b^2\pi|\ln(\hat r)|+\tilde W^{\rm micro}(x_0)+o(1).
\end{equation}
where $\di f_\o(\hat R):=\inf_{\substack{w\in H^1[B(0,{\hat R})\setminus \overline{\o},\S^1]\\\deg(w)=1}}\frac{1}{2}\int_{B(0,\hat R)\setminus \overline{\o}}|\n w|^2$.
  
It is clear that there exists  $C_\o\in\R$ [depending only on $\o$] s.t. when $\hat R\to\infty$ we have $f_\o(\hat R)=\pi\ln(\hat R)+C_\o$.

Then, by denoting $ W^{\rm micro}(x_0):=\tilde W^{\rm micro}(x_0)+C_\o$, we get from \eqref{DefRenMicroEn2} :
\begin{equation}\label{DefRenMicroEn3}
\inf_{\substack{ \hat w\in H^1(\hat\dom,\S^1)\\\deg(w)=1}}\frac{1}{2}\int_{\hat\dom} a^2|\n\hat  w|^2=\pi\ln(\hat R)+b^2\pi|\ln(\hat r)|+ W^{\rm micro}(x_0)+o(1).
\end{equation}
Moreover, from \cite{Publi3} we know that $W^{\rm micro}$ admits minimizers in $\o$.
\section{Sharp upper bound: construction of a test function}\label{SecUpperBound}
From now on we assume that Hypothesis \eqref{NonDegHyp} holds. We thus may use  for $p\in\Lambda$ and $D\in\N^*$ the constant $C_{p,D}$ defined in \eqref{DefCpD}. We denote also $C_{p,0}:=0$.\\

We let for $d\in\N^*$ :
\begin{equation}\label{DefEnsCouplageEnergieRen}
\LamN:=\left.\left\{{\bf D}\in\left\{\left\lceil\dfrac{d}{{N_0}}\right\rceil;\left\lfloor\dfrac{d}{{N_0}}\right\rfloor\right\}^{N_0}\,\right|\,\sum_{k=1}^{N_0} D_k=d\right\},
\end{equation}
\begin{equation}\label{CouplageEnergieRen}
\Wmin_{d,\O}=\Wmin_{d}:=\min_{{\bf D}\in\LamN}\left\{W^{\rm macro}\pD+\sum_{k=1}^{N_0}C_{p_k,D_k}+\tilde{V}[\zeta_\pD]\right\}
\end{equation}
where,  for $x\in\R$, $\lceil x\rceil$ is the ceiling of $x$, $\lfloor x\rfloor$ is the floor of $x$, $W^{\rm macro}(\cdot)$ is defined in Proposition \ref{Prop.EnergieRenDef} and $\tilde{V}[\zeta_\pD]$ is defined in Proposition \ref{Prop.Information.Zeta-a}.

 We now state an easy lemma whose proof is left to the reader.

\begin{lem}\label{LemLaisseLectTrucSimple}

Let  $d\in\N^*$ and ${\bf D}\in\LamN$. Then the following quantities are independent of ${\bf D}$:
\[
\L_1(d):=\dfrac{\pi}{2}\left[\left(\sum_{k=1}^{N_0} D_k^2\right)-d\right],
\]

\[
\L_2(d):=\Wmin_{d}+\dfrac{\pi}{2}\sum_{\substack{k=1\\\text{s.t. }D_k\geq1}}^{N_0}(D_k-D_k^2)\ln\left({D_k}\right).
\]

Moreover: $d\leq N_0\Longleftrightarrow\L_1(d)=0\Longleftrightarrow\L_2(d)=\Wmin_{d}$.
\end{lem}
\begin{notation}\label{NotL0}
We let $\L_1(0)=\L_2(0)=0$.
\end{notation}
The main result of this section is the following proposition.
\begin{prop}\label{Prop.BorneSupSimple}Assume that $\h=\mathcal{O}(|\ln\v|)$, $\h\to+\infty$,
\begin{equation}\label{HypLambdaDeltaConstrFoncTest}
\lambda^{1/4}|\ln\v|\to0\text{ and }{\delta\sqrt \h\to0}
\end{equation}
and  assume that Hypothesis \eqref{NonDegHyp} holds.

Let $d\in\N^*$ and  let ${\bf D}\in\LamN$ be a minimizer of the  minimizing problem \eqref{CouplageEnergieRen}. 

For $0<\v<1$, there exists $(v_\v,A_\v)\in\H$ which is in the Coulomb gauge with $d$ vortices of degree $1$ s.t.
\begin{eqnarray}\label{ExactExpEnerg}
\F(v_\v,A_\v)=\h^2 \Jo+d\Pic\left[-\h+\HoC\right]+\L_1(d)\ln\h+\L_2(d)+o(1)
\end{eqnarray}
with $\Pic:=2\pi\|\xi_0\|_{L^\infty(\O)}$ and 
\begin{equation}\label{DefH0c1}
\HoC:=\dfrac{b^2|\ln\v|+(1-b^2)|\ln(\lambda\delta)|}{2\|\xi_0\|_{L^\infty(\O)}}+\tilde\gamma_{b,\o}
\end{equation}
where 
\begin{equation}\label{DefGammaBo}
\tilde\gamma_{b,\o}:=\dfrac{\di\min_\o W^{\rm micro}+b^2[\gamma+\pi \ln b ]}{2\pi\|\xi_0\|_{L^\infty(\O)}},
\end{equation}
{ $\gamma$ is a universal constant defined in Lemma IX.1 \cite{BBH}} and $W^{\rm micro}$ is defined in Section \ref{SecMicrRenEn}.
\end{prop}
Proposition \ref{Prop.BorneSupSimple} is proved in Appendix \ref{AppProofUpBound}.

\section{Tool box}\label{Sec.ToolBox}
The proof of the main theorems of this article is done in a classic way: by matching upper and lower bounds. A [sharp] upper bound is obtained by Proposition \ref{Prop.BorneSupSimple}. Getting a sharp lower bound is the most challenging part of the proof. It needs the proof of several facts related with the vorticity defects of a family of quasi-minimizers [quantization, localization, size ...]. 

In this section we present some technical and quite classical results adapted to our situation.  
\subsection{An $\eta$-ellipticity property}\label{Sec.EtaElli}
In this section  we focus on quasi-minimizers. 
We let $\h=\mathcal{O}(|\ln\v|)$ and we consider $\{(v_\v,A_\v)\,|\,0<\v<1\}$ be a  family of quasi-minimizers for $\F$, {\it i.e.}, 
\begin{equation}\label{QuasiMinDef}
\F(v_\v,A_\v)\leq\inf_\H\F+o(1).
\end{equation}
 We assume that for all $\v\in(0;1)$, $(v_\v,A_\v)$ is in the Coulomb gauge and that $v_\v\in H^1(\O,\C)$ is s.t. 
\begin{equation}\label{BoundGrpaLin}
 \|\n |v_\v|\|_{L^\infty(\O)}=\mathcal{O}(\v^{-1}).
\end{equation}
The major result of this section is a key tool in this article: an $\eta$ ellipticity property.
\begin{prop}\label{Prop.EtaEllpProp}
Let $\h=\mathcal{O}(|\ln\v|)$ and let $\{(v_\v,A_\v)\,|\,0<\v<1\}\subset\H$ be a family in the Coulomb gauge  satisfying  \eqref{QuasiMinDef} and \eqref{BoundGrpaLin}.

For $\eta\in(0,1)$ there exist $\v_\eta>0$  and $C_\eta>0$ [depending on the bound of $\v\|\n |v_\v|\|_{L^\infty(\O)}$] s.t. for $0<\v<\v_\eta$, if $z\in\O$ is s.t.
\[
b^2\int_{B(z,\sqrt{\v})\cap\O}|\n v_\v|^2+\dfrac{b^2}{\v^2}(1-|v_\v|^2)^2\leq {C_\eta}{}|\ln\v|,
\]
then $|v_\v(z)|>\eta$.
\end{prop}
Proposition \ref{Prop.EtaEllpProp} is proved in Appendix \ref{AppendixPruveEtaEllipt}.

By combining Proposition \ref{Prop.EtaEllpProp} with Theorem  \ref{ThmBorneDegréMinGlob} we get immediately a first step in the [macroscopic] localization of the vorticity defects. In order to apply Theorem  \ref{ThmBorneDegréMinGlob} we need assume 
\begin{equation}\label{MagneticIntenHyp}
\begin{cases}\text{$\lambda,\delta$ satisfy \eqref{CondOnLambdaDelta}, $\delta^2|\ln\v|\to0$, $\h\to\infty$}
\\\text{\eqref{BorneKMagn} holds for $\h$ with some $K\geq0$  independent of $\v$}
\end{cases}.
\end{equation}
\begin{cor}\label{Cor.DegNonNul}
Assume that $\lambda,\delta$ and $\h$ satisfy \eqref{MagneticIntenHyp} and let $\{(v_\v,A_\v)\,|\,0<\v<1\}\subset\H$ be s.t.  \eqref{QuasiMinDef} and \eqref{BoundGrpaLin} hold. There exist $0<\v_0\leq\v_K$ and $M\geq1$ s.t. for $0<\v<\v_0$, letting $\tilde\Lambda_\v:=\Lambda\cap\cup_{d_i\neq0}B(a_i,2\M_K|\ln\v|^{-s_0})$ where the $(a_i,d_i)$'s [depend on $\v$] are given by Proposition \ref{Prop.BorneInfLocaliseeSandSerf} and $\v_K\&\M_K\&s_0$ are given by Theorem \ref{ThmBorneDegréMinGlob}, we have
\[
\di\{|v_\v|\leq1/2\}\subset\bigcup_{p\in\tilde\Lambda_\v}B(p,M|\ln\v|^{-\tilde s_0})\text{where }\tilde s_0:=\min\{s_0,10\}.
\]
\end{cor}
\begin{proof}
We argue by contradiction and we assume that there exist $\v=\v_n\downarrow0$ and a sequence $((v_\v,A_\v))_\v\subset\H$ s.t. \eqref{QuasiMinDef} and \eqref{BoundGrpaLin} hold and s.t. for all $n$ there exists 
\[
z_0=z_0^{n}\in\{|v|\leq1/2\}\setminus\bigcup_{p\in\tilde\Lambda_\v}B(p,n|\ln\v|^{-\tilde s_0}).
\]
Since  \eqref{QuasiMinDef} and \eqref{BoundGrpaLin} are gauge invariant we may assume that, for all $\v$, $(v_\v,A_\v)$ is in the Coulomb gauge. 

Let $\mathcal{B}:=\{(B(a_i,r_i),d_i)\,|\,i\in\J\}$ be  given by Proposition \ref{Prop.BorneInfLocaliseeSandSerf}. 
Write $B_i:=B(a_i,r_i)$ for $i\in\J$. Note that by Theorem \ref{ThmBorneDegréMinGlob}, from the  quasi-minimality of $(v_\v,A_\v)$, for $\v$ sufficiently small, we have $d_i\geq0$ for all $i$ and  $d:=\sum |d_i|=\sum d_i=\mathcal{O}(1)$. Up to pass to a subsequence, we may thus assume that $d$ is independent of $\v$.

 From the definition of $\tilde\Lambda_\v$, we have 
 \[
 \bigcup_{d_i>0}B_i\subset\bigcup_{p\in\tilde\Lambda}B(p,2\M_K|\ln\v|^{-s_0}).
 \]
 
 Note that from Theorem \ref{ThmBorneDegréMinGlob} we have $\F(v_\v,0)=\mathcal{O}(|\ln\v|^2)$. Then we may use Proposition \ref{Prop.BorneInfLocaliseeSandSerf} for the configuration $(v_\v,0)\in\H$ to get a covering $\cup_{i\in\tilde\J} \tilde B_i$ of $\{|v_\v|<1-|\ln\v|^{-2}\}$ with disjoint disks $\tilde B_i=B(\tilde a_i,\tilde r_i)$, $\sum\tilde r_i<|\ln\v|^{-10}$. 
 
 Therefore there is $\rho\in[2\M_K|\ln\v|^{-\tilde s_0};(2\M_K+6)|\ln\v|^{-\tilde s_0}]$ s.t. 
 \[
\left[\bigcup_{p\in\tilde\Lambda_\v}\p B(p,\rho)\right]\cap\left[\bigcup_{i\in\J}  B_i\cup\bigcup_{i\in\tilde\J} \tilde B_i\right]=\emptyset.
 \]
 In particular $|v_\v|\geq 1-|\ln\v|^{-2}$ on $\bigcup_{p\in\tilde\Lambda_\v}\p B(p,\rho)$. Thus, writing $\tilde d_i:=\deg_{\p \tilde B_i}(v_\v)$ when $\tilde B_i\subset\O$, we get for $p\in\tilde\Lambda_\v$ 
 \[
 \sum_{\tilde B_i\subset B(p,\rho)}|\tilde d_i|\geq\left|\sum_{\tilde B_i\subset B(p,\rho)}\tilde d_i\right|=\deg_{\p B(p,\rho)}(v_\v)=\sum_{ B_i\subset B(p,\rho)}d_i.
 \]

 Note that for  sufficiently large $n$ we have $B(z_0,\sqrt\v)\cap\bigcup_{p\in\tilde\Lambda_\v} B(p,\rho)=\emptyset$. 
 
On the other hand, since $\sum \tilde r_i<|\ln\v|^{-10}$, we have for $\tilde B_i\subset\O$
\begin{equation}\nonumber
F(v,\tilde B_i)\geq \pi b^2|\tilde d_i|(|\ln\v|-C\ln|\ln\v|).
\end{equation}
Using Proposition \ref{Prop.EtaEllpProp} we obtain 
\begin{equation}\label{ContraRER}
F(v)\geq (\pi b^2 d+C_{1/2})|\ln\v|-\mathcal{O}(\ln|\ln\v|)
\end{equation}
where $C_{1/2}>0$ is given by Proposition \ref{Prop.EtaEllpProp} with $\eta=1/2$.  Estimate \eqref{ContraRER} is in contradiction with \eqref{NiceDecSharpBorneSupBisso}.

\end{proof}

\subsection{Construction of the $\v^s$-bad discs}\label{ConstructionEpsBadDis}

As in the previous section we assume that $\lambda,\delta$ and $\h$ satisfy \eqref{MagneticIntenHyp}. In this section we establish the existence of {\it $\v^s$-bad discs associated to a quasi-minimizing sequence}. The construction of the bad discs requires the hypotheses: $|v_\v|\in W^{2,1}(\O)$. \\

An $\v^s$-bad discs family associated to a familly $\{(v_\v,A_\v)\,|\,0<\v<1\}\subset\H$ consists in sets of discs that have small diameters [a roots of $\v$] s.t. for fix $\v$ the discs are "well separated", the union of the discs is a covering of  $\{|v|\leq1/2\}$ and each "heart" of a disc intersects $\{|v|\leq1/2\}$. Such sets of discs give thus a nice visualization of $\{|v|\leq1/2\}$.

In the next section [Section \ref{Sect.ShapInfo}], adding an extra hypothesis on $\lambda,\delta$ and $\h$ we get some informations in terms of location and quantification of the $\v^s$-bad discs.

\begin{prop}\label{Prop.ConstrEpsMauvDisk}
Assume that $\lambda,\delta$ and $\h$ satisfy \eqref{MagneticIntenHyp}. There exists $M_0\in\N^*$ s.t. for $\mu\in(0,1/2)$, if $\{(v_\v,A_\v)\,|\,0<\v<1\}$ is in the Coulomb gauge and agrees \eqref{HypGlobalSurQuasiMin}$\&$\eqref{QuasiMinDef}, then there exist $\v_\mu>0$  and $C_\mu\geq1$ [independent of $\v$] s.t. for $0<\v<\v_\mu$, there is $J_\mu=J_{\mu,\v}\subset\{1,...,M_0\}$ [possibly empty] s.t. if $J_\mu=\emptyset$ then $|v|>1/2$ in $\O$ and if $J_\mu\neq\emptyset$ then there are  $\{z_i\,|\,i\in J_\mu\}\subset\O$, a set of mutually distinct points, and $\rad\in[\v^\mu,\v^{\mu_*}]$ with $\mu_*:=2^{-L_0^2}\mu$ verifying: 
\begin{enumerate}
\item\label{Prop.ConstrEpsMauvDisk1} $|z_i-z_j|\geq \rad^{3/4}$ for $i,j\in J_\mu$, $i\neq j$,
\item\label{Prop.ConstrEpsMauvDisk2} $\{|v_\v|\leq1/2\}\subset\cup_{ J_\mu}B(z_i,\rad)\subset\O$ and, for  $i\in J_\mu$, $B(z_i,\rad/4)\cap\{|v_\v|\leq1/2\}\neq\emptyset$,
\item\label{Prop.ConstrEpsMauvDisk3} For $i\in J_\mu$ we have $\di\rad\int_{\p B(z_i,\rad)}|\n v_\v|^2+\dfrac{1}{2\v^2}(1-|v_\v|^2)^2\leq C_\mu$ and $|v|\geq1-|\ln\v|^{-2}$ on $\p B(z_i,\rad)$.
\end{enumerate}
\end{prop}
Proposition \ref{Prop.ConstrEpsMauvDisk} is proved in Appendix \ref{SectAppenPreuveConstructionPetitDisque}. We have the following standard  estimate.

\begin{prop}\label{Prop.ProprieteEpsMauvDisk}
Assume \eqref{MagneticIntenHyp} and let  $\{(v_\v,A_\v)\,|\,0<\v<1\}$ be as in Proposition \ref{Prop.ConstrEpsMauvDisk}. Fix  $\mu\in(0,1/2)$ and let  $\v_\mu$, $C_\mu$ be given by Proposition \ref{Prop.ConstrEpsMauvDisk}. For $0<\v<\v_\mu$ we consider $J_\mu$, $\{z_i\,|\,i\in J_\mu\}\subset\O$ and $\rad$ obtained in Proposition \ref{Prop.ConstrEpsMauvDisk}. We denote $ d_i:=\deg_{\p B(z_i,\rad)}(v_\v)$. 

There exists $c_{\mu,b}\geq1$ independent of $\v$ s.t. for $\v<\v_\mu$ we have
\begin{equation}\label{BorneDegré}
| d_i|\leq 4\sqrt {C_\mu},
\end{equation}
\begin{equation}\label{BorneInfEn}
\dfrac{1}{2}\int_{B(z_i,\rad)}|\n v_\v|^2+\dfrac{b^2}{2\v^2}(1-|v_\v|^2)^2\geq\pi| d_i|\ln\left(\dfrac{r}{\v}\right)-c_{\mu,b}
\end{equation}
and then
\begin{equation}\label{BorneInfEnWeight}
F(v_\v,B(z_i,\rad))\geq\pi| d_i|\inf_{B(z_i,\rad)}\alpha\left[\ln\left(\dfrac{r}{\v}\right)-c_{\mu,b}\right]\geq\pi\inf_{B(z_i,\rad)}\alpha\,| d_i|[(1-\mu)\ln\v -c_{\mu,b}].
\end{equation}
Moreover there is $0<\tilde\v_\mu\leq\v_\mu$ s.t. for $0<\v<\tilde\v_\mu$ we have
\begin{equation}\label{DegNonNulEpsS}
d_i\neq0\text{ for all }i
\end{equation}
and
\begin{equation}\label{BorneTotaleSommeDeg}
\sum_{i\in J_\mu}|d_i|\leq\mathcal{D}_{K,b}:=\dfrac{3\M_K}{b^2}
\end{equation}
\end{prop}
\begin{proof}
It is classical to get \eqref{BorneDegré} from Proposition \ref{Prop.ConstrEpsMauvDisk}.\ref{Prop.ConstrEpsMauvDisk3} and the  Cauchy Schwartz inequality. Estimate \eqref{BorneInfEn} follows from Proposition \ref{Prop.ConstrEpsMauvDisk} $\&$ Lemma VI.1 in \cite{AB1} and \eqref{BorneInfEnWeight} is a consequence of  \eqref{BorneInfEn}.

The proof of \eqref{DegNonNulEpsS} is done arguing by contradiction with the construction of a comparaison function $\tilde v:=\begin{cases}v&\text{in }\O\setminus B(z_{i_0},\rad)\\\tilde\rho\e^{\imath\tilde\phi}&\text{in }B(z_{i_0},\rad)\end{cases}$ s.t. $\tilde v\in H^1(\O,\C)$ and  $F(\tilde v,B(z_{i_0},\rad))=\mathcal{O}(1)$ where we assumed $d_{i_0}=0$.

Since $(v,A)$ is a quasi-minimizer of $\F$ we have $\F(v,A)\leq\F(\tilde v,A)+o(1)$. 

On the other hand, by direct calculations $\F(v,A)-\F(\tilde v,A)=F(v,B(z_{i_0},\rad))-F(\tilde v,B(z_{i_0},\rad))+o(1)$. Consequently $F(v,B(z_{i_0},\rad))=\mathcal{O}(1)$ which is in contradiction with $F(v,B(z_{i_0},\rad))\geq C_{1/2}|\ln\v|$ [given by Proposition \ref{Prop.EtaEllpProp}] for small $\v$.

We now prove \eqref{BorneTotaleSommeDeg}. From \eqref{BorneInfEnWeight} we have $\sum_{J_\mu}|d_i|\left[\pi (1-\mu)|\ln\v|-c_{\mu,b}\right]\leq\dfrac{\M_K|\ln\v|}{b^2}$. Since $\mu\in(0,1/2)$, the last estimate gives the result for $\v>0$ sufficiently small.
\end{proof}
\subsection{Lower bounds in perforated disks}\label{Sec.StrongEffectDilution}
The goal of this section is to get lower bounds for $\frac{1}{2}\int_{\dom}\alpha|\n v|^2$ where $\dom$ is a perforated disk s.t. $\dom\subset\O$ and $|v|\geq1/2$ in $\dom$.

The starting point of  the argument is an estimate on circles. Let $\tilde b\in(0,1)$, $\beta\in L^\infty((0,2\pi),[\tilde b,1])$. With Lemma D.7 in \cite{Publi4}, for  $\varphi\in H^1((0,2\pi),\R)$ s.t. $\varphi(2\pi)-\varphi(0)=2\pi$, we have the following lower bound:
\begin{equation}\label{EstimaBasiCercle}
\dfrac{1}{2}\int_0^{2\pi}\beta|\p_\theta\varphi|^2\geq\dfrac{2\pi^2}{\di\int_0^{2\pi}\dfrac{1}{\beta}}.
\end{equation}
In order to use \eqref{EstimaBasiCercle} we need to do a preliminary analysis.\\

For $\alpha=U_\v^2\in L^\infty(\O,[b^2,1])$, using Lemma E.1 in \cite{Publi4}, we have the existence of $C\geq1$ [independent of $\v$]  s.t. 
\begin{equation}\label{CqFondDilution}
\left\{\begin{array}{l}\text{For almost all $s\geq\delta/3$, letting $\mathscr{C}_s$ be a circle with radius $s$,}\\ \text{we have $\int_{\mathscr{C}_s\cap\O}{(1-\alpha)}\leq C\lambda s$.}\end{array}\right.
\end{equation}
From now on, in all this section, we consider a sequence $\v=\v_n\downarrow0$, $\lambda,\delta,\h$ and $((v_\v,A_\v))_\v\subset\H$ satisfying the hypotheses of Proposition \ref{Prop.ConstrEpsMauvDisk} [namely \eqref{HypGlobalSurQuasiMin}, \eqref{QuasiMinDef} and \eqref{MagneticIntenHyp}]. We drop the subscript $\v$ writing $(v,A)$ instead of $(v_\v,A_\v)$

Recall that $\eta_\O$ is defined in \eqref{DefEtaO} and consider
\begin{equation}\label{HypOxRr}
\text{$x_\v\in\O$ and $0<r=r_\v<R=R_\v<\eta_\O$ s.t. $\dist(x_\v,\p\O)>\eta_\O>0$.}
\end{equation}
We then denote $\Ring:=B(x_\v,R)\setminus\overline{B(x_\v,r)}\subset\O$.\\

 Assume $|v|\geq1/2$ in $\Ring$ and let $d:=\deg_{\Ring}(v)$. From the proof of Proposition \ref{Prop.ConstrEpsMauvDisk} [see \eqref{BorneValeurMesure} in Appendix \ref{SectAppenPreuveConstructionPetitDisque}], there exists $1/2<t_\v<1$, $t_\v=1+o(1)$ s.t. $t_\v\in{\rm Im}(|v|)\cap[1-2/|\ln\v|;1-1/|\ln\v|]$ and
\begin{equation}\label{DefTeps}
\left\{\begin{array}{c}\text{$V(t_\v):=\{|v|=t_\v\}$ is a finite union of Jordan curves included in  $\O$ and } \\
\text{ of simple curves whose endpoints are on $\p\O$ and  $\Haus^1[V(t_\v)]=o(1)$.}
\end{array}\right.
\end{equation}
and since $\Haus^2(\{|v|\leq t_\v\})=o(1)$ we then have
\begin{equation}\label{PropIntDefTeps}
\left\{\begin{array}{c}
\text{if $U$ is a connected component of $\{|v|\leq t_\v\}$ s.t. $\overline U\subset\O$ then there is $\Gamma$,}
\\\text{a connected component of $V(t_\v)$, which is a Jordan curve s.t. $U\subset{\rm int }(\Gamma)$.}
\end{array}\right.
\end{equation}
\begin{remark}\label{RmMontargis}
Since $\Haus^1[V(t_\v)]=o(1)$, for sufficiently small $\v$, if $\Gamma$ [resp. $U$] is a connected component of $V(t_\v)$ [resp. $\{|v|\leq t_\v\}$] which intersects $\Ring$ then $\Gamma$ is a Jordan curve [resp. $\p U$ is a union of connected components of $V(t_\v)$].
\end{remark}
We have the following lemma:
\begin{lem}\label{Lem.BorneLongueEnradian}
Assume $x_\v,r,R$ satisfy \eqref{HypOxRr} and we assume $|v|\geq1/2$ in $\Ring$. Then, for $s\in(r,R)$, letting
\[
K_{s}:=\{\theta\in[0,2\pi)\,|\,|v(x_\v+s\e^{\imath\theta})|\leq t_\v\}
\]
we have
\begin{equation}\nonumber
\Haus^1(K_s)\leq \pi\dfrac{\Haus^1[V(t_\v)]}{s}.
\end{equation}
\end{lem}
\begin{proof}
Let $s\in(r,R)$ be s.t. $\Haus^1(K_s)>0$ and denote $\K_s:=\{x_\v+s\e^{\imath\theta}\,|\,\theta\in K_s\}\subset\p B(x_\v,s)$. Then $\Haus^1(\K_s)=s\Haus^1(K_s)$. 

On the one hand, letting $\mathcal{V}_\Ring(t_\v)$ be the union of the connected components of $\{|v|\leq t_\v\}$ which intersect $\Ring$, we have $\K_s=\mathcal{V}_\Ring(t_\v)\cap\p B(x_\v,s)$.

On the other hand, by Remark \ref{RmMontargis}, $\p\mathcal{V}_\Ring(t_\v)$ is a union of connected components of $V(t_\v)$ which are Jordan curves. Among these Jordan curves, we may select the maximal curves w.r.t. the inclusion of their interior. We denote these maximal curves by  $\Gamma_1,...,\Gamma_N$ and we let for $i\in\{1,...,N\}$, $\mathcal{V}_i:=\overline{{\rm int}(\Gamma_i)}$. We then obtain $\mathcal{V}_\Ring(t_\v)\subset\cup_{i=1}^N\mathcal{V}_i$ and thus  $\K_s\subset\cup_{i=1}^N[\p B(x_\v,s)\cap\mathcal{V}_i]$.
 
For $i\in\{1,...,N\}$, we fix $x_i\in\mathcal{V}_i$ and we define the disk $B_i:=\overline{B(x_i,{\rm diam}(\mathcal{V}_i))}$. It is clear that $\mathcal{V}_i\subset B_i$ . Consequently
\[
\Haus^1[\p B(x_\v,s)\cap\mathcal{V}_i]\leq\Haus^1[\p B(x_\v,s)\cap B_i]\leq2\pi\,{\rm diam}(\mathcal{V}_i).
\]
We claim that $2{\rm diam}(\mathcal{V}_i)\leq\Haus^1(\Gamma_i)$. Since the curves $\Gamma_i$ are  pairwise disjoint, we have $\sum_{i=1}^N\Haus^1(\Gamma_i)\leq\Haus^1[V(t_\v)]$.

We may now conclude: 
\[
s\Haus^1(K_s)=\Haus^1(\K_s)\leq\sum_{i=1}^N\Haus^1[\p B(x_\v,s)\cap\mathcal{V}_i]\leq\pi\sum_{i=1}^N2{\rm diam}(\mathcal{V}_i)\leq\pi\Haus^1[V(t_\v)].
\]

\end{proof}
The next proposition is one of the major use of the dilution [$\lambda\to0$].
\begin{prop}\label{Prop.ComparaisonAnneau}
Let  $x_\v,r,R$ satisfying \eqref{HypOxRr} and assume  $|v|\geq1/2$ in $\Ring$. We write $d:=\deg_{\Ring}(v)$ and, in $\Ring$, we let $w:=v/|v|\,\&\,\rho:=|v|$.
\begin{enumerate}
\item\label{Prop.ComparaisonAnneau1} If $r\geq\delta/3$ and if $\Haus^1[V(t_\v)]/r+(1-t^2_\v)+\lambda=o[\ln(R/r)]$ then
\[
\dfrac{1}{2}\int_\Ring\alpha|\n v|^2\geq\dfrac{1}{2}\int_\Ring\alpha\rho^2|\n w|^2\geq \pi d^2\left[\ln\left(\dfrac{R}{r}\right)-o(1)\right].
\]
\item\label{Prop.ComparaisonAnneau2} If $r=o(1)$ and if $\Haus^1[V(t_\v)]/r+(1-t^2_\v)=o[\ln(R/r)]$ then
 \[
\dfrac{1}{2}\int_\Ring|\n v|^2\geq\dfrac{1}{2}\int_\Ring\rho^2|\n w|^2\geq \pi d^2\left[\ln\left(\dfrac{R}{r}\right)-o(1)\right].
\]
\end{enumerate}
\end{prop}
\begin{proof}We prove the first assertion. We claim that, up to replace $v$ with $\underline{v}$, we may assume $|v|\leq1$ in $\O$. Moreover, if $d=0$ then there is nothing to prove. We then assume  $d\neq0$.

We write $v=\rho\e^{\imath d\varphi}$ where $\varphi$  is locally defined and its gradient is globally defined. Letting $x_\v+\R^+:=\{x_\v+s\,|\,s\geq0\}$, we may assume $\varphi\in H^1(\Ring\setminus(x_\v+\R^+),\R)$. For $s\in(r,R)$, we let  $\varphi_s(\theta)=\varphi(x_\v+s\e^{\imath\theta})$, $\rho_s(\theta)=|v(x_\v+s\e^{\imath\theta})|$ and $\alpha_s(\theta)=\alpha(x_\v+s\e^{\imath\theta})$. Then $\varphi_s\in H^1((0,2\pi),\R)$ is s.t. $\varphi_s(2\pi)-\varphi_s(0)=2\pi$ and we immediately get
\[
\dfrac{1}{2}\int_\Ring\alpha\rho^2|\n w|^2\geq \dfrac{d^2}{2}\int_r^R\dfrac{{\rm d}s}{s}\int_0^{2\pi}\alpha_s\rho_s^2|\p_\theta\varphi_s|^2{\rm d}\theta.
\]
From \eqref{EstimaBasiCercle} with $\beta:=\alpha_s\rho_s^2$ we get 
\[
\displaystyle
\dfrac{1}{2}\int_0^{2\pi}\alpha_s\rho_s^2|\p_\theta\varphi_s|^2\geq\dfrac{2\pi^2}{\di\int_0^{2\pi}\dfrac{1}{\alpha_s\rho_s^2}}.
\]
Since $b^2/4\leq \alpha_s\rho_s^2\leq 1$ we have
\[
0\leq\left(\int_0^{2\pi}\dfrac{1}{\alpha_s\rho_s^2}\right)-2\pi=\int_0^{2\pi}\dfrac{1-\alpha_s\rho_s^2}{\alpha_s\rho_s^2}
\leq\dfrac{4}{b^2}\left(\int_0^{2\pi}{1-\rho_s^2}+\int_0^{2\pi}{1-\alpha_s}\right).
\]

On the one hand, from Lemma \ref{Lem.BorneLongueEnradian} we have
\[
\int_0^{2\pi}{1-\rho_s^2}\leq\Haus^1(K_s)+\left[2\pi-\Haus^1(K_s)\right](1-t_\v^2)\leq\dfrac{\pi\Haus^1[V(t_\v)]}{s}+2\pi(1-t^2_\v).
\]
On the other hand, using \eqref{CqFondDilution}, there is $C\geq1$ [independent of $\v$] s.t.$\displaystyle
\int_0^{2\pi}{1-\alpha_s}\leq C\lambda$. Then
\[
\int_0^{2\pi}\dfrac{1}{\alpha_s\rho_s^2}\leq2\pi+\dfrac{4}{b^2}\left[\dfrac{\pi\Haus^1[V(t_\v)]}{s}+2\pi(1-t^2_\v)+C\lambda\right].
\]
We thus get
\begin{eqnarray*}
\dfrac{1}{2}\int_\Ring\alpha\rho^2|\n w|^2&\geq& {d^2}\int_r^R\dfrac{{\rm d}s}{s}\dfrac{2\pi^2}{2\pi-\dfrac{4}{b^2}\left[\pi\Haus^1[V(t_\v)]/s+2\pi(1-t^2_\v)+C\lambda\right]}
\\&=&\pi d^2\left[\ln\left(\dfrac{R}{r}\right)+o(1)\right].
\end{eqnarray*}
The second assertion is obtain exactly in the same way than the first one. Indeed, since $\alpha$ plays no role in the statement, we may use the same argumentation with $\lambda=0$ and $\delta>0$ an arbitrary small number.
\end{proof}
We now state the reformulation of Proposition \ref{Prop.ComparaisonAnneau} by replacing the annular $\Ring$ with a perforated disk.
\begin{cor}\label{Cor.BorneInfProcheIncl}
Let $D_0\in\N^*$ be independent of $\v$, $0<r=r_\v<R=R_\v$ be s.t. $r=o(R)$, $N=N_\v\in\N^*$ be s.t. $ N\leq D_0$ and $z_1=z_1^\v,...,z_N=z_N^\v$ be s.t. $|z_i-z_j|\geq8 r$ for $i\neq j$.

Let $y=y_\v\in\O$ and assume $z_1,...,z_N\in B(y,R)\subset B(y,4R)\subset B(y,\eta_\O)\subset\O$. We let $\dom:=B(y,2R)\setminus\cup_{i=1}^N\overline{B(z_i,r)}$.

Assume $\rho=|v|\geq1/2$ in $\dom$. For $i\in\{1,...,N\}$, we let $d_i:=\deg_{\p B(z_i,r)}(v)$.  We also assume $d_i>0$ for all $i\in\{1,...,N\}$ and $\sum_{i=1}^Nd_i\leq D_0$. Write $v=\rho w$ in $\dom$.

Then there exists $C_0>0$ depending only on $D_0$ s.t. :
\begin{enumerate}
\item\label{Cor.BorneInfProcheIncl1} If $r\geq \delta/3$  and  $\Haus^1[V(t_\v)]/r+(1-t^2_\v)+\lambda=o[\ln(R/r)]$ then,  for sufficiently small $\v$, we have
\[
\dfrac{1}{2}\int_\dom\alpha|\n v|^2\geq \dfrac{1}{2}\int_\dom\alpha\rho^2|\n w|^2\geq\pi\sum_{i=1}^N d_i^2\ln(R/r)-C_0.
\]
\item\label{Cor.BorneInfProcheIncl2}  If  $\Haus^1[V(t_\v)]/r+(1-t^2_\v)=o[\ln(R/r)]$ then,  for sufficiently small $\v$, we have
\[
\dfrac{1}{2}\int_\dom|\n v|^2\geq\dfrac{1}{2}\int_\dom\rho^2|\n w|^2\geq\pi\sum_{i=1}^N d_i^2\ln(R/r)-C_0.
\]
\end{enumerate}
\end{cor}
\begin{proof}We claim that, up to replace $v$ with $\underline{v}$, we may assume $|v|\leq1$ in $\O$.

We first proceed to a scaling with the conformal mapping:
\[
\begin{array}{cccc}
\Phi:&B(y,4R)&\to&B(0,4)\\&x&\mapsto&\dfrac{x-y}{R}
\end{array}.
\]
We then let $\hat z_i:=\Phi(z_i)$, $\hat r:= r/R$, $\hat\dom:=\Phi[\dom]=B(0,2)\setminus\cup_{i=1}^N\overline{B(\hat z_i,\hat r)}$, $\hat\alpha:=\alpha\circ\Phi^{-1}$ and $\hat v:=v\circ\Phi^{-1}$.

If $N=1$ or $N\geq2$ and $|\hat z_i-\hat z_j|\geq 4\times10^{-2D_0}$ for $i\neq j$ then, letting $\tilde\O:=B(0,4)$, $\eta_{\tilde\O}=10^{-1}$, we may apply Proposition \ref{Prop.ComparaisonAnneau}.\ref{Prop.ComparaisonAnneau1}
\begin{eqnarray*}
\dfrac{1}{2}\int_\dom\alpha|\n v|^2=\dfrac{1}{2}\int_{\hat\dom}\hat\alpha|\n \hat v|^2&\geq&\sum_{i=1}^N\dfrac{1}{2}\int_{B(\hat z_i,2\times10^{-2D_0})\setminus\overline{B(\hat z_i,\hat r)}}\hat\alpha|\n\hat v|^2
\\&\geq&\pi\sum_{i=1}^N d_i^2\left(|\ln( R/ r)|-|\ln(2\times10^{-2D_0})|\right)+o(1).
\end{eqnarray*}
This estimate is the desired result with  $C_0=\pi D_0^2|\ln(2\times10^{-2D_0})|+1$.

If we are not in the previous case, {\it i.e.} $N\geq 2$ and there exists $i\neq j$ s.t. $|\hat z_i-\hat z_j|< 4\times10^{-2D_0}$, then we apply the separation process presented  Appendix C [Section C.3.1] in \cite{Publi4} to the domain $\hat\dom$ with $\eta_{\rm stop}:=10^{-2D_0}$.

The key ingredient in the separation process is a variant of Theorem IV.1 in \cite{BBH} [stated with $\num=9$, the general case $\num\in\N\setminus\{0,1\}$ is left to the reader]:
\begin{lem}\label{Lem.Separation}
Let $N\geq2$, $\num\in\N\setminus\{0,1\}$, $x_1,...,x_N\in\R^2$ and $\eta>0$. There are $\kappa\in\{\num^0,...,\num^{N-1}\}$ and $\emptyset\neq J\subset\{1,...,N\}$ s.t.
\[
\cup_{i=1}^N B(x_i,\eta)\subset\cup_{i\in J}B(x_i,\kappa\eta)\text{ and }|x_i-x_j|\geq(\num-1)\kappa\eta\text{ for }i,j\in J,\,i\neq j.
\]
\end{lem}
The separation process is an iterative selection of points in $\{\hat z_1,...,\hat z_N\}$ associated to the construction of a good radius.

We initialize the process by letting  $\eta_0:=\hat r$, $M_0:=N$ and $J_0=\{1,...,M_0\}$.

For $k\geq1$ [where $k$ is the index in the iterative process] we construct a set $\emptyset\neq J_k\subsetneq J_{k-1}$, $M_k:={\rm Card}(J_k)$ and 3 numbers
\begin{center}
$\kappa_k\in\{9^1,...,9^{M_{k-1}-1}\}$, $\eta_k':=\dfrac{1}{4}\di\min_{\substack{i,j\in J_{k-1}\\i\neq j}}|\hat z_i-\hat z_j|$ and $\eta_k:=2\kappa_k\eta_k'$.
\end{center}
These objects are obtained with Lemma \ref{Lem.Separation} with $P=9$, $N=M_{k-1}={\rm Card}(J_{k-1})$, $\{x_1,...,x_N\}=\{z_i\,|\,i\in J_{k-1}\}$, $J=J_k$, $\eta=\eta_k$, $\kappa=\kappa_k$
The process stops at the end of Step  $K_0\geq1$ if $M_{K_0}=1$ or $M_{K_0}\geq 2$ and $\di\min_{\substack{i,j\in J_{K_0}\\i\neq j}}|\hat z_i-\hat z_j|>4\eta_{\rm stop}$.

By construction, we have for $1\leq k\leq K_0$, $\emptyset\neq J_k\subsetneq J_{k-1}$ and $\eta_{k-1}\leq\eta_k'<\eta_k$. In particular, since ${\rm Card}( J_0)\leq D_0$, we get $K_0\leq D_0-1$.

By definition, for $k\in\{1,...,K_0\}$ we have $2\cdot 9\eta_k'\leq\eta_k\leq9^{D_0}\eta_k'$. We let
\[
\eta_0:=\begin{cases}9^{D_0}\cdot\eta_{\rm stop}&\text{if }M_{K_0}=1\\\min\{9^{D_0}\cdot\eta_{\rm stop},\dfrac{1}{4}\di\min_{\substack{i,j\in J_{K_0}\\i\neq j}}|\hat z_i-\hat z_j|\}&\text{if }M_{K_0}\geq2\end{cases}
\]
 and then $\eta_0\geq\eta_{\rm stop}=10^{-2D_0}$. For $k\in\{0,...,{K_0}-1\}$ and $i\in J_k$ we denote $\Ring_{i,k}:=B(\hat z_i,\eta_{k+1}')\setminus\overline{B(\hat z_i,\eta_{k})}$, and, for $i\in J_{K_0}$, $\Ring_i:=B(\hat z_i,\eta_0)\setminus\overline{B(\hat z_i,\eta_{K_0})}$. By construction, the previous rings are pairwise disjoint. From Proposition \ref{Prop.ComparaisonAnneau}.\ref{Prop.ComparaisonAnneau1} we have for $k\in\{0,...,{K_0}-1\}$ and $i\in J_k$ :
\begin{eqnarray*}
\dfrac{1}{2}\int_{\Ring_{i,k}}\hat\alpha|\n\hat v|^2&\geq&\pi \deg_{\Ring_{i,k}}(\hat v)^2\left[\ln(\eta_{k+1}/\eta_{k})-\ln(9^{D_0})\right]-o(1)
\\&\geq&\pi \sum_{\hat z_j \in B(\hat z_i,\eta_{k+1}')}d_j^2\ln(\eta_{k+1}/\eta_{k})-\pi D_0^2\ln(9^{D_0})-o(1).
\end{eqnarray*}
And for $i\in J_{K_0}$:
\begin{eqnarray*}
\dfrac{1}{2}\int_{\Ring_{i}}\hat\alpha|\n\hat v|^2&\geq&\pi \deg_{\Ring_{i}}(\hat v)^2\ln(\eta_0/\eta_{K_0})-o(1)
\\&\geq&\pi \sum_{\hat z_j \in B(\hat z_i,\eta_0)}d_j^2\ln(\eta/\eta_{K_0})-o(1).
\end{eqnarray*}
By summing the previous lower bound we get the result. As for Proposition \ref{Prop.ComparaisonAnneau}, the second assertion is obtained in a similar way than the first assertion.
\end{proof}
\subsection{Lower bounds in a perforated domain}
In this section we state a lower bound for a weighted Dirichlet energy in the domain $\O$ perforated by small [but not too small] disks. The philosophy of this lower bound is that in the case which interest us we may ignore the weight if the perforations are not too small ; it is an effect of the dilution $\lambda\to0$.
\begin{prop}\label{VeryNiceCor}
Let $\beta\in(0,1)$, $(\tilde\alpha_n)_n\subset L^\infty(\O,[\beta^2,1])$ be s.t. 
\[
K_n:=\sqrt{\int_\O(1-\tilde\alpha_n)^2}\to0.
\]
Let $N\in\N^*$ and $\zd=\zd^{(n)}\subset\Ostar\times\Z^N$ be s.t. ${\bf d}$ is independant of $n$.  We denote $\tae:=\min_i\dist(z_i,\p\O)$.

Assume the existence of $\Rad>0$  s.t. $\Rad=o(1)$, \eqref{HypRayClass} holds and s.t. there is $C_1>0$ [independent of $n$] satisfying $\dfrac{\Rad|\ln\Rad|}{\tae}\leq C_1$. Write $\O_{\Rad}:=\O\setminus\cup\overline{B(z_i,\Rad)}$.

Let  $(u_n)_n\subset H^1(\O,\C)$  satisfying   $|u_n|\geq\dfrac{1}{2}$ in $\O_\Rad$ and $\deg_{\p B(z_i,\Rad)}(u_n)=d_i$ for all $i$.

Assume also
\[
L_n:=\sqrt{\int_{\O_\Rad}(1-|u_n|^2)^2}\to0.
\]

Then
\[
\int_{\O_\Rad}\tilde\alpha_n|\n u_n|^2\geq\int_{\O_\Rad}|\n \Pstar|^2-(4\beta^{-1}+3)\|\n\Pstar\|_{L^\infty(\O_\Rad)}\|\n\Pstar\|_{L^2(\O_\Rad)}\left(K_n+L_n\right)-\mathcal{O}(X)
\]
with $\Pstar$ is defined in Remark \ref{Remark.DefConjuHarmPhase} and $X$ is defined in \eqref{DefX}.
\end{prop}
Proposition \ref{VeryNiceCor} is proved Appendix \ref{Sec.PreuveVeryNiceCor}.
\section{Study of the $\v^s$-bad discs}\label{Sect.ShapInfo}

In this section, in addition to the assumption  \eqref{MagneticIntenHyp} on  $\lambda,\delta$ and $\h$, we assume  that \eqref{PutaindHypTech} holds.
This [technical] hypothesis \eqref{PutaindHypTech} is a little bit more restrictive than \eqref{HypLambdaDeltaConstrFoncTest} [$\delta\sqrt{\h}\to0$] used to get a nice upper bound.\\

Let $\v=\v_n\downarrow0$ and let $((v,A))_\v=((v_\v,A_\v))_\v$ be a sequence that agrees \eqref{HypGlobalSurQuasiMin} and \eqref{QuasiMinDef}. Let also $\mu\in(0,1/2)$.

Since \eqref{HypGlobalSurQuasiMin} and \eqref{QuasiMinDef} are gauge invariant we may assume that $(v,A)$ is in the Coulomb gauge.

The goal of this section is to prove that, for  sufficiently small $\v\&\mu$, if $J_\mu\neq\emptyset$ then $d_i=1$ $\&$ $\dist(z_i,\Lambda)\leq\ln(\h)/\sqrt{\h}$ $\&$ $z_i\in\o_\v$ for all $i\in J_\mu$ and for $i\neq j$, $|z_i-z_j|\geq\ln(\h)/\h$ with a "uniform" distribution of the $z_i$'s around $\Lambda$. 

With the notation of Proposition \ref{Prop.ConstrEpsMauvDisk} we let $\O_\rad:=\O\setminus\cup_{i\in J_\mu}\overline{B(z_i,\rad)}$ and $d:=\sum_{i\in J_\mu}|d_i|$.

In view of the goal of this section we may argue on subsequences. First note that from \eqref{DegNonNulEpsS} we have $d_i\neq0$ for all $i$. Up to pass to a subsequence, from \eqref{BorneTotaleSommeDeg}, we may assume that $J_\mu\neq\emptyset$ and independent of $\v$  as well as the $d_i$'s. 

Since we are interested here only on informations related with $|v|$ and the $d_i$'s, we may consider that $(v,A)$ is in the Coulomb gauge and we may also change the potential vector. Namely, we may assume that $A=\n^\bot\xi$ with $\xi=\xi_\v\in H^1_0\cap H^2(\O,\R)$ is the unique solution of \eqref{Eq.MinPb.Pot}. Note that  \eqref{QuasiMinDef} still holds.

Consequently, $\rot(A)\in H^1$ and then with \eqref{FullGLuAEq}$\&$\eqref{EstH3}: $\|\xi\|_{H^3(\O)}\leq C\|\rot(A_\v)\|_{H^1(\O)}\leq C|\ln\v|$.

From Proposition \ref{Docmpen} and letting $\zeta=\zeta_\v:=\xi-\h\xi_0$
 \begin{equation}\nonumber
\F(v,\n^\bot\xi)=\h^2\Jo+F(v)+2\pi\h\sum d_i\xi_0(z_i)+\tilde{V}_\zd(\zeta)+o(1).
\end{equation}

Proposition \ref{Prop.Information.Zeta-a} infers $\tilde{V}_\zd(\zeta)=\mathcal{O}(1)$. Consequently
 \begin{equation}\label{FullDecDiscqVal0-ApplicGrossiere}
\F(v,\n^\bot\xi)=\h^2\Jo+F(v)+2\pi\h\sum d_i\xi_0(z_i)+\mathcal{O}(1).
\end{equation}
In particular we have $\F(v,\n^\bot\xi)\leq\h^2\Jo+o(1)$, thus with \eqref{FullDecDiscqVal0-ApplicGrossiere} we get
\begin{equation}\label{FullDecDiscqVal0-PetitChamp}
F(v)\leq-2\pi\h\sum d_i\xi_0(z_i)+\mathcal{O}(1).
\end{equation}
From Corollary \ref{Cor.DegNonNul} and Propositions \ref{Prop.ConstrEpsMauvDisk}$\&$\ref{Prop.ProprieteEpsMauvDisk} we deduce $-\sum d_i\xi_0(z_i)=\|\xi_0\|_{L^\infty(\O)}\sum d_i+o(1)$ and we immediately obtained 
\begin{equation}\label{SumDegPos}
\sum d_i\geq0.
\end{equation}
On the other hand, from  Proposition \ref{Prop.BorneSupSimple}, we have
 \begin{equation}\label{BrneSup-Applic}
\F(v,\n^\bot\xi)\leq\h^2 \Jo+d\Pic\left[-\h+\HoC \right]+\L_1(d)\ln\h+\mathcal{O}(1).
\end{equation}
By combining \eqref{FullDecDiscqVal0-ApplicGrossiere} and \eqref{BrneSup-Applic} we get
 \begin{equation}\label{BrneSup-ApplicF(v)}
 F(v)\leq d\pi\left[b^2|\ln\v|+(1-b^2)|\ln(\lambda\delta)|\right]+\L_1(d)\ln\h+\mathcal{O}(1).
 \end{equation}
 In conclusion, from  \eqref{BorneInfEn} in conjunction with \eqref{BrneSup-ApplicF(v)} we obtain
  \begin{equation}\label{BrneSup-ApplicDirEn}
\dfrac{1}{2}\int_{\O_\rad}\alpha|\n v|^2\leq d\pi\left[b^2|\ln\rad|+(1-b^2)|\ln(\lambda\delta)|\right]+\L_1(d)\ln\h+\mathcal{O}(1).
 \end{equation}
 We first have the following proposition.
 \begin{prop}\label{PropDegPinning1}
Assume
 \begin{equation}\label{HypSurMu}
 0<\mu<\min\left\{\dfrac{1}{\mathcal{D}_{K,b}+1},\dfrac{1-b^2}{2(\mathcal{D}_{K,b}+1)}\right\}
 \end{equation}
 where $\mathcal{D}_{K,b}=\dfrac{3\M_K}{b^2}$ and $\M_K$ is as in Theorem \ref{ThmBorneDegréMinGlob}.
 
Then there exists $\tilde\v_\mu'>0$ s.t. for $0<\v< \tilde\v_\mu'$ if $J_\mu\neq\emptyset$ then
\begin{enumerate}
\item $d_i>0$ for all $i$,
\item\label{PropDegPinning1.2} $\dist(z_i,\o_\v)<\sqrt\v$.
\end{enumerate}
\end{prop}
\begin{proof}{\bf Step 1.  We prove that $d_i>0$ for all $i$}\\

We argue by contradiction and we assume the existence of an extraction still denoted by $\v=\v_n\downarrow0$ s.t. $J_-:=\{i\in J_\mu\,|\,d_{i}<0\}\neq\emptyset$ [from \eqref{DegNonNulEpsS}, for $0<\v<\tilde\v_\mu$, we have $d_i\neq0$ for all $i\in J_\mu$]. 

From \eqref{SumDegPos} we thus obtain: $\sum_{i\in J_\mu\setminus J_-}d_i\geq d+1$. Then, with the help of \eqref{BorneInfEnWeight}, we obtain
\[
F(v)\geq b^2(1-\mu)\pi|\ln\v|\left(\sum_{i\in J_-}|d_i|+\sum_{i\in J_\mu\setminus J_-}d_i\right)\geq (d+2)\pi(1-\mu)b^2|\ln\v|+\mathcal{O}(1).
\]
Consequently \eqref{BrneSup-ApplicF(v)} implies  $d(1+o(1))\geq(d+2)(1-\mu)-o(1)$. This inequality gives $\mu\geq\dfrac{2}{d+2}-o(1)$ which is in contradiction with $0<\mu<(\mathcal{D}_{K,b}+1)^{-1}$ for sufficiently small $\v>0$ [here we used $\mathcal{D}_{K,b}\geq\M_K\geq d$].\\

{\bf Step 2. We prove that $\dist(z_i,\o_\v)<\sqrt\v$ for all $i$}\\

We argue by contradiction and we assume the existence of a subsequence still denoted by $\v=\v_n\downarrow0$ and $i_0\in J_\mu$ s.t. $\dist(z_{i_0},\o_\v)\geq\sqrt\v$. From \eqref{EstLoinInterfaceU} we have $\inf_{B(z_{i_0},\rad)}\alpha\geq1-o(|\ln\v|^{-2})$. Consequently using \eqref{BorneInfEnWeight} we get  $F(v,B(z_{i_0},\rad))\geq d_{i_0}\pi(1-\mu)|\ln\v|-\mathcal{O}(1)$. Then $F(v)\geq \pi b^2(1-\mu)d|\ln\v|+\pi(1-b^2)(1-\mu)d_{i_0}|\ln\v|-\mathcal{O}(1)$.

From \eqref{BrneSup-ApplicF(v)} we obtain
\[
db^2|\ln\v|+\mathcal{O}(\ln|\ln\v|)\geq b^2(1-\mu)d|\ln\v|+(1-b^2)(1-\mu)|\ln\v|-\mathcal{O}(1).
\]
The last estimate implies $\mu\geq\dfrac{1-b^2}{b^2d+1-b^2}+o(1)$ which is in contradiction with $\mu\leq \dfrac{1-b^2}{2(\mathcal{D}_{K,b}+1)}$ for $\v>0$ sufficiently small.
\end{proof}
\begin{defi}\label{DefiSousEnsJ}
\begin{itemize}
\item For $i\in J_\mu$we let $y_i\in\delta\cdot\Z^2$ be the unique point s.t. $z_i\in B(y_i,\delta/2)$. Since $\dist(z_i,\o_\v)<\sqrt\v$ for all $i$,  $y_i$ is well defined. 
\item We denote also $\tilde J\subseteq J_\mu$ a set of indices s.t. $\cup_{i\in J_\mu} B(z_i,\rad)\subset\cup_{k\in\tilde J}B(y_k,2\lambda\delta)$ and for $k,l\in\tilde J$ s.t. $k\neq l$ we have $y_k\neq y_l$. We then let for $k\in\tilde J$, $\tilde{J}_k:=\{i\in J_\mu\,|\,z_i\in B(y_k,2\lambda\delta)\}$.
\item We may also select "good indices" in order to get well separated centers $y_k$'s. Using Lemma \ref{Lem.Separation} with $\num=17,\eta=\delta$, there exists a set $\emptyset\neq J^{(y)}\subset J_\mu$ and a number $\kappa\in\{1,17,...,17^{{\rm Card}(J_\mu)-1}\}$ [dependent on $\v$] s.t.
\[
\cup_{k\in\tilde J} B(y_k,\delta)\subset\cup_{k\in J^{(y)}}B(y_k,\kappa\delta)\text{ and for $k,l\in J^{(y)}$ with $k\neq l$ we have }\,|y_k-y_l|\geq16\kappa\delta.
\]
We denote, for $k\in J^{(y)}$, $\tilde d_k:=\deg_{\p B(y_k,\kappa\delta)}(v)$.
\item There exists also $\{J_k\,|\,k\in J^{(y)}\}$, a partition of $J_\mu$ with non empty sets [dependent on $\v$], s.t. 
\[
B(z_i,\delta/2)\subset B(y_k,\kappa\delta)\Longleftrightarrow i\in J_k\text{ for }k\in J^{(y)}.
\]
\end{itemize}
\end{defi}
We are going to prove that $\tilde J= J_\mu$ and for all $k\in J^{(y)}$ we have ${J}_k=\tilde{J}_k$.
\begin{prop}\label{PropToutLesDegEg1}
Assume \eqref{HypSurMu}, for $\v>0$ sufficiently small, if $J_\mu\neq\emptyset$ then $d_i=1$ for all $i\in J_\mu$.
\end{prop}
\begin{proof}

We argue by contradiction and we assume the existence of a subsequence [still denoted by $\v=\v_n\downarrow0$] s.t. for all $\v$ there exits $i_0\in J_\mu$ s.t. $ d_{i_0}\geq2$.

From Corollary \ref{Cor.BorneInfProcheIncl}.\ref{Cor.BorneInfProcheIncl2} applied in $B(y_k,2\lambda\delta)\setminus\cup_{i\in \tilde{J}_k}\overline{B(z_i,\rad)}$ :
\begin{eqnarray*}
\dfrac{1}{2}\int_{\O_\rad}\alpha|\n v|^2
&\geq&\sum_{k\in \tilde J}\dfrac{b^2}{2}\int_{B(y_k,2\lambda\delta)\setminus\cup_{i\in \tilde{J}_k}\overline{B(z_i,\rad)}}|\n v|^2
\\&\geq&\pi b^2\sum_{k\in \tilde J}\sum_{i\in J_k}d_i^2\ln\left(\dfrac{\lambda\delta}{\rad}\right)-\mathcal{O}(1)
\\&\geq&\pi b^2\left(1+\sum_{i\in J_\mu}d_i\right)\ln\left(\dfrac{\lambda\delta}{\rad}\right)-\mathcal{O}(1).
\end{eqnarray*}
We then get $F(v)\geq\pi b^2(d|\ln\v|+|\ln\rad|)+\mathcal{O}(|\ln(\lambda\delta)|)$. Since $|\ln\v|=\mathcal{O}(|\ln\rad|)$ and $|\ln(\lambda\delta)|+\ln\h=o(|\ln\v|)$, this estimate is in contradiction with \eqref{BrneSup-ApplicF(v)} for sufficiently small $\v$.
\end{proof}
\begin{prop}\label{PropVortexProcheLambda}
Assume $\mu$ satisfies \eqref{HypSurMu} and $J_\mu\neq\emptyset$. Then for sufficiently small $\v>0$ we have $\dist({\bf z},\Lambda)\leq\dfrac{\ln\h}{\sqrt{\h}}$.
\end{prop}
The proof of the proposition uses the following obvious lemma whose proof is left to the reader.
\begin{lem}\label{LemSommeDegCarréDec}
\begin{enumerate}
\item\label{LemSommeDegCarréDec1} Let $N\in\N^*$, ${\bf D}\in\N^N$ and for $k\in\{1,...,N\}$ let $N_k\in\N^*$ and ${\bf d}^{(k)}\in\N^{N_k}$ be s.t. $D_k=\sum_id_i^{(k)}$. Then we have 
\[
\sum_{k=1}^ND_k^2\geq\sum_{k=1}^N\sum_{i=1}^{N_k}(d_i^{(k)})^2.
\]
Moreover the equality holds if and only if for all $k\in\{1,...,N\}$ and for all $i\in\{1,...,N_k\}$ we have $d_i^{(k)}\in\{0,D_k\}$. 
\item\label{LemSommeDegCarréDec2} Let $N,d\in\N^*$ and denote $\displaystyle E_d:=\min_{{{\bf D}\in\N^N,\,\sum D_k=d}}\sum_{k=1}^N D_k^2$. Then we have for ${\bf D}\in\N^N$ s.t. $\sum D_k=d$:
\[
\sum_{k=1}^N D_k^2=E_d\Longleftrightarrow{\bf D}\in\{\lfloor d/N\rfloor;\lceil d/N\rceil\}^N.
\]
\end{enumerate}
\end{lem}
\begin{proof}[Proof of Proposition \ref{PropVortexProcheLambda}]
We argue by contradiction and we assume the existence of a subsequence [still denoted by $\v=\v_n\downarrow0$] and $i_0\in J_\mu$ s.t. $\dist(z_{i_0},\Lambda)>\dfrac{\ln\h}{\sqrt{\h}}$. 

Then there exists $\eta>0$ [independent of $\v$] s.t. $\h\xi_0(z_{i_0})\geq-\h\|\xi_0\|_{L^\infty(\O)}+4\eta(\ln\h)^2$. Consequently: $-2\pi\h\sum \xi_0(z_i)\leq2\pi d\h\|\xi_0\|_{L^\infty(\O)}-4\eta(\ln\h)^2.$

From \eqref{FullDecDiscqVal0-PetitChamp} we get [for small $\v$]
\begin{eqnarray*}
F(v)&\leq& 2\pi d\h\|\xi_0\|_{L^\infty(\O)}-3\eta(\ln\h)^2
\\{[\text{Hyp. \eqref{BorneKMagn}}]}&\leq&\pi d|\ln\v|-2\eta(\ln\h)^2.
\end{eqnarray*}
Using \eqref{BorneInfEn} we get 
\begin{equation}\label{BorneSupContraDistLambda}
\dfrac{1}{2}\int_{\O_\rad}\alpha|\n v|^2\leq d\pi\left[b^2|\ln\rad|+(1-b^2)|\ln(\lambda\delta)|\right]-\eta(\ln\h)^2.
\end{equation}

We let $\chi:=10\max_{k\in\tilde J}\dist(y_k,\Lambda)$ and for $p\in\Lambda$, $D_p:=\deg_{\p B(p,\chi)}(v)$, $J_p:=\{k\in J^{(y)}\,|\,y_k\in B(p,\chi)\}$. For a latter use we claim that $\chi\geq\ln(\h)/\sqrt{\h}$ and then
\begin{equation}\label{Hyp.DistLambdaContra}
\lambda|\ln\chi|/\chi\to0.
\end{equation}
We have [see Definition \ref{DefiSousEnsJ} for notation]
\begin{eqnarray}\nonumber
&&\dfrac{1}{2}\int_{\O_\rad}\alpha|\n v|^2
\\\nonumber&\geq&\sum_{k\in \tilde J}\dfrac{1}{2}\int_{B(y_k,2\lambda\delta)\setminus\cup_{i\in \tilde{J}_k}\overline{B(z_i,\rad)}}\alpha|\n v|^2+
\sum_{k\in \tilde J}\dfrac{1}{2}\int_{B(y_k,\delta/3)\setminus\overline{B(y_k,2\lambda\delta)}}\alpha|\n v|^2+
\\\label{BigDecomp}&&+
\sum_{p\in\Lambda}\dfrac{1}{2}\int_{B(p,\chi)\setminus\cup_{k\in J_p}\overline{B(y_k,\kappa\delta)}}\alpha|\n v|^2
+\dfrac{1}{2}\int_{\O\setminus\cup_{p\in\Lambda}\overline{B(p,\chi)}}\alpha|\n v|^2.
\end{eqnarray}
It is clear that, for $k\in \tilde J$, we may use Corollary \ref{Cor.BorneInfProcheIncl}.\ref{Cor.BorneInfProcheIncl2} in $B(y_k,2\lambda\delta)\setminus\cup_{i\in \tilde{J}_k}\overline{B(z_i,\rad)}$ in order to get

\begin{equation}\label{OnCompteRelFin1}
\sum_{k\in \tilde J}\dfrac{1}{2}\int_{B(y_k,2\lambda\delta)\setminus\cup_{i\in \tilde{J}_k}\overline{B(z_i,\rad)}}\alpha|\n v|^2\geq b^2d\pi\ln\left(\dfrac{\lambda\delta}{\rad}\right)+\mathcal{O}(1).
\end{equation}

Let $k\in \tilde J$, from \eqref{EstLoinInterfaceU}  and Proposition \ref{Prop.ComparaisonAnneau}.\ref{Prop.ComparaisonAnneau2} we obtain 
\begin{equation}\label{OnCompteRelFin2}
\dfrac{1}{2}\int_{B(y_k,\delta/3)\setminus\overline{B(y_k,2\lambda\delta)}}\alpha|\n v|^2\geq\pi\deg_{\p B(y_k,2\lambda\delta)}(v)^2|\ln\lambda|+\mathcal{O}(1).
\end{equation}

Let $p\in\Lambda$ be s.t. $D_p\neq0$, Corollary \ref{Cor.BorneInfProcheIncl}.\ref{Cor.BorneInfProcheIncl1} gives 
\[
\dfrac{1}{2}\int_{B(p,\chi)\setminus\cup_{k\in J_p}\overline{B(y_k,\kappa\delta)}}\alpha|\n v|^2\geq\pi\sum_{k\in J_p}\tilde d_k^2\ln \left(\dfrac{\chi}{\delta}\right)+\mathcal{O}(1).
\]
From Propositions \ref{MinimalMapHomo}$\&$\ref{Prop.EnergieRenDef}$\&$\ref{VeryNiceCor} [with \eqref{Hyp.DistLambdaContra}] we deduce
\[
\dfrac{1}{2}\int_{\O\setminus\cup_{p\in\Lambda}\overline{B(p,\chi)}}\alpha|\n v|^2\geq\pi\sum_{p\in\Lambda}D_p^2|\ln\chi|+\mathcal{O}(1).
\]
From Lemma \ref{LemSommeDegCarréDec}.\ref{LemSommeDegCarréDec1} we have $d\leq\sum_{k\in\tilde J}\deg_{\p B(y_k,2\lambda\delta)}(v)^2\leq\sum_{p\in\Lambda}\sum_{k\in J_p}\tilde d_k^2\leq\sum_{p\in\Lambda}D_p^2$. Then we get
\[
\dfrac{1}{2}\int_{\O_\rad}\alpha|\n v|^2\geq d\pi\left[b^2|\ln\rad|+(1-b^2)|\ln(\lambda\delta)|\right]+\mathcal{O}(1).
\]
This estimate is in contradiction with \eqref{BorneSupContraDistLambda} for sufficiently small $\v$.
\end{proof}
\begin{prop}\label{Prop.BonEcartement}
Assume $\mu$ satisfies \eqref{HypSurMu} and let $\v=\v_n\downarrow0$ be a sequence.
\begin{enumerate}
\item If ${\rm Card}(J_\mu)\geq2$ then for $\v>0$ sufficiently small and for $i\neq j$, $|z_i-z_j|\geq\h^{-1}\ln\h$. 
\item For $\v>0$ sufficiently small we have for $p\in\Lambda$, $\deg_{\p B(p,\h^{-1/2}\ln\h)}(v)\in\{\lfloor d/N_0\rfloor;\lceil d/N_0\rceil\}$.
\end{enumerate}
\end{prop}
The proof of Proposition \ref{Prop.BonEcartement} is postponed to Appendix \ref{Proof.Prop.BonEcartement}.

Since $\lambda\delta\h\to0$,  Proposition \ref{Prop.BonEcartement} implies that each cell of period contains at most a disc $B(z_i,\rad)$ with $i\in J_\mu$. 

Following the argument in \cite{Publi4} [proof of the third part in Proposition 3.6, see Appendix D-Section 4.5], we may refined Proposition \ref{PropDegPinning1}.\ref{PropDegPinning1.2}.
\begin{prop}\label{Prop.PinningComplet}
Assume $\mu$ satisfies \eqref{HypSurMu}, then there is $\eta_{\o,b}>0$ depending only on $\o$ and $b$ s.t. for $i\in J_\mu$ we have $B(z_i,2\eta_{\o,b}\lambda\delta)\subset\o_\v$.
\end{prop}
\begin{cor}
Assume $\mu$ satisfies \eqref{HypSurMu}. Then we have
\begin{equation}\label{Onfaitcommeonpeut}
\int_{\O\setminus \cup_{i\in J_\mu}B(z_i,\lambda^2\delta^2)}|\n v|^2+\dfrac{1}{\v^2}(1-|v|^2)^2=\mathcal{O}(|\ln(\lambda\delta)|).
\end{equation}
Moreover
\begin{equation}\label{Onfaitcommeonpeut...}
\text{$|v|=1+o(1)$ in $\O\setminus \cup_{i\in J_\mu}B(z_i,2\lambda^2\delta^2)$.}
\end{equation}
\end{cor}
\begin{proof}
We have
\[
\dfrac{b^4}{4}\int_{\O\setminus \cup_{i\in J_\mu}B(z_i,\lambda^2\delta^2)}|\n v|^2+\dfrac{1}{\v^2}(1-|v|^2)^2\leq F(v)-\sum_{i\in J_\mu}F(v,B(z_i,\lambda^2\delta^2)).
\]
For $i\in J_\mu$, from Corollary \ref{Prop.ComparaisonAnneau}.\ref{Prop.ComparaisonAnneau2} :
\begin{eqnarray*}
F(v,B(z_i,\lambda^2\delta^2))&\geq&\dfrac{b^2}{2}\int_{B(z_i,\lambda^2\delta^2)\setminus\overline{B(z_i,\rad)}}|\n v|^2+F(v,B(z_i,\rad))
\\&\geq&2b^2\pi\ln(\lambda\delta)+b^2\pi|\ln\v|+\mathcal{O}(1).
\end{eqnarray*}
Since, by Proposition \ref{Prop.BonEcartement}, the discs $B(z_i,\lambda^2\delta^2)$ are pairwise disjoint, we obtain with \eqref{BrneSup-ApplicF(v)}:
\[
\dfrac{b^4}{4}\int_{\O\setminus \cup_{i\in J_\mu}B(z_i,\lambda^2\delta^2)}|\n v|^2+\dfrac{1}{\v^2}(1-|v|^2)^2\leq \mathcal{O}(|\ln(\lambda\delta)|).
\]This estimate is equivalent to \eqref{Onfaitcommeonpeut}.

We are going to prove \eqref{Onfaitcommeonpeut...}. We argue by contradiction and we assume the existence of an extraction still denoted $\v=\v_n\downarrow0$, $t\in(0,1)$ and $(x_n)_n\subset\O\setminus \cup_{i\in J_\mu}B(z_i,2\lambda^2\delta^2)$ s.t. $|v_{\v_n}(x_n)|<t$. 

By Proposition \ref{Prop.EtaEllpProp}, there exists $C_t>0$ s.t. for sufficiently large $n$:
\begin{equation}\label{BarEqBar}
\int_{B(x_n,\sqrt{\v}_n)\cap\O}|\n v_{\v_n}|^2+\dfrac{1}{\v_n^2}(1-|v_{\v_n}|^2)^2> {C_t}{}|\ln\v_n|.
\end{equation}
Moreover, for $n$ sufficiently large to get  $\sqrt{\v}_n<\lambda^2\delta^2$,  we have $[B(x_n,\sqrt{\v}_n)\cap\O]\subset\O\setminus \cup_{i\in J_\mu}B(z_i,\lambda^2\delta^2)$. This inclusion  is in contradiction with \eqref{Onfaitcommeonpeut} and \eqref{BarEqBar}.
\end{proof}





From Proposition \ref{Prop.PinningComplet}, for $i\in J_\mu$, we have $\hat z_i:=\dfrac{z_i-y_i}{\lambda\delta}\in\o$ where $y_i\in\delta\Z^2$ is s.t. $z_i\in B(y_i,\lambda\delta)$. Moreover, up to consider an extraction, we may assume that, for $i\in J_\mu$, there exits $\hat z^0_i\in\o$ s.t. $\hat z_i\to\hat z_i^0$.

We start with the following proposition.
\begin{prop}\label{Prop.BorneInfTrèsFine}We have the following sharp lower bound:
\begin{eqnarray*}
\F(v,A)&\geq&\h^2 \Jo+d\Pic\left[-\h+\HoC \right]+\L_1(d)\ln\h+\L_2(d)+\phantom{gsgsgsgs}
\\&&\phantom{gsgsgsgs}+\sum_{i\in J_\mu}[W^{\rm micro}(\hat{z}_i^0)-\min_\o W^{\rm micro}]+[\W_{d}({\bf D})-\Wmin_{d}]+o(1)
\end{eqnarray*}
where  $\Wmin_{d}=\min_{\LamN} \W_{d}$ is defined in \eqref{CouplageEnergieRen} and
\begin{equation}\label{DefWdOpD}
\W_{d}({\bf D}):=W_{N_0}^{\rm macro}\pD+\sum_{p\in\Lambda}C_{p,D_p}+\tilde{V}[\zeta_\pD]
\end{equation}
where for $p\in\Lambda$, $D\in\N^*$, $C_{p,D}$ is defined in \eqref{DefCpD}, $C_{p,0}:=0$ and $\tilde{V}[\zeta_\pD]$ is defined in Proposition \ref{Prop.Information.Zeta-a}.
\end{prop}
We split the proof of Proposition \ref{Prop.BorneInfTrèsFine} in {several} lemmas.

The first step is the following lemma consisting in a "macroscopic/mesoscopic" version of Proposition \ref{Prop.BorneInfTrèsFine}.
\begin{lem}\label{BorneTresFineSansIncl}
Let  $\rho=|v|$ and  $w=v/\rho$ in $\O\setminus\cup_{i\in J_\mu}\overline{B(y_i,\delta/3)}$ . We then have
\begin{eqnarray*}
\dfrac{1}{2}\int_{\O\setminus\cup_{i\in J_\mu}\overline{B(y_i,\delta/3)}}\alpha\rho^2|\n w|^2&\geq& d\pi|\ln(\delta/3)|-\pi\sum_{\substack{p\in\Lambda\\D_p\geq2}}\sum_{\substack{i,j\in J_p\\i\neq j}}\ln|z_i-z_j|+
\\&&\phantom{jdjdjdjdjdjdjdjd}+W^{\rm macro}_{N_0}\pD+o(1).
\end{eqnarray*}
\end{lem}
\begin{proof}
On the one hand, from Proposition \ref{PropVortexProcheLambda} and letting $\chi:=\h^{-1/4}$ we have $|v|\geq1/2$ in $\O\setminus\cup_{p\in\Lambda}\overline{B(p,\chi)}$. Then, from Proposition \ref{VeryNiceCor}, we have
\begin{equation}\label{EstimationFineLoinLambda}
\dfrac{1}{2}\int_{\O\setminus\cup_{p\in\Lambda}\overline{B(p,\chi)}}\alpha|\n v|^2\geq\pi\sum_{p\in\Lambda}D_p^2|\ln\chi|+W^{\rm macro}_{N_0}\pD+o(1).
\end{equation}
On the other hand, from Proposition \ref{Prop.BonEcartement}, if ${\rm Card}(J_\mu)\geq2$ then, for $i,j\in J_\mu$ with $i\neq j$, we have  {$|y_i-y_j|\geq\h^{-1}\ln(\h)-2\lambda\delta$}. 

Consequently, if $D_p=\deg_{\p B(p,\eta_\O)}(v)\neq0$ [$\eta_\O$ is defined in \eqref{DefEtaO}], letting $J_p:=\{i\in J_\mu\,|\,z_i\in B(p,\eta_\O)\}$, $\dom_p:=B(p,\chi)\setminus\cup_{i\in J_p}\overline{B(y_i,\h^{-1})}$,
\[
\begin{array}{cccc}
\Phi:&B(p,\chi)&\to&\D=B(0,1)\\&x&\mapsto&\dfrac{x-p}{\chi}
\end{array},
\] $\hat v=v\circ\Phi^{-1}$, $\hat\alpha=\alpha\circ\Phi^{-1}$, $\hat\dom_p:=\Phi(\dom_p)$ and $\hat y_i:=\Phi(y_i)$ for $y_i\in B(p,\chi)$, then we may apply Proposition \ref{VeryNiceCor}. Writing $(\hat{\bf y}_p,{\bf 1}):=\{(\hat y_i,1)\,|\,i\in J_p\}$, Proposition \ref{VeryNiceCor} gives:
\begin{equation}\label{EstimationAutourLambda}
\dfrac{1}{2}\int_{\dom_p}\alpha|\n v|^2=\dfrac{1}{2}\int_{\hat \dom_p}\hat\alpha|\n \hat v|^2\geq\pi D_p\ln(\chi\h)+W^{\rm macro}_{D_p,\D}(\hat{\bf y}_p,{\bf 1})+o(1)
\end{equation}
where $W^{\rm macro}_{D_p,\D}$  is the macroscopic renormalized energy in  the unit disc $\D$ with $D_p$ points.

From Proposition 1 in \cite{LR1} we have
\[
W^{\rm macro}_{D_p,\D}(\hat{\bf y}_p,{\bf 1})=-\pi\sum_{\substack{i,j\in J_p\\i\neq j}}\left[\ln|\hat y_i-\hat y_j|-\ln|1-\hat y_i\overline{\hat y}_j|\right]+\pi\sum_{i\in J_p}\ln(1-|\hat y_i|^2).
\]
Using Proposition \ref{PropVortexProcheLambda}, we get for $i\in J_p$, $|\hat y_i|\leq\dfrac{\h^{-1/2}\ln\h}{\chi}=o(1)$ and then
\begin{equation}\label{EstimationAutourLambdaBis}
W^{\rm macro}_{D_p,\D}(\hat{\bf y}_p,{\bf 1})
=-\pi\sum_{\substack{i,j\in J_p\\i\neq j}}\ln|y_i- y_j|-\pi(D_p^2-D_p)|\ln\chi|+o(1).
\end{equation}
For $i\in J_\mu$, we let $\Ring_i:=B(y_i,\h^{-1})\setminus\overline{B(y_i,\delta/3)}$. With Proposition \ref{Prop.ComparaisonAnneau}.\ref{Prop.ComparaisonAnneau1} we obtain
\begin{equation}\label{EstimationAutourInclusion}
\dfrac{1}{2}\int_{\Ring_i}\alpha|\n v|^2\geq\pi|\ln\left(\delta\h/3\right)|.
\end{equation}
By combining \eqref{EstimationFineLoinLambda}, \eqref{EstimationAutourLambda}, \eqref{EstimationAutourLambdaBis} and \eqref{EstimationAutourInclusion} the result is proved.
\end{proof}
The second step is a "microscopic" version of Proposition \ref{Prop.BorneInfTrèsFine}.
\begin{lem}\label{Prop.EstFineAtraversIncl}

If $\rad\leq\Rad\leq\lambda^2\delta^2$, then :
\[
\sum_{i\in J_\mu}F(v,\Ring_i)\geq d\pi\left(|\ln(3\lambda)|+b^2|\ln(\lambda\delta/\Rad)|\right)+ \sum_{i\in J_\mu}W^{\rm micro}(\hat z_i^0)+o(1)
\]
where, for $i\in J_\mu$, $\Ring_i:=B(y_i,\delta/3)\setminus\overline{B(z_i,\Rad)}$.
\end{lem}
\begin{proof}

We first note that in order to prove Lemma \ref{Prop.EstFineAtraversIncl} [up to replace $v$ by $\underline{v}$] we may assume $\rho=|v|\leq1$. We may also assume
\begin{equation}\label{BorneATraversIncl1}
\sum_{i\in J_\mu}F_\v(v,\Ring_i)=\mathcal{O}(|\ln(\lambda\delta)|)
\end{equation}
since in the contrary case there is nothing to prove.

Fix $i\in J_\mu$ and let $v_\star$ be a minimizer of $\displaystyle  F_\v(\cdot,\Ring_i)$ in $H^1(\Ring_i,\C)$ with the Dirichlet boundary condition $\tr_{\p \Ring_i}(\cdot)=\tr_{\p \Ring_i}(v)$. Note that such minimizers exist and we have $F_\v(v_\star,\Ring_i)\leq F_\v(v,\Ring_i)=\mathcal{O}(|\ln(\lambda\delta)|)$.

The key ingredient consists in noting that since $v_\star$ is a minimizer of a weighted Ginzburg-Landau type energy we may thus  use a sharp interior $\eta$-ellipticity result. Namely, following the strategy of \cite{Publi3} to prove Lemma 1 [see  Appendix C in \cite{Publi3}], by using the first part of the proof [the interior argument which does not required any information on $\tr_{\p \Ring_i}(v_\star)$], we  get 
\begin{equation}\label{Putaind'etaTruc}
\rho_\star:=|v_\star|\geq1-\mathcal{O}(\sqrt{|\ln(\lambda\delta)|/|\ln\v|})\text{ in }\tilde\Ring_i:=B(y_i,\delta/3-\v^{1/4})\setminus\overline{B(z_i,\Rad+\v^{1/4})}.
\end{equation}
Write in $\tilde\Ring_i$: $v_\star=\rho_\star w_\star$ where $w_\star\in H^1(\tilde\Ring,\S^1)$.

Note that by \eqref{CondOnLambdaDelta} [namely $|\ln(\lambda\delta)|=\mathcal{O}(\ln|\ln\v|)$] we have $|\ln(\lambda\delta)|^3/|\ln\v|=o(1)$ and then from \eqref{BorneATraversIncl1} $\&$ \eqref{Putaind'etaTruc} [and aslo $\rho_\star\leq1$] we have
\[
\int_{\tilde\Ring_i}\alpha\rho_\star^2|\n w_\star|^2=\int_{\tilde\Ring_i}\alpha|\n w_\star|^2+o(1).
\]
We then immediately get:
\[
F(v,\Ring_i)\geq F(v_\star,\Ring_i)\geq \dfrac{1}{2}\int_{\tilde\Ring_i}\alpha|\n w_\star|^2+o(1)\geq\inf_{\substack{\tilde w\in H^1(\tilde\Ring_i,\S^1)\\\deg(\tilde w)=1}}\frac{1}{2}\int_{\tilde\Ring_i}\alpha|\n\tilde w|^2+o(1).
\]
It suffices now to claim that from \eqref{DefRenMicroEn3} we have 
\[
\inf_{\substack{\tilde w\in H^1(\tilde\Ring_i,\S^1)\\\deg(\tilde w)=1}}\frac{1}{2}\int_{\tilde\Ring_i}\alpha|\n\tilde w|^2=\pi\left(|\ln(3\lambda)|+b^2|\ln(\lambda\delta/\Rad)|\right)+ W^{\rm micro}(\hat z_i^0)+o(1)
\]
in order to get $F(v,\Ring_i)\geq \pi\left(|\ln(3\lambda)|+b^2|\ln(\lambda\delta/\Rad)|\right)+ W^{\rm micro}(\hat z_i^0)+o(1)$. By summing these lower bounds we get the result.
\end{proof}
\begin{lem}
There exits $\rad\leq\Rad=o(\lambda^2\delta^2)$ s.t. for $i\in J_\mu$ we have
\[
F[v,B(z_i,\Rad)]\geq b^2[\pi\ln(\Rad/\v)+\ln b +\gamma]+o(1).
\]
\end{lem}
\begin{proof}
We first note that we have
\begin{equation}\label{BorneDansAnneaPourConstruBonneTrace}
\sum_{i\in J_\mu}F[v,B(z_i,\lambda^2\delta^2)\setminus\overline{B(z_i,\rad)}]\leq db^2 \pi\ln(\lambda^2\delta^2/\rad)+\L_1(d)\ln\h+\mathcal{O}(1).
\end{equation}
The above estimate is proved by contradiction and assuming the existences of an extraction [still denoted by  $\v=\v_n\downarrow0$] and of a sequence  $R_n\uparrow\infty$ s.t. 
\[
\sum_{i\in J_\mu}F[v,B(z_i,\lambda^2\delta^2)\setminus\overline{B(z_i,\rad)}]\geq db^2\pi\ln(\lambda^2\delta^2/\rad)+\L_1(d)\ln\h+R_n.
\]
From  \eqref{BorneInfEnWeight} we get
\[
\sum_{i\in J_\mu}F[v,B(z_i,\lambda^2\delta^2)]\geq db^2\pi\ln(\lambda^2\delta^2/\v)+\L_1(d)\ln\h+R_n+\mathcal{O}(1).
\]
Using Lemmas \ref{BorneTresFineSansIncl} and \ref{Prop.EstFineAtraversIncl} we get an estimate which contradicts \eqref{BrneSup-ApplicF(v)}.

By a classical argument, for sufficiently small $\v$, there exists $\sqrt\rad\leq\Rad\leq\rad^{1/4}$ s.t. for $i\in J_\mu$
\[
\dfrac{\Rad}{2}\int_{\p B(z_i,\Rad)}|\n v|^2+\dfrac{b^2}{2\v^2}(1-|v|^2)^2\leq \pi+\dfrac{4\L_1(d)\ln\h+\mathcal{O}(1)}{|\ln\rad|}
\]
Arguing as in the proof of Proposition \ref{Prop.ConstrEpsMauvDisk} [Step 3 in Appendix \ref{SectAppenPreuveConstructionPetitDisque}] it is clear that we may assume $|v|\geq 1-|\ln\v|^{-2}$ on $ \p B(z_i,\Rad)$ for $i\in J_\mu$.

We now define for $i\in J_\mu$, $\rho_i:=\tr_{\p B(z_i,\Rad)}(|v|)$, $w_i:=\tr_{\p B(z_i,\Rad)}(v/|v|)$. We immediately get
\[
\dfrac{\Rad}{2}\int_{\p B(z_i,\Rad)}|\n w_i|^2=\pi+o(1),\,\dfrac{\Rad}{2}\int_{\p B(z_i,\Rad)}|\n \rho_i|^2+\dfrac{b^2}{2\v^2}(1-\rho_i^2)^2=o(1).
\]
On the other hand, since $\deg(w_i)=1$, there exists $\phi_i=\phi_{i,\v}\in H^1((0,2\pi),\R)$ s.t. $\phi_i(0)=\phi_i(2\pi)\in[0,2\pi)$ and $w_i\left(z_i+\Rad \e^{\imath\theta}\right)=\e^{-\imath(\theta+\phi_i(\theta))}$. A direct calculation gives:
\[
2\pi+o(1)={\Rad}{}\int_{\p B(z_i,\Rad)}|\p_\tau w_i|^2=\int_0^{2\pi}\left|(\phi_i+\theta)'\right|^2=2\pi+\int_0^{2\pi}\left|\phi'_i\right|^2.
\]
The last equalities imply  $\phi_i'\to0$ in $L^2(0,2\pi)$ and then $\phi_i-\phi_i(0)\to0$ in $L^2(0,2\pi)$. Hence, up to pass to a subsequence, we get the existence of $\theta_i\in[0,2\pi]$ s.t. $\phi_i\to\theta_i$ in $H^1(0,2\pi)$.

We now define $\tilde w_i\in H^1(B(z_i,2\Rad)\setminus\overline{B(z_i,\Rad)},\S^1)$ by
\[
\tilde w_i(z_i+s\e^{\imath\theta})=\e^{\imath[\theta+\tilde\phi_i(z_i+s\e^{\imath\theta})]}
\text{ with }
\tilde\phi_i(z_i+s\e^{\imath\theta})=\left[\phi_i(\theta)-\theta_i\right]\dfrac{2\Rad-s}{\Rad}+\theta_i.
\]
A direct calculation gives $\int_{B(z_i,2\Rad)\setminus\overline{B(z_i,\Rad)}}|\n\tilde\phi_i|^2=o(1)$ and then
\[
\dfrac{1}{2}\int_{B(z_i,2\Rad)\setminus\overline{B(z_i,\Rad)}}|\n\tilde w_i|^2=\dfrac{1}{2}\int_{B(z_i,2\Rad)\setminus\overline{B(z_i,\Rad)}}|\n[\theta+\tilde\phi_i(z_i+s\e^{\imath\theta})]|^2+o(1)=\pi\ln(2)+o(1).
\]
Let $\tilde\rho_i\in H^1[B(z_i,2\Rad)\setminus\overline{B(z_i,\Rad)},\R^+]$ be s.t.
 $\tilde\rho_i(z_i+s\e^{\imath\theta}):=\tilde\rho_i(z_i+\Rad\e^{\imath\theta})\dfrac{2\Rad-s}{\Rad}+\dfrac{s-\Rad}{\Rad}$.
 
We then have $F[\tilde\rho_i,B(z_i,2\Rad)\setminus\overline{B(z_i,\Rad)}]=o(1)$. Consequently, letting $v_i:=\tilde\rho_i\tilde w_i\in H^1[B(z_i,2\Rad)\setminus\overline{B(z_i,\Rad)},\C]$ we have 
\[
F[v_i,B(z_i,2\Rad)\setminus\overline{B(z_i,\Rad)}]=\dfrac{b^2}{2}\int_{B(z_i,2\Rad)\setminus\overline{B(z_i,\Rad)}}|\n \tilde w_i|^2+o(1).
\]

In order to conclude we let $u_i:=\begin{cases}v_i&\text{in }B(z_i,2\Rad)\setminus\overline{B(z_i,\Rad)}\\v&\text{in }B(z_i,\Rad)\end{cases}$.

It is clear that $u_i(z_i+2\Rad\e^{\imath\theta})=\e^{\imath\theta_i}\e^{\imath\theta}$ and then, using Lemma  IX.1 in \cite{BBH}, we get 
\[
F[u_i,B(z_i,2\Rad)]\geq b^2[\pi\ln(2\Rad/\v)+\gamma+\pi\ln b ]+o(1).
\]
The last estimate ends the proof of the lemma.
\end{proof}
\begin{proof}[Proof of Proposition \ref{Prop.BorneInfTrèsFine}]
From the three previous lemmas we have
\begin{eqnarray}\nonumber
F(v)&\geq& d\pi\left[b^2|\ln\v|+(1-b^2)|\ln(\lambda\delta)|\right]-\pi\sum_{\substack{p\in\Lambda\\D_p\geq2}}\sum_{\substack{i,j\in J_p\\i\neq j}}\ln|z_i-z_j|+
\\\label{BigBig1}&&+W^{\rm macro}_{N_0}\pD+\sum_{i\in J_\mu}W^{\rm micro}(\hat z_i^0)+db^2[\pi\ln b +\gamma]+o(1).
\end{eqnarray}
On the other hand, with Corollary \ref{Cor.DecompPourCluster} [estimate \eqref{NiceDecSharpSplitTildeV}] we get
\begin{equation}\label{BigBig2}
\F(v,A)\geq \h^2\Jo+2\pi\h\sum_{i\in J_\mu}\xi_0(z_i)+F(v)+\tilde{V}[\zeta_\pD]+o(1)
\end{equation}
where $\zeta_\pD$ is defined in Proposition \ref{PropPartieMinimalSandH0}.

From Proposition \ref{EnergieRenMeso} [estimate \eqref{DevMesoscopicDef}], for $p\in\Lambda$ s.t. $D_p\geq2$, we have:
\begin{equation}\label{BigBig3}
-\pi\sum_{\substack{i,j\in J_p\\i\neq j}}\ln|z_i-z_j|+2\pi\h\sum_i[\xi_0(z_i)-\xi_0(p)]\geq\dfrac{\pi}{2}(D_p^2-D_p)\ln\left(\dfrac{\h}{D_p}\right)+C_{p,D_p}+o(1).
\end{equation}
By combining \eqref{BigBig1}, \eqref{BigBig2} and \eqref{BigBig3} [and also $\xi_0\leq0$] we obtain
\begin{eqnarray}\nonumber
\F(v,A)&\geq&\h^2\Jo+d\pi\left[b^2|\ln\v|+(1-b^2)|\ln(\lambda\delta)|\right]-2\pi d\h\|\xi_0\|_{L^\infty(\O)}+
\\\nonumber&&+\dfrac{\pi}{2}\sum_{\substack{p\in\Lambda\\D_p\geq2}}\left[(D_p^2-D_p)\ln\left(\dfrac{\h}{D_p}\right)+C_{p,D_p}\right]+W^{\rm macro}_{N_0}\pD+
\\\label{BigBig4}&&+\sum_{i\in J_\mu}W^{\rm micro}(\hat z_i^0)+\tilde{V}[\zeta_\pD]+db^2[\pi\ln b +\gamma]+o(1).
\end{eqnarray}
It suffices to see that, since ${\bf D}\in\LamN$, from the definition of $\L_1(d)$ we have
\[
\dfrac{\pi}{2}\sum_{\substack{p\in\Lambda\\D_p\geq2}}(D_p^2-D_p)\ln\left(\dfrac{\h}{D_p}\right)=\L_1(d)\ln{\h}+\dfrac{\pi}{2}\sum_{{p\in\Lambda}}(D_p-D_p^2)\ln\left({D_p}\right)
\] in order to deduce from \eqref{BigBig4} that 
\begin{eqnarray*}
\F(v,A)&\geq&\h^2\Jo+d\pi\left[-2 \h\|\xi_0\|_{L^\infty(\O)}+b^2|\ln\v|+(1-b^2)|\ln(\lambda\delta)|\right]+
\\&&+\L_1(d)\ln{\h}+\sum_{i\in J_\mu}W^{\rm micro}(\hat z_i^0)+\W_{d}({\bf D})+
\\&&+\dfrac{\pi}{2}\sum_{{p\in\Lambda}}(D_p-D_p^2)\ln\left({D_p}\right)+db^2[\pi\ln b +\gamma]+o(1)
\end{eqnarray*}
 where $\W_{d}({\bf D})$ is defined in \eqref{DefWdOpD}. This estimate with the definition of $\HoC$ and $\Wmin_d$ [see \eqref{CouplageEnergieRen}$\&$\eqref{DefH0c1}$\&$\eqref{DefGammaBo}] ends the proof of the proposition.
\end{proof}
\section{The first critical field and the  location of the vorticity defects}\label{Sec.LocationVorticity}
We assume that $\lambda,\delta,\h$ satisfy \eqref{CondOnLambdaDelta} and \eqref{BorneKMagn} for some $K\geq0$  independent of $\v$. We assume also \eqref{PutaindHypTech}. We consider a sequence $\v=\v_n\downarrow0$.

As in the previous section we focus on sequences of quasi-minimizers of $\F$. For simplicity we write $(v,A)$ instead of $(v_\v,A_\v)$. We assume that \eqref{HypGlobalSurQuasiMin}$\&$\eqref{QuasiMinDef} holds and since \eqref{HypGlobalSurQuasiMin}$\&$\eqref{QuasiMinDef} are gauge invariant we may also assume that $(v,A)$ is in the Coulomb gauge.

From above results, for a fixed $\mu>0$ sufficiently small [satisfying \eqref{HypSurMu}] and for $\v>0$ sufficiently small, there exists a [finite] set $\Zz\subset\O$, depending on $\v$ and possibly empty s.t. letting $d:={\rm Card}(\Zz)$ [we write $\Zz=\{z_1,...,z_2\}$]:
\begin{itemize}
\item If $d=0$, then  $|v|\geq1/2$ in $\O$. 
\item If $d>0$, then  $|z_i-z_j|\gtrsim\h^{-1}\ln \h$ if $i\neq j$,  $|v|\geq1/2$ in $\O\setminus\cup_{i=1}^d\overline{B(z_i,\v^\mu)}$ and $\deg_{\p B(z,\v^\mu)}(v)=1$ for $z\in\Zz$.
\end{itemize}
Moreover $d=\mathcal{O}(1)$. Then if needed, up to pass to a subsequence, we may assume that $d$ is independent of $\v$.

By combining Corollary \ref{CorEtudeSansVortex}, Propositions \ref{EnergieRenMeso}, \ref{Prop.BorneSupSimple}, \ref{Prop.BonEcartement} and \ref{Prop.BorneInfTrèsFine} we get the following corollary.
\begin{cor}\label{Cor.ExactEnergyExp} 
Assume $\lambda,\delta,\h$ satisfy \eqref{CondOnLambdaDelta} and \eqref{BorneKMagn} for some $K\geq0$  independent of $\v$. Let $\v=\v_n\downarrow0$ and and let $((v_\v,A_\v))_\v\subset\H$  be a sequence  satisfying \eqref{HypGlobalSurQuasiMin}$\&$\eqref{QuasiMinDef}.   Assume that $d$ is independent of $\v$. Without loss of generality we may assume that $(v_\v,A_\v)$ is in the Coulomb gauge. We have
\begin{equation}\label{Exact...Expending}
\F(v_\v,A_\v)=\h^2 \Jo+d\Pic\left[-\h+\HoC \right]+\L_1(d)\ln\h+\L_2(d)+o(1).
\end{equation}


Moreover, if $d\neq0$ then:
\begin{itemize}
\item We have ${\bf D}\in\LamN$ [see \eqref{DefEnsCouplageEnergieRen}] and ${\bf D}$ minimises $\W_{d}$ in $\LamN$  where $\W_{d}$ is defined in \eqref{DefWdOpD}.
\item For $p\in\Lambda$ s.t. $D_p>0$ and $i\in J_p$, we denote $\breve z_i:=(z_i-p)\sqrt{D_p/\h}$ and $\breve{\bf z}_p:=\{\breve z_i\,|\,i\in J_p\}$. Then, up to pass to a subsequence, $\breve{\bf z}_p$ converges to a minimizer of $W^{\rm meso}_{p,D_p}$ defined in \eqref{DefEnergyRenMeso}.
\item For $i\in \{1,...,d\}$, we write $\hat z_i:=(z_i-y_i)/(\lambda\delta)\in\o$ where $y_i\in\delta\Z^2$ is s.t. $z_i\in B(y_i,\lambda\delta)$. Then, up to pass to a subsequence, $\hat z_i$ converges to a minimizer of $W^{\rm micro}$.
\end{itemize}
\end{cor}

For a further use, we claim that for  $\dtest\geq0$, from Proposition \ref{Prop.BorneSupSimple}, there exits a configuration $(\vtest,\Atest)\in\H$ which is in the  Coulomb gauge s.t.
\begin{equation}\label{ExactExpEnergTest}
\F(\vtest,\Atest)-\h^2 \Jo=\dtest\Pic\left[-\h+\HoC \right]+\L_1(\dtest)\ln\h+\L_2(\dtest)+o(1).
\end{equation}

Recall that, from Lemma \ref{LemLaisseLectTrucSimple}, for $d\neq0$, we have  $d\in\{1,...,N_0\}$ if and only if $\L_1(d)=0$ and $\L_2(d)=\Wmin_d$.  For further use we state another lemma whose proof is left to the reader:
\begin{lem}\label{TechLemmaDefDelta}For  $0\leq d<d'$ we let :
\begin{enumerate}
\item  $\di\DLU_d:=\dfrac{\L_1(d+1)-\L_1(d)}{\Pic}=\dfrac{\pi}{\Pic}\left\lfloor \dfrac{d}{N_0}\right\rfloor$.
\item  $\di\DLU_{d',d}:=\dfrac{\L_1(d')-\L_1(d)}{\Pic(d'-d)}=\dfrac{\pi}{\Pic(d'-d)}\sum_{k=d}^{d'-1}\left\lfloor \dfrac{k}{N_0}\right\rfloor$.
\item  $\di\DLD_d:=\dfrac{\L_2(d+1)-\L_2(d)}{\Pic}$ and $\di\DLD_d-\dfrac{\Wmin_{d+1}-\Wmin_d}{\Pic}=$
\[
=\left|\begin{array}{l}0\text{ if }d\leq N_0-1
\\
-\dfrac{\pi}{2\Pic}\left\lfloor \dfrac{d}{N_0}\right\rfloor\left[\left(1+\left\lfloor \dfrac{d}{N_0}\right\rfloor\right)\ln\left(1+\left\lfloor \dfrac{d}{N_0}\right\rfloor\right)+\left(1-\left\lfloor \dfrac{d}{N_0}\right\rfloor\right)\ln\left\lfloor \dfrac{d}{N_0}\right\rfloor\right]\text{ if }d\geq N_0
\end{array}\right..
\]
\item {$\di\DLD_{d',d}:=
\dfrac{\L_2(d')-\L_2(d)}{\Pic(d'-d)}$ thus, if $d'\leq N_0$, then $\DLD_{d',d}=\dfrac{\Wmin_{d'}-\Wmin_d}{\Pic(d'-d)}$}.
\end{enumerate}
\end{lem}
By using \eqref{Exact...Expending} and \eqref{ExactExpEnergTest} we easily get the following corollary.
\begin{cor}\label{Cor.ExactEnergyExpPreCritField} 
Let $\v=\v_n\downarrow0$, $\lambda$, $\delta$, $\h$ and  $((v_\v,A_\v))_\v\subset\H$ be as in Corollary \ref{Cor.ExactEnergyExp}.

  Assume that $d$ is independent of $\v$. Then we have for $ d'>d$
\begin{eqnarray}\nonumber
 \h\leq \HoC +\DLU_{d',d}\times\ln\h+\DLD_{d',d}+o(1).
\end{eqnarray}
Then, letting $\chi$ be s.t. $\h =\HoC (1+\chi)$ [$\chi=o(1)$ from \eqref{BorneKMagn}], we have thus 
\begin{eqnarray}\label{BorneSUPUNDEGENnplusPrecise}
\h\leq \HoC +\DLU_{d',d}\times\ln\HoC +\DLD_{d',d}+o(1).
\end{eqnarray}
If $d>d'\geq0$ then 
\begin{eqnarray}\label{BorneINFUNDEGENnplusPrecise}
\h\geq \HoC +\DLU_{d,d'}\times\ln\HoC +\DLD_{d,d'}+o(1).
\end{eqnarray}
\end{cor}

We are now in position to give an asymptotic value for the first critical field. Indeed with Corollary \ref{Cor.ExactEnergyExpPreCritField}  [\eqref{BorneSUPUNDEGENnplusPrecise} with $d=0\&d'\in\{1,...,N_0\}$ and \eqref{BorneINFUNDEGENnplusPrecise} with $d\geq1\&d'=0$]. 

Recall that we write, for $x\in\R$,  $[x]^+=\max(x,0)$ and $[x]^-=\min(x,0)$
\begin{cor}\label{CorDefPremierChampsCrit}
Denote $H_{c_1}:=\HoC +\min_{d\in\{1,...,N_0\}}{\dfrac{\Wmin_{d}}{d\Pic}}$. Let $\{(v_\v,A_\v)\,|\,0<\v<1\}\subset\H$ be a family of quasi-minimizers  satisfying \eqref{HypGlobalSurQuasiMin}.
\begin{enumerate}
\item\label{CorDefPremierChampsCrit1} If   for sufficiently small $\v$ we have $d=0$ then $[\h-H_{c_1}]^+\to0$.
\item\label{CorDefPremierChampsCrit2} If  for sufficiently small $\v$ we have $d>0$ then $[\h-H_{c_1}]^-\to0$.
\end{enumerate}
\end{cor}
\begin{proof}
The corollary is a direct consequence of Corollary \ref{Cor.ExactEnergyExpPreCritField} taking $d'\in\{1,...,N_0\}$ which minimizes $\DLD_{d',0}=W_{d'}/(\Pic d')$ in \eqref{BorneSUPUNDEGENnplusPrecise} for the first assertion and $d'=0$ in \eqref{BorneINFUNDEGENnplusPrecise} for the second.
\end{proof}
\subsection{Secondary critical fields for $d\in\{1,...,N_0\}$}\label{Sec.SecondaryCriticFields}
If $N_0=1$,  if $\h$ is near $H_{c_1}$ and if $d>0$, then it is standard to prove that $d=1$. If  $N_0\geq2$ and $d\in\{1,...,N_0\}$, then the situation is more involved: we have no  {\it a priori} sharp informations about the number of vorticity defects and their [macroscopic] location. The goal of this section is to get such informations.

\subsubsection{Preliminaries}
Note that for $0\leq d<d'\leq N_0$ we have $\DLU_{d',d}=0$ and $\DLD_{d',d}=\dfrac{\Wmin_{d'}-\Wmin_{d}}{\Pic (d'-d)}$. 

Rephrasing Corollary \ref{Cor.ExactEnergyExpPreCritField} for $d,d'\in\{0,...,N_0\}$ we have the following key lemma.
\begin{lem}\label{Lem.PremEtapChamSec}
Let $\v=\v_n\downarrow0$, $\lambda$, $\delta$, $\h$ and  $((v_\v,A_\v))_\v\subset\H$ be as in Corollary \ref{Cor.ExactEnergyExp}.

Assume ${\rm Card}(\Zz)=d$ is independent of $\v$ then the following properties hold:
\begin{enumerate}
\item\label{Lem.PremEtapChamSec2} If $0\leq d'<d$ then, letting $\Wmin_0:=0$, we have $\h\geq \HoC +\dfrac{\Wmin_{d}-\Wmin_{d'}}{\Pic (d-d')}+o(1)$.

In particular taking $d'=0$ we get $\h\geq  \HoC +\dfrac{\Wmin_{d}}{\Pic d}+o(1)$.
\item\label{Lem.PremEtapChamSec3} If $d<N_0$ and $ d<d'\leq N_0$ then $\h\leq \HoC +\dfrac{\Wmin_{d'}-\Wmin_{d}}{\Pic (d'-d)}+o(1)$.

\item\label{lemQqePropQuot1} If $N_0\geq 2$, $N_0\geq d'>d\geq1$  then
\[
\dfrac{\Wmin_{d'}}{d'}<\dfrac{\Wmin_{d'}-\Wmin_{d}}{d'-d}\Longleftrightarrow\dfrac{\Wmin_{d}}{d}<\dfrac{\Wmin_{d'}}{d'}\text{ and }\dfrac{\Wmin_{d'}}{d'}>\dfrac{\Wmin_{d'}-\Wmin_{d}}{d'-d}\Longleftrightarrow\dfrac{\Wmin_{d}}{d}>\dfrac{\Wmin_{d'}}{d'}.
\]
\item\label{lemQqePropQuot2} If $N_0\geq 2$ and $N_0\geq d'>d\geq1$ then
\[
\dfrac{\Wmin_{d'}}{d'}=\dfrac{\Wmin_{d'}-\Wmin_{d}}{d'-d}\Longleftrightarrow\dfrac{\Wmin_{d}}{d}=\dfrac{\Wmin_{d'}}{d'}.
\]

\item\label{lemQqePropQuot3} If $N_0\geq 2$ and $0\leq d<d'<d''\leq N_0$ then  we have the following convex combination
\[
\dfrac{\Wmin_{d''}-\Wmin_{d}}{d''-d}=\dfrac{d''-d'}{d''-d}\dfrac{\Wmin_{d''}-\Wmin_{d'}}{d''-d'}+\dfrac{d'-d}{d''-d}\dfrac{\Wmin_{d'}-\Wmin_{d}}{d'-d}.
\]
Consequenlty $\dfrac{\Wmin_{d''}-\Wmin_{d}}{d''-d}$ is between $\dfrac{\Wmin_{d''}-\Wmin_{d'}}{d''-d'}$ and $\dfrac{\Wmin_{d'}-\Wmin_{d}}{d'-d}$.
\end{enumerate}
\end{lem}
\begin{proof}
The two first assertions are obtained with Corollary \ref{Cor.ExactEnergyExpPreCritField}. The remaining part of the lemma consists in basic calculations.
\end{proof}
\subsubsection{First step in the definition of the critical fields}\label{DefFirstCriticalFields}
Assume $N_0\geq2$. We are going to define some energetic levels [in term of $\Wmin_{d}$] related with the number of vorticity defects and their [macroscopic] location. 

 We denote $d^\star_0:=0$, $\mathscr{S}_1:=\{1,...,N_0\}$, $\Krit_1^\star:=\min_{d\in\mathscr{S}_1}\dfrac{\Wmin_{d}}{d}=\min_{d\in\mathscr{S}_1}\dfrac{\Wmin_{d}-\Wmin_{d_0^\star}}{d-d_0^\star}$,  $\mathscr{S}^\star_1:=\{d\in\mathscr{S}_1\,|\,{\Wmin_{d}}/{d}=\Krit_1^\star\}$ and  $\mathscr D_1:=\{\,{\bf D}\in \Lambda_d\,|\,d\in \mathscr{S}^\star_1\text{ and }{\bf D}\text{ minimizes } \W_{d}\}$. We let also $d_1^{\star}:=\max\mathscr{S}^\star_1$ and ${\mathscr{D}}_1^\star:={\mathscr{D}}_1\cap\tilde \Lambda_{d_1^\star}$.

If $d_1^\star=N_0$ we are going to prove that for $\h\geq H_{c_1}+o(1)$ [but $\h$ not too large], then there is exactly one vorticity defect close to each point of $\Lambda$. In the contrary case [$1\leq d_1^\star<N_0$], then there are other critical fields which govern the number of vorticity defects.

If $d_1^{\star}<N_0$, then $\mathscr{S}_2:=\{d_1^{\star}+1,...,N_0\}\neq\emptyset$. For $d\in\mathscr{S}_2$ we let $\Krit_2(d):=\dfrac{\Wmin_d-\Wmin_{d^\star_1}}{d-d^\star_1}$, $\mathscr{S}_2^\star:=\left\{d\in\mathscr{S}_2\,|\,\Krit_2(d)=\min_{\mathscr{S}_2}\Krit_2\right\}$, $d_2^\star:=\max\mathscr{S}_2^\star$ and $\Krit_2^\star:=\Krit_2(d_2^\star)$.

We denote $\mathscr D_2:=\{{\bf D}\in \Lambda_d\,|\,d\in \mathscr{S}^\star_2\text{ and }{\bf D}\text{ minimizes } \W_{d}\}$ and ${\mathscr{D}}_2^\star:=\mathscr D_2\cap \Lambda_{d_2^\star}$.

We claim that for  $d\in\mathscr{S}_2$ we have $\Wmin_d/d>\Wmin_{d_1^\star}/d_1^\star$. Then, with Lemma \ref{Lem.PremEtapChamSec}.\ref{lemQqePropQuot1}, we get   $\Krit_2(d)>\Wmin_{d_1^\star}/d_1^\star$. In particular
\begin{equation}\label{K_1<K_2}
\Krit_2^\star=\dfrac{\Wmin_{d_2^\star}-\Wmin_{d^\star_1}}{d_2^\star-d^\star_1}>\dfrac{\Wmin_{d_1^\star}}{d_1^\star}=\Krit_1^\star.
\end{equation}
If $d_2^\star=N_0$ then we stop the construction. In the contrary case, for $d\in\mathscr{S}_3:=\{d_2^\star+1,...,N_0\}\neq\emptyset$ we have $\Krit_2(d)>\Krit_2(d_2^\star)$.

We continue the iterative construction. For $k\geq2$, assume that we have $1<d_{k-1}^\star<d_k^\star<N_0$, we let   $\mathscr{S}_{k+1}:=\{d_k^\star+1,...,N_0\}\neq\emptyset$ and we assume that for  $d\in\mathscr{S}_{k+1}$:
\begin{equation}\label{RecurrencePropertyGrangemeon}
\Krit_{k}(d):=\dfrac{\Wmin_{d}-\Wmin_{d^\star_{k-1}}}{d-d^\star_{k-1}}>\dfrac{\Wmin_{d_k^\star}-\Wmin_{d^\star_{k-1}}}{d_k^\star-d^\star_{k-1}}=\Krit_{k}^\star.
\end{equation}
For $d\in\mathscr{S}_{k+1}$ we let $\Krit_{k+1}(d):=\dfrac{\Wmin_d-\Wmin_{d^\star_k}}{d-d^\star_k}$, 
\[
\mathscr{S}_{k+1}^\star:=\left\{d\in\mathscr{S}_{k+1}\,|\,\Krit_{k+1}(d)=\min_{\mathscr{S}_{k+1}}\Krit_{k+1}\right\},
\] $d_{k+1}^\star:=\max\mathscr{S}_{k+1}^\star$ and $\Krit_{k+1}^\star:=\Krit_{k+1}(d_{k+1}^\star)$.

We define also
\[
\mathscr D_{k+1}:=\{{\bf D}\,|\,{\bf D}\in \Lambda_d,\,d\in \mathscr{S}^\star_{k+1}\text{ and }{\bf D}\text{ minimizes } \W_{d}\}
\text{ and }
{\mathscr{D}}_{k+1}^\star
:=\mathscr D_{k+1}\cap\Lambda_{d_{k+1}^\star}.
\]

From \eqref{RecurrencePropertyGrangemeon} we have
\begin{equation}\label{RecurrencePropertyGrangemeonbis}
\Krit_{k}(d_{k+1}^\star)=\dfrac{\Wmin_{d_{k+1}^\star}-\Wmin_{d^\star_{k-1}}}{d_{k+1}^\star-d^\star_{k-1}}>\dfrac{\Wmin_{d_{k}^\star}-\Wmin_{d^\star_{k-1}}}{d_{k}^\star-d^\star_{k-1}}=\Krit_{k}^\star.
\end{equation}
Then, from Lemma \ref{Lem.PremEtapChamSec}.\ref{lemQqePropQuot3} with $d=d_{k-1}^\star$, $d'=d_k^\star$ and $d''=d_{k+1}^\star$, we get that $\Krit_{k}(d_{k+1}^\star)$  is between $\Krit_{k}^\star$ and $\Krit_{k+1}^\star$. Consequently, with  \eqref{RecurrencePropertyGrangemeonbis} we get
\begin{eqnarray}\label{VeryVeryGoodCondGrnag}
&&\Krit_{k+1}^\star>\Krit_{k}^\star.
\end{eqnarray}
We stop the construction at Step $L$ s.t.  $d_L^\star=N_0$. Since $1\leq d_k^\star<d_{k+1}^\star\leq N_0$, it is clear that a such  $L$ exists and $1\leq L\leq N_0$.

We then have two possibilities: $L=1$ or $L\in\{2,...,N_0\}$. If $L\geq2$ then, for $k\in\{1,...,L-1\}$, \eqref{VeryVeryGoodCondGrnag} holds. 
We also claim that  $(1,...,1)\in\mathscr D_{L}$.
\begin{lem}\label{Lem.ComparaisonEntreD_kd_k+1}
Let $k\in\{1,...,L\}$, assume that $d_{k}^\star-d_{k-1}^\star\geq2$ and fix $d_{k-1}^\star<d<d_{k}^\star$. We have
\[
\dfrac{\Wmin_{d_{k}^\star}-\Wmin_{d}}{d_{k}^\star-d}\leq \Krit_k^\star\leq \dfrac{\Wmin_{d}-\Wmin_{d^\star_{k-1}}}{d-d^\star_{k-1}}.
\]
Moreover, if $d\notin\mathscr S_k^\star$, then
\[
\dfrac{\Wmin_{d_{k}^\star}-\Wmin_{d}}{d_{k}^\star-d}\leq \Krit_k^\star< \dfrac{\Wmin_{d}-\Wmin_{d^\star_{k-1}}}{d-d^\star_{k-1}}.
\]
\end{lem}
\begin{proof}
From Lemma \ref{Lem.PremEtapChamSec}.\ref{lemQqePropQuot3}, $ \Krit_k^\star$ is between $\dfrac{\Wmin_{d}-\Wmin_{d^\star_{k-1}}}{d-d^\star_{k-1}}$ and $\dfrac{\Wmin_{d_{k}^\star}-\Wmin_{d}}{d_{k}^\star-d}$. On the other hand, from the definition of $d_k^\star$,  $\Krit_k^\star\leq \dfrac{\Wmin_{d}-\Wmin_{d^\star_{k-1}}}{d-d^\star_{k-1}}$. Clearly the first part of the lemma holds. If $d\notin\mathscr S_k^\star$ then, by  definition,  $\Krit_k^\star< \dfrac{\Wmin_{d}-\Wmin_{d^\star_{k-1}}}{d-d^\star_{k-1}}$.
\end{proof}
\subsubsection{Main result}
For $k\in\{1,...,L\}$ we let 
\begin{equation}\label{TheExprionnI}
\Critical_k:=H_{c_1}^0+\dfrac{\Krit_k^\star}{\Pic}
\end{equation}
and we let also 
\begin{equation}\label{TheExprionnII}
\Criticalbis_1:=H_{c_1}^0+\DLU_{N_0}\times\ln\HoC +\DLD_{N_0}.
\end{equation}
Recall that the $\Krit_k^\star$'s are defined in Section \ref{DefFirstCriticalFields} and $\DLU_{N_0}\&\DLD_{N_0}$ in Lemma \ref{TechLemmaDefDelta}. Note that $H_{c_1}=\Critical_1$.
\begin{prop}\label{Prop.SHarperdescriptionNonSatured}

Assume that \eqref{NonDegHyp} holds and $\lambda,\delta,\h,K$ satisfy \eqref{CondOnLambdaDelta}, \eqref{BorneKMagn} and \eqref{PutaindHypTech}.

Let $\{(v_\v,A_\v)\,|\,0<\v<1\}\subset\H$ be a family  satisfying \eqref{HypGlobalSurQuasiMin}$\&$\eqref{QuasiMinDef} which is in the Coulomb gauge. Assume $d_\v={\rm Card}(\Zz_\v)\in\{1,...,N_0\}$. 

We denote  ${\bf D}=(D_1,...,D_{N_0})$ with $D_l=\deg_{\p B(p_l,\eta_\O)}(v)$ [$\eta_\O$ is defined in \eqref{DefEtaO}].

\begin{enumerate}
\item Assume $L=1$. For sufficiently small $\v>0$ we have ${\bf D}\in\mathscr{D}_1$.

Moreover, if  $\v=\v_n\downarrow0$ is a sequence s.t.  $d_\v$ is independent of $\v$ and $ d_\v\neq N_0$ [{\it i.e.} ${\bf D}\neq(1,...,1)$]  then $\left[\h- \Critical_1\right]^+\to0$.
\item Assume $L\geq2$. For $k\in\{1,...,L-1\}$, if $d^\star_{k-1}<d_\v\leq d_k^\star$ for small $\v$ or for a sequence indexed by $\v=\v_n\downarrow0$, then
\begin{equation}\label{ChapsCritSec3}
\left[\h-\Critical_k\right]^-\to0\text{ and }\left[\h-\Critical_{k+1}\right]^+\to0.
\end{equation}
  Moreover, for sufficiently small  $\v$,  ${\bf D}\in\mathscr{D}_{k}$. And if  ${\bf D}\in\mathscr{D}_{k}\setminus\mathscr{D}^\star_{k}$ [{\it i.e.} $d^\star_{k-1}<d_\v<d_k^\star$] then
\begin{equation}\label{ChapsCritSec4}
\left[\h-\Critical_k\right]^+\to0.
\end{equation}
 \item If $d^\star_{L-1}< d_\v\leq d_L^\star=N_0$ for small $\v$ or for a sequence indexed by $\v=\v_n\downarrow0$, then
\begin{equation}\label{ChapsCritSec5}
\left[\h-\Critical_L\right]^-\to0\text{ and }\left[\h-\Criticalbis_1\right]^+\to0.
\end{equation}
Moreover, for sufficiently small $\v$, ${\bf D}\in\mathscr{D}_{L}$. And if $d_\v<N_0$ [{\it i.e} ${\bf D}\neq(1,...,1)$] then 
\begin{equation}\label{ChapsCritSec6}
\left[\h-\Critical_L\right]^+\to0.
\end{equation}
\end{enumerate}
In particular, for sufficiently small $\v$, we have ${\bf D}\in \cup_{l=1}^{L}\mathscr{D}_l$.
\end{prop}
\begin{proof}
We prove the first item arguing by contradiction. First note that if $N_0=1$ then there is nothing to prove. Assume thus $N_0\geq2\,\&\,L=1$ and let $\{(v_\v,A_\v)\,|\,0<\v<1\}$ be as in the proposition. Assume there exists  $\v=\v_n\downarrow0$ s.t. ${\bf D}\notin\mathscr{D}_1$. Up to pass to a subsequence we may assume that ${\bf D}$ is independent of $\v$.

From Corollary \ref{Cor.ExactEnergyExp}, for sufficiently small $\v$, ${\bf D}$ minimizes $\W_d$ and then, from the definition of $\mathscr{D}_1$, we get $d\notin\mathscr{S}^\star_1$. Consequently $\Wmin_{N_0}/N_0<\Wmin_{d}/d$ and thus, from Lemma \ref{Lem.PremEtapChamSec}.\ref{Lem.PremEtapChamSec3}$\&$\ref{Lem.PremEtapChamSec}.\ref{lemQqePropQuot1} [with $d'=N_0$], we get the existence of $t>0$ s.t. $\h\leq H_{c_1}-t$. This last estimate is in contradiction with Corollary \ref{CorDefPremierChampsCrit}.\ref{CorDefPremierChampsCrit2}. Thus ${\bf D}\in\mathscr{D}_1$ for sufficiently small $\v$. The rest of the first assertion is a direct consequence of $d\in\mathscr{S}^\star_1\setminus\{N_0\}$ and Lemma \ref{Lem.PremEtapChamSec}.\ref{Lem.PremEtapChamSec3}$\&$\ref{Lem.PremEtapChamSec}.\ref{lemQqePropQuot2} [with $d'=N_0$].\\

We now prove the second assertion. Assume $L\geq2$. For $k\in\{1,...,L-1\}$, if $d^\star_{k-1}<d\leq d_k^\star$, then, from Lemma \ref{Lem.PremEtapChamSec}.\ref{Lem.PremEtapChamSec2} [with $d'=d_{k-1}^\star$] and Lemma \ref{Lem.PremEtapChamSec}.\ref{Lem.PremEtapChamSec3} [with $d'=d_{k+1}^\star$], we get
\begin{equation}\label{RERCONtrzBis}
\dfrac{\Wmin_{d}-\Wmin_{d_{k-1}^\star}}{\Pic(d-d_{k-1}^\star)}+o(1)\leq\h-H_{c_1}^0\leq\dfrac{\Wmin_{d_{k+1}^\star}-\Wmin_{d}}{\Pic(d_{k+1}^\star-d)}+o(1).
\end{equation}
From the definition of $d_{k}^\star$ we have ${\Krit_k^\star}{}\leq\dfrac{\Wmin_{d}-\Wmin_{d_{k-1}^\star}}{d-d_{k-1}^\star}$ and then the lower bound in \eqref{RERCONtrzBis} gives the first convergence in \eqref{ChapsCritSec3}.

On the other hand, if $d=d_k^\star$ then, from the definition of $\Krit_{k+1}^\star$,  the upper bound in \eqref{RERCONtrzBis} gives the second convergence in \eqref{ChapsCritSec3}.

If $d\neq d^\star_k$, using Lemma \ref{Lem.PremEtapChamSec}.\ref{lemQqePropQuot3} [with $d<d_{k}^\star<d_{k+1}^\star$] we obtain that $\dfrac{\Wmin_{d_{k+1}^\star}-\Wmin_{d}}{d_{k+1}^\star-d}$  is between $\dfrac{\Wmin_{d_k^\star}-\Wmin_{d}}{d_{k}^\star-d}$ and $\Krit_{k+1}^\star$. But, from Lemma \ref{Lem.ComparaisonEntreD_kd_k+1}, we get $\dfrac{\Wmin_{d_k^\star}-\Wmin_{d}}{d_{k}^\star-d}\leq \Krit_k^\star$. Since from \eqref{VeryVeryGoodCondGrnag} we have $\Krit_{k+1}^\star>\Krit_{k}^\star$, we obtain $\dfrac{\Wmin_{d_{k+1}^\star}-\Wmin_{d}}{d_{k+1}^\star-d}\leq \Krit^\star_{k+1}$. Therefore the upper bound of \eqref{RERCONtrzBis} gives the second convergence in \eqref{ChapsCritSec3}.

We now demonstrate that, for sufficiently small  $\v$,  ${\bf D}\in\mathscr{D}_{k}$ arguing by contradiction. We assume the existence of sequence $\v=\v_n\downarrow0$ s.t. $d_{k-1}^\star<d\leq d_{k}^\star$ with $k\in\{1,...,L-1\}$, ${\bf D}$ is independent of $\v$ and ${\bf D}\notin\mathscr{D}_{k}$. From Corollary \ref{Cor.ExactEnergyExp}, ${\bf D}$ minimizes $\W_d$ and then, from the definition of $\mathscr{D}_k$, we get $d\notin\mathscr{S}^\star_k$ [then $d<d_k^\star$]. 

On the one hand, with Lemma \ref{Lem.PremEtapChamSec}.\ref{Lem.PremEtapChamSec2} [with $d'=d_{k-1}^\star$]  and  Lemma \ref{Lem.PremEtapChamSec}.\ref{Lem.PremEtapChamSec3} [with $d'=d_{k}^\star$] we have  
\[
\dfrac{\Wmin_{d}-\Wmin_{d_{k-1}^\star}}{\Pic(d-d_{k-1}^\star)}+o(1)\leq \h-\HoC\leq \dfrac{\Wmin_{d}-\Wmin_{d^\star_{k}}}{\Pic(d-d^\star_{k})}+o(1).
\] 
On the other hand, with Lemma \ref{Lem.ComparaisonEntreD_kd_k+1}, we have $\dfrac{\Wmin_{d}-\Wmin_{d^\star_{k}}}{d-d^\star_{k}}<\dfrac{\Wmin_{d}-\Wmin_{d_{k-1}^\star}}{d-d_{k-1}^\star}$. This inequality gives a contradiction.

Lemma \ref{Lem.PremEtapChamSec}.\ref{Lem.PremEtapChamSec3} [with $d'=d_k^\star$] and Lemma \ref{Lem.ComparaisonEntreD_kd_k+1} give immediately \eqref{ChapsCritSec4}.\\

We now treat the last item of the proposition and we assume $d_{L-1}^\star<d\leq d_L^\star=N_0$. From \eqref{BorneINFUNDEGENnplusPrecise} [with $d'=d_{L-1}^\star$] we get $\h-H_{c_1}^0\geq\DLD_{d,d_{L-1}^\star}+o(1)$. On the other hand, from the definition of $\Krit_L^\star$, we  get 
\begin{equation}\label{PremierEtapePergola}
\h-H_{c_1}^0\geq\dfrac{\Krit_L^\star}{\Pic}+o(1).
\end{equation}
Before ending the proof of \eqref{ChapsCritSec5} we prove that \eqref{ChapsCritSec6} holds and, for sufficiently small $\v$,  ${\bf D}\in\mathscr{D}_{L}$. Assume that there exists $\v=\v_n\downarrow0$ s.t. ${\bf D}$ is independent of $\v$ and $d_{L-1}^\star<d<N_0$.

From Lemma \ref{Lem.PremEtapChamSec}.\ref{Lem.PremEtapChamSec3} [with $d'=N_0$] we have
\begin{equation}\label{PremierEtapePergolabis}
\h- \HoC\leq\dfrac{\Wmin_{N_0}-\Wmin_{d}}{\Pic (N_0-d)}+o(1).
\end{equation}
Using \eqref{PremierEtapePergola} with \eqref{PremierEtapePergolabis} we get ${\Krit_L^\star}{}\leq{(\Wmin_{N_0}-\Wmin_{d})}/( N_0-d)$.  Lemma \ref{Lem.ComparaisonEntreD_kd_k+1} [with $d_{L-1}^\star<d<N_0$] gives $(\Wmin_{N_0}-\Wmin_{d})/(N_0-d)\leq \Krit_L^\star$. Therefore, $(\Wmin_{N_0}-\Wmin_{d})/(N_0-d)= \Krit_L^\star$ and then by combining \eqref{PremierEtapePergola} and \eqref{PremierEtapePergolabis} we deduce that, if for some sequence $\v=\v_n\downarrow0$ we have $d_{L-1}^\star< d<N_0$, then \eqref{ChapsCritSec6} holds. 

Arguing as above, [using \eqref{ExactExpEnergTest} with $\dtest=N_0$], 
one may prove that for sufficiently small $\v$ we have $d\in \mathscr{S}_L^\star$ and thus ${\bf D}\in\mathscr{D}_L$.

We  complete the proof  of \eqref{ChapsCritSec5}. Assume that $\h$ is sufficiently large in order to have $d=N_0$ [here we used \eqref{ChapsCritSec6}]. It suffices to use \eqref{BorneSUPUNDEGENnplusPrecise} [with $d=N_0$ and $d'=N_0+1$] in order to get the remaining part of \eqref{ChapsCritSec5}.

\end{proof}
\subsection{Secondary critical fields for $d\geq N_0+1$}\label{Sec.SecondaryCriticFieldsBis}
The case $d\geq N_0+1$ is easier to handle than the case $1\leq d\leq N_0$.

For $k\in\N^*$, we let
\[
\Criticalbis_k:=\HoC+\DLU_{N_0+k}\times\ln\HoC +\DLD_{N_0+k}
\]
where  $\DLU_{N_0+k}\&\DLD_{N_0+k}$ are defined in Lemma \ref{TechLemmaDefDelta}.
We have the following proposition.
\begin{prop}\label{Prop.SHarperdescriptionNonSaturedII}
Assume that \eqref{NonDegHyp} holds and $\lambda,\delta,\h,K$ satisfy \eqref{CondOnLambdaDelta}, \eqref{BorneKMagn} and \eqref{PutaindHypTech}.

Let $\{(v_\v,A_\v)\,|\,0<\v<1\}\subset\H$ be a family  satisfying \eqref{HypGlobalSurQuasiMin}$\&$\eqref{QuasiMinDef} which is  in the Coulomb gauge. 

Let $k\in\N^*$. If for a sequence $\v=\v_n\downarrow0$ we have  $d_\v=N_0+k$ then
\[
\left[\h-\Criticalbis_k\right]^-\to0\text{ and }\left[\h-\Criticalbis_{k+1}\right]^+\to0.
\] 
\end{prop}
\begin{proof}
The proposition is a direct consequence of \eqref{BorneSUPUNDEGENnplusPrecise} [with $d=N_0+k$ and $d'=N_0+k+1$] and \eqref{BorneINFUNDEGENnplusPrecise} [with $d=N_0+k$ and $d'=N_0+k-1$].

\end{proof}

\appendix
\section{Proof of Estimate \eqref{EstLoinInterfaceGradU}}\label{ProofVotation}
Consider a conformal mapping $\Phi:\D\to\O$. From a result of  Painlevé  [see Footnote \ref{NumFootNoteConformal} page \pageref{NumFootNoteConformal}], the maps $\Phi$ and $\Phi^{-1}$ may be extended  in $\overline\O$ and $\overline{\D}$ by smooth maps. Then there exists $C_\star\geq1$ s.t. 
\begin{equation}\label{BorneGradConfMapPainleve}
\|\n \Phi\|_{L^\infty(\D)},\|\n \Phi^{-1}\|_{L^\infty(\O)}\leq C_\star.
\end{equation}
Write $\tilde a_\v:=a_\v\circ \Phi$ and $\tilde U_\v:=U_\v\circ \Phi$. Since the function $\tilde U_\v$ is a minimizers of $\tilde E_\v$, the analog of $E_\v$ in $\D$, $\tilde U_\v$ is a solution of 
\[
\begin{cases}
-\Delta \tilde U=w\dfrac{\tilde U}{\v^2}(\tilde a_\v^2-|\tilde U|^2)&\text{in }\D\\\p_\nu\tilde U=0&\text{on }\S^1
\end{cases}
\] 
with $w={\rm Jac}\, \Phi$ is the Jacobian of $\Phi$.

Define $V_\v:B(0,2)\to[b^2,1]$ by
\[
V_\v(x)=\begin{cases}\tilde U_\v(x)&\text{ if }x\in\D\\\tilde U_\v({x}/{|x|^2})&\text{ if }x\in B(0,2)\setminus\D\end{cases}.
\]
Then $-\Delta V_\v=-\Delta \tilde U_\v$ in $\D$ and $-\Delta V_\v(x)=-|x|^{-4}\Delta \tilde U_\v({x}/{|x|^2})$ in $B(0,2)\setminus\D$. Thus $V_\v\in H^2(B(0,2),\C)$.

First note that if  $r\leq\v$, then \eqref{EstLoinInterfaceGradU} is given by \eqref{EstGlobGradU}. 

Let $r>\v$ and $x_0\in\O$ be s.t. $\dist(x_0,\p\o_\v)>r$.  Let $\eta:=a_\v(x_0)-V_\v$ in $B(x_0,r/2)$. From Lemma A.1 in \cite{BBH1} and \eqref{EstLoinInterfaceU} we get for $x\in B(x_0,r/4)$ :
\begin{eqnarray*}
|\n V_\v(x)|^2=|\n\eta(x)|^2&\leq& C\left(\|\Delta \eta\|_{L^\infty(B(x_0,r/2))}+\dfrac{4}{r^2}\| \eta\|_{L^\infty(B(x_0,r/2))}\right)\| \eta\|_{L^\infty(B(x_0,r/2))}
\\&\leq&\frac{C\e^{-\frac{\ab r}{2\v}}}{\v^2}.
\end{eqnarray*}
In the previous estimate the constants are independent of $\v,r$ and $x_0$. From \eqref{BorneGradConfMapPainleve} we then get  \eqref{EstLoinInterfaceGradU}.

\section{Proof of Theorem \ref{ThmBorneDegréMinGlob}}\label{SectionProofAppSSBound}
Assume that \eqref{NonDegHyp} holds and $\lambda,\delta,\h,K$ satisfy \eqref{CondOnLambdaDelta}, \eqref{BorneKMagn}  and  $\delta^2|\ln\v|\leq1$.

Consider a family of configurations $\{(v_\v,A_\v)\,|\,0<\v<1\}\subset\H$  which is in the Coulomb gauge and s.t.
\[
\F(v_\v,A_\v)\leq\inf_\H\F+\mathcal{O}(\ln|\ln\v|).
\]
We drop the subscript $\v$. From Lemma \ref{LemAuxConstructMagnPot}, we may consider $ A_v\in H^1(\O,\R^2)$ s.t. $(v,A_v)$ is in the Coulomb gauge and \eqref{MagnetiqueEq} holds.

We then have
\begin{equation}\label{Borne RERDAEI}
\F(v, A_v)\leq\F(v,A)\leq\inf_\H\F+\mathcal{O}(\ln|\ln\v|)=\mathcal{O}(\h^2).
\end{equation}
Proposition \ref{Prop.BorneInfLocaliseeSandSerf} gives the existence of $C,\v_0>0$ [independent of $\v$] s.t., for $\v<\v_0$, there exists a family of disjoint disks $\{B_i\,|\,i\in \J\}$ with $B_i=B(a_i,r_i)$ satisfying :
\begin{enumerate}
\item $\{|v|<1-|\ln\v|^{-2}\}\subset\cup B_i$
\item $\sum r_i<|\ln\v|^{-10}$,
\item writing $\rho=|v|$ and $v=\rho\e^{\imath\varphi}$ we have
\begin{equation}\tag{\ref{EstimateSS3ball}}
\dfrac{1}{2}\int_{B_i}\rho^2|\n \varphi-A|^2+| \rot(A)-\h|^2\geq\pi|d_i|(|\ln\v|-C\ln|\ln\v|),
\end{equation}
where $d_i=\deg_{\p B_i}(v)$ if $B_i\subset\O$ and $0$ otherwise.
\end{enumerate}
From now on, the notation $C$ stands for a positive constant independent of $\v$ whose value may change from one line to another.
\subsection{A substitution lemma}
As in \cite{SS2}, we first state a substitution lemma.

\begin{lem}\label{Lem.SubsSandSerf}

There exists $(\tilde v,\tilde A)\in \H$ which is in the Coulomb gauge and s.t., writing, $\rho=|v|,\, v=\rho\e^{\imath\varphi}$ and $\tilde\rho=|\tilde v|,\, \tilde v=\tilde \rho\e^{\imath\tilde \varphi}$ we have
\begin{enumerate}
\item $(\tilde v,\tilde A)$ satisfies \eqref{MagnetiqueEq} and $\tilde\rho\leq1$,
\item $\tilde\rho=1$  and $\varphi=\tilde{\varphi}$ in $\O\setminus\cup B_i$,
\item\label{Assertkjh2} $\|\rho(\n\varphi-A_v)-\tilde \rho(\n\tilde\varphi-\tilde A)\|_{L^2(\O)}^2=o(1)$,
\item\label{Assertkjh3} $\|\rot(A_v)-\rot(\tilde A)\|_{L^2(\O)}^2\leq C|\ln\v|^{-2}$,
\item $\F(\tilde v,\tilde A)\leq \F( v, A_v)+o(1)$.
\end{enumerate}
\end{lem}
Lemma \ref{Lem.SubsSandSerf} is proved in \cite{SS2} [Lemma 1] for $\alpha\equiv1$. The adaptation to our case is presented below.

\begin{proof}[Proof of Lemma \ref{Lem.SubsSandSerf}]
The proof of the lemma follows the same lines than in  \cite{SS2}. 

We define a continuous function $\chi_\v=\chi:[0,1]\to[0,1]$ by letting
\[
\begin{cases}\chi(x)=x&\text{if }0\leq x\leq1/2\\\chi(x)=1&\text{if }x\geq1-|\ln\v|^{-2}\\\chi\text{ is affine}&\text{if }1/2\leq x\leq 1-|\ln\v|^{-2}\end{cases}.
\]
We then let $\tilde v:=\dfrac{\chi(\rho)}{\rho} v\in H^1(\O,\C)$ and we let 
$\tilde A=A_{\tilde v}$ given by Lemma \ref{LemAuxConstructMagnPot}. 
Letting $\tilde h=\rot(\tilde A)$ we then get
\begin{equation}\label{MagneticGLEquationTilde}
-\n^\bot \tilde h=\alpha(\imath \tilde v)\cdot(\n \tilde v-\imath\tilde A \tilde v).
\end{equation}
Exactly as in \cite{SS2} we have
\begin{eqnarray}
\label{DiferenceCourantsubstitution}
\|v\wedge\n v-\tilde v\wedge\n \tilde v\|^2_{L^2(\O)}\leq C|\ln\v|^{-2}.
\end{eqnarray}
As in \cite{SS2}, from \eqref{JaugeCoulomb}, \eqref{MagnetiqueEq} and \eqref{MagneticGLEquationTilde} we obtain PDE of the second order satisfied by $A$ and $\tilde A$.   

By considering the difference of these PDE we get
\begin{equation}\label{EquationDiffMagnPotSubst}
-\Delta(\tilde A-A)+\alpha(\tilde A-A)=\alpha(\tilde v\wedge\n \tilde v-  v\wedge\n  v)+\alpha(1-\rho^2)A+\alpha(1-\tilde\rho^2)\tilde A.
\end{equation}
From \eqref{EstLpA}, \eqref{Borne RERDAEI} and \eqref{DiferenceCourantsubstitution}, the RHS of  \eqref{EquationDiffMagnPotSubst} is bounded in $L^2(\O)$ by $\dfrac{C}{|\ln\v|}$.

Since $(\tilde A-A)\cdot \nu=0$ on $\p\O$, by elliptic regularity, we deduce 
Assertions \ref{Assertkjh2}$\&$\ref{Assertkjh3} of the lemma.

The end of the proof is exactly as in \cite{SS2}
\end{proof}
From now on we replace $(v,A_v)$ with $(\tilde v,\tilde A)$ and we claim that the valued disks given by Proposition \ref{Prop.BorneInfLocaliseeSandSerf} is valid for $(v,A_v)$ and $(\tilde v,\tilde A)$ and getting the conclusions of Theorem \ref{ThmBorneDegréMinGlob} for $(\tilde v,\tilde A)$ implies the same for $(v,A)$.

 In order to simplify the presentation we write $(v,A)$ instead of $(\tilde v,\tilde A)$.
\subsection{Energetic Decomposition}
We have the following lower bound:
\begin{prop}\label{PropBorneInfSS3}

 
 Let $h:=\rot(A)$, $h_0:=\Delta\xi_0=1+\xi_0$, $f:=h-\h h_0$ and let $\{B_i=B(a_i,r_i)\,|\,i\in\J\}$ be the disks given by Proposition \ref{Prop.BorneInfLocaliseeSandSerf}.  We have:
\begin{equation}
\F(v,A)\geq\h^2 \Jo+\sum \F[(v,A),B_i]+2\pi\h\sum d_i\xi_0(a_i)+\dfrac{1}{2}\int_{\O\setminus\cup B_i}|\n f|^2+\dfrac{1}{2}\int_\O f^2-o(1)
\end{equation}
where 
\begin{equation}\label{Est.EnConcentrationlocal}
\F[(v,A),B_i]\geq\pi b^2|d_i|(|\ln\v|-C\ln|\ln\v|).
\end{equation}
\end{prop}
This estimate is the starting point of the main argument of  \cite{SS2}. 
\begin{proof}[Proof of Proposition \ref{PropBorneInfSS3}]
Let  $\tilde\O:=\O\setminus\cup \overline{B_i}$. With \eqref{Est.EnConcentrationlocal} we get
\begin{equation}\nonumber
\F_{}[(v,A),\cup B_i]
\geq\pi b^2\sum_i|d_i|[|\ln\v|-C\ln|\ln\v|].
\end{equation}

On the other hand, letting $f:=h-\h h_0$ and since $\alpha|\n v-\imath Av|^2\geq |\n h|^2$, we get
\begin{eqnarray*}
&&\dfrac{1}{2}\int_{\tilde\O}\alpha|\n v-\imath Av|^2+| h-\h|^2
\\&\geq&\h^2{\Jo}+\dfrac{1}{2}\|f\|^2_{H^1(\tilde\O)}+\h\int_{\tilde\O}\n f\cdot\n(h_0-1)+f(h_0-1)+{o(1)}.
\end{eqnarray*}
Before refining the above lower bound we make some preliminary claims. We first note that from  \eqref{MagneticGLEquationTilde} we have $\|h-\h\|_{H^1(\O)}^2\leq C\|\n v-\imath Av\|_{L^2(\O)}^2=\mathcal{O}(\h^2)$. Then $
\|f\|_{H^1(\O)}^2=\mathcal{O}(\h^2)$. Consequently for $g\in\{f,h\}$ we have
\begin{equation}\label{GrosseGalereEst1}
\h\int_{\cup B_i\cap\O}|\n g\cdot\n(h_0-1)|+|g(h_0-1)|\leq C\|g\|_{H^1(\O)}\h\sum r_i=o(1).
\end{equation}
We 	also observe that
\begin{equation}\label{GrosseGalereEst3}
\int_\O -A^\bot\cdot\n(h_0-1)+h(h_0-1)=0.
\end{equation}
With \eqref{EstH4} we get $\|A\|_{L^\infty(\O)}\leq C\h$ and then [with \eqref{MagneticGLEquationTilde}]
\begin{eqnarray*}
\sum_{B_i\subset\O}\left|\int_{\p B_i}\p_\tau \varphi(h_0-h_0(a_i))\right|&=&\sum_{B_i\subset\O}\left|\int_{\p B_i}(h_0-h_0(a_i))(\alpha^{-1}\n^\bot h+ A)\cdot\tau\right|
\\&\leq&\sum_{B_i\subset\O}\left[\left|\int_{\p B_i}\alpha^{-1}(h_0-h_0(a_i)) \p_\nu h\right|+C\h r_i\right].
\end{eqnarray*}
If $B_i\subset\O$ we have
\begin{eqnarray*}
&&\left|\int_{\p B_i}\alpha^{-1}(h_0-h_0(a_i)) \p_\nu h\right|
\\&=&\left|\int_{B_i}\alpha^{-1}\n h_0\cdot\n h+(h_0-h_0(a_i))\Div(\alpha^{-1}\n h)\right|
\\&\leq&\left|\int_{B_i}(h_0-h_0(a_i))\Div[v\wedge(\n^\bot v-\imath A^\bot v)]\right|+\mathcal{O}(\h r_i)
\\&\leq&\int_{B_i}|h_0-h_0(a_i)|[2|\p_1 v\wedge\p_2v|+4|\n(|v|)|| A|+|v|^2|h|]+\mathcal{O}(\h r_i)
\\&\leq& Cr_i\h^2.
\end{eqnarray*}
And then
\begin{equation}\label{GrosseGalereEst4-bis}
\sum_{B_i\subset\O}\left|\int_{\p B_i}\p_\tau \varphi(h_0-h_0(a_i))\right|\leq C\sum_{B_i\subset\O} r_i\h^2.
\end{equation}
If $B_i\not\subset\O$, then $\|h_0-1\|_{L^\infty(B_i\cap\O)}\leq C r_i$ and 
\begin{eqnarray}\nonumber
&&\left|\int_{\p( B_i\cap\O)}(h_0-1)\p_\tau \varphi \right|
\\\nonumber&\leq&\int_{B_i\cap\O}\left|\n h_0\cdot \n h\right|+|h_0-1|\left[2|\p_1 v\wedge\p_2v|+4|\n(|v|)|| A|+|v|^2|h|\right]
\\\label{GrosseGalereEst4-bisbis}&\leq&C r_i\h^2.
\end{eqnarray}
By combining \eqref{GrosseGalereEst4-bis} with \eqref{GrosseGalereEst4-bisbis} we deduce:
\begin{equation}\label{GrosseGalereEst4}
\sum\int_{\p B_i\cap\O}(h_0-1)\p_\tau \varphi=2\pi\sum d_i (h_0(a_i)-1)+o(1).
\end{equation}
We used that if $B_i\not\subset\O$ then $d_i=0$.

We end the preliminary claims by noting that
\begin{equation}\label{GrosseGalereEst5}
\int_{\O}|\alpha^{-1}-1||\n h\cdot\n(h_0-1)|\leq C\h\|\alpha^{-1}-1\|_{L^2(\O)}=o(\h^{-1}).
\end{equation}
On the one hand, since $-\Delta f+f=-\Delta h+h$, we have with \eqref{GrosseGalereEst1}, \eqref{GrosseGalereEst3}, \eqref{GrosseGalereEst4}, \eqref{GrosseGalereEst5} and integrations by parts:
\begin{eqnarray*}
\int_{\tilde\O}\n f\cdot\n(h_0-1)+f(h_0-1)
&=&\int_{\O}\alpha^{-1}\n h\cdot\n(h_0-1)+h(h_0-1)+o(\h^{-1})
\\
&=&o(\h^{-1})+\sum_i\int_{\p B_i}\p_\tau \varphi(h_0-1)
\\
&=&o(\h^{-1})+2\pi\sum_{B_i\subset\O}d_i [h_0(a_i)-1]
\\&=&o(\h^{-1})+2\pi\sum_{B_i\subset\O}  d_i \xi_0(a_i).
\end{eqnarray*}
On the other hand, since $\|f\|_{L^4(\O)}\leq C\h$, we get $\displaystyle\int_{\cup B_i} f^2=o(\h^{-4})$, and this estimate ends the proof.

\end{proof}

\subsection{Estimate related with the signs of the $d_i$'s}

By Proposition \ref{PropBorneInfSS3} we have:
\begin{eqnarray}\nonumber
0&\geq&\pi b^2\sum_i|d_i|(|\ln\v|-C\ln|\ln\v|)+2\pi\h\sum_i d_i\xi_0(a_i)+
\\\label{FirstEstSSBo}&&\phantom{kqjsdffnkdkdkdkd}+\dfrac{1}{2}\int_{\O\setminus\cup B_i}|\n f|^2+\dfrac{1}{2}\int_\O f^2-o(1).
\end{eqnarray}
Denote $I_+:=\{i\in \J\,|\,d_i>0\}$, $I_-:=\{i\in \J\,|\,d_i<0\}$, $D:=\sum_\J |d_i|$, $D_+:=\sum_{i\in I_+}d_i$ and  $D_-:=\sum_{i\in I_-}|d_i|$. 

With \eqref{FirstEstSSBo}  we obtain $2\h D_+\|\xi_0\|_{L^\infty(\O)}\geq b^2 D|\ln\v|\left(1-\frac{C\ln|\ln\v|}{|\ln\v|}\right)+o(1)$ 
and then:
\begin{equation}\label{SecEstSSBo}
D_-\leq D_+\times\frac{C\ln|\ln\v|}{|\ln\v|}+o(1).
\end{equation}
\subsection{Estimate related with $\dist(a_i,\Lambda)$} 

From Lemma \ref{Lem.DescriptionLambda}, there exist $\eta>0$ and $M\geq 1$ s.t., for $a\in\O$, $\xi_0(a)\geq \min\xi_0+\eta\dist(a,\Lambda)^M$.

We let $I_0:=\{i\in I\,|\,\dist(a_i,\Lambda)<|\ln\v|^{-\frac{1}{2M}}\}$  and $D_0:=\sum_{i\in I_0}|d_i|$. 

If $i\notin I_0$, then $|\xi_0(a_i)|\leq\|\xi_0\|_{L^\infty(\O)}-\dfrac{\eta}{\sqrt{|\ln\v|}}$. We thus have
\begin{eqnarray*}
\left|\sum_i d_i\xi_0(a_i)\right|&\leq&\left|\sum_{i\in I_0} d_i\xi_0(a_i)\right|+\left|\sum_{i\notin I_0} d_i\xi_0(a_i)\right|
\\&\leq&D_0\|\xi_0\|_{L^\infty(\O)}+(D-D_0)\left(\|\xi_0\|_{L^\infty(\O)}-\dfrac{\eta}{\sqrt{|\ln\v|}}\right)
\\&\leq&D\|\xi_0\|_{L^\infty(\O)}-(D-D_0)\dfrac{\eta}{\sqrt{|\ln\v|}}.
\end{eqnarray*}
From  \eqref{FirstEstSSBo} we may deduce
\begin{equation}\nonumber
2\h\left(D\|\xi_0\|_{L^\infty(\O)}-(D-D_0)\dfrac{\eta}{\sqrt{|\ln\v|}}\right)\geq b^2D(|\ln\v|-C\ln|\ln\v|)-o(1)
\end{equation}
and consequently
\begin{equation}\label{3EstSSBo}
D-D_0\leq CD\dfrac{\ln|\ln\v|}{\sqrt{|\ln\v|}}+o(1).
\end{equation}
\subsection{Estimate of the two last terms in \eqref{FirstEstSSBo}}


We let $t\geq |\ln\v|^{-\frac{1}{2M}}\geq{|\ln\v|}^{-1/2}$ and then $t\geq\delta$ since $\delta|\ln\v|^{1/2}\leq 1$.

On the one hand, from Lemma E.1 in \cite{Publi4}, by denoting $\Cr_t$ a circle with radius $t$ we get:
\begin{equation}\label{EstmSur"grand"cercle}
\int_{\Cr_t\cap\O}(1-\alpha^{-1})=\int_{\Cr_t\cap\O}|1-\alpha^{-1}|\leq C_b\lambda t.
\end{equation}
We assume now that the center of  $\Cr_t$ is in $\Lambda$ and $t$ is s.t. $\Cr_t\subset\tilde\O=\O\setminus\cup \overline{B_i}$. We denote also $B_t$ the disk bounded by  $\Cr_t$. On $\Cr_t$ we have $|v|=1$ and then $v=\e^{\imath\varphi}$ with $\varphi$ locally defined.

By direct calculations, we have [with $f=h-\h h_0$, $\nu$ the outward normal unit vector to  $\Cr_t$ and $\tau=\nu^\bot$ ]:
\[
\int_{\Cr_t}\alpha^{-1}\p_\nu h=-\int_{\Cr_t}[\p_\tau\varphi-A\cdot\tau]=-2\pi d_t+\int_{B_t}h\text{ with $d_t:=\deg_{\Cr_t}(v)$.}
\]

On the other hand $\displaystyle\int_{\Cr_t}\alpha^{-1}\p_\nu h_0=\int_{B_t} h_0+\int_{\Cr_t}(\alpha^{-1}-1)\p_\nu h_0$. Note that
\[
\left|\int_{\Cr_t}(\alpha^{-1}-1)\p_\nu h_0\right|\leq\|\n h_0\|_{L^\infty(\O)}\int_{\Cr_t}|1-\alpha^{-1}|\leq C_b\lambda t\|\n h_0\|_{L^\infty(\O)}.
\]
Then for  $\v>0$ sufficiently small:
$\displaystyle-\int_{\Cr_t}\alpha^{-1}\p_\nu f+\int_{B_t}f\geq 2\pi d_t-C\lambda\h t$. 
Consequently we obtain
\begin{equation}\nonumber
\displaystyle2\int_{\Cr_t}\alpha^{-2}\int_{\Cr_t}|\p_\nu f|^2+2\pi t^2\int_{B_t}f^2\geq4\pi^2 d_t^2-Ct\lambda\h |d_t|
\end{equation}
 and thus, by denoting $m_t:=\di\int_{\Cr_t}\alpha^{-2}$, we get
 \begin{equation}\nonumber
\dfrac{1}{2}\int_{\Cr_t}|\p_\nu f|^2+\dfrac{\pi t^2}{2m_t}\int_{B_t}f^2\geq\dfrac{\pi^2 d_t^2}{m_t}-\dfrac{Ct\lambda\h |d_t|}{m_t}.
\end{equation}
Since $2\pi t\leq m_t\leq b^{-4}2\pi t$, for sufficiently $\v>0$ small we obtain
 \begin{equation}\label{MajEstBorneVorticite}
\dfrac{1}{2}\int_{\Cr_t}|\p_\nu f|^2+\dfrac{t}{4}\int_{B_t}f^2\geq b^4 \dfrac{\pi d_t^2}{2t}-C\lambda\h |d_t|\geq b^4 \dfrac{\pi d_t^2}{4t}.
\end{equation}
Following exactly the argument in \cite{SS2} we get
\[
\dfrac{1}{2}\int_{\O\setminus\cup B_i}|\n f|^2+\dfrac{1}{2}\int_\O f^2\geq C'D^2\ln|\ln\v|+o(1).
\]
With \eqref{FirstEstSSBo} and $\xi_0(a_i)\leq-\|\xi_0\|_{L^\infty(\O)}$ there are $C_1,C_2>0$ [ independent of $\v$] s.t.
\[
(C_1 D^2-C_2D)\ln|\ln\v|\leq g(\v)\text{ with }g(\v)\to0\text{ for }\v\to0.
\]
This estimate implies $D\leq \dfrac{C_1}{C_2}$. Therefore with \eqref{SecEstSSBo} and \eqref{3EstSSBo} we get the three first assertion of the theorem.

It remains to get \eqref{CrucialBoundedkjqbsdfbn} whose  proof  follows the same lines as in \cite{SS2} [Section 4].
\section{Proof of Proposition \ref{Docmpen}}\label{Sec.PreuveDocmpen}
Let $\borneh>1$, $(v_\v)_{0<\v<1}\subset H^1(\O,\C)$, $(\h)_{0<\v<1}\subset(0,\infty)$ and $(\xi_\v)_{0<\v<1}\subset H^1_0\cap H^2\cap W^{1,\infty}(\O,\R)$ be s.t. \eqref{AbsNatBorneuh} and \eqref{BorneXiPourLaDec} hold. For simplicity of the presentation we omit the index $\v$. 

Let $\{(B(a_i,r_i),d_i)\,|\,i\in \J\}$ be as in the proposition and write $B_i:=B(a_i,r_i)$.

In this proof the letter "$C$" stands for a quantity bounded by a power of  $\borneh$ whose value may differ from one line to another.

We let $A=\n^{\bot }\xi$ and $\tilde\O:=\begin{cases}\O\setminus\cup_{}\overline{B_i}&\text{if }|v|\not>1/2\text{ in }\O\\\O&\text{if }|v|>1/2\text{ in }\O\end{cases}$. 
The heart of the proof consists in estimating the quantity $\int_\O(v\wedge\n v)\cdot A$ in \eqref{EGALITEdenettoyga}.


We first  get with the help of \eqref{AbsNatBorneuh} and \eqref{BorneXiPourLaDec} that if $|v|\not>1/2$ in $\O$ then $\int_{\cup B_i} v\wedge\n v\cdot A=o(1)$.

We also claim that, letting $w:=v/|v|$ in $\tilde\O$: $\int_{\tilde\O}(v\wedge\n v-w\wedge\n w)\cdot A=o(1)$.

In particular, if $|v|>1/2$ in $\O$ then we have $\int_\O(v\wedge\n v)\cdot A=o(1)$. We thus assume that $|v|\not>1/2$ in $\O$. 

Then, with an integration by part we get
\begin{eqnarray}\nonumber
&&-\int_\O v\wedge\n v\cdot A
\\\nonumber&
=&-\sum_{B_i\subset\O}\left\{\xi(a_i)\int_{\p B_i} (w\wedge\n^\bot w)\cdot \nu+\int_{\p B_i} {(\xi-\xi(a_i))}(w\wedge\n^\bot w)\cdot \nu\right\}+
\\\label{DevEnergyChaton1}&&\phantom{ughfhffffhflfllflflfgjgdvjdjdn}+\sum_{B_i\not\subset\O}\int_{\p (B_i\cap\O)}\xi(w\wedge\n^\bot w)\cdot \nu.
\end{eqnarray}
For $B_i\subset\O$ we immediately have :
\begin{equation}\label{DevEnergyChaton2}
\int_{\p B_i} (w\wedge\n^\bot w)\cdot \nu=-2\pi d_i.
\end{equation}

We now define $u:=\begin{cases}v&\text{in }\tilde\O\\u_i&\text{in }B_i\cap\O\end{cases}$ where $u_i$ is the harmonic extension of $\tr_{\p (B_i\cap\O)}(v)$ in $B_i\cap\O$. By the Dirichlet principle we have for all $i$:
\begin{equation}\label{DirPrincui}
\|\n u\|_{L^2(B_i\cap\O)}\leq\|\n v\|_{L^2(B_i\cap\O)}=\mathcal{O}(|\ln \v |).
\end{equation}

It is easy to check that  $(w\wedge\n^\bot w)\cdot \nu=|u|^{-2}(u\wedge\n^\bot u)\cdot \nu$ on $\cup_i\p B_i$. For $i\in\J$, we let 
\[
f_i=
\begin{cases}
\xi-\xi(a_i)&\text{if }B_i\subset\O\\\xi&\text{if }B_i\not\subset\O
\end{cases}\in  H^2\cap W^{1,\infty}(B_i\cap\O).
\]
From \eqref{BorneXiPourLaDec} we get
\begin{equation}\label{Grad-LinftyBiEstXi}
\|\n f_i\|_{L^\infty(B_i\cap\O)}\leq C|\ln\v|.
\end{equation}


Our goal is now to estimate $\int_{\p (B_i\cap\O)} f_i(w\wedge\n^\bot w)\cdot \nu$. We first consider the case where $i\in\J$ is s.t. $|u|\geq1/2$ in $B_i\cap\O$. In this case we may write in  $B_i$: $u=|u|\e^{\imath\phi}$ with $\phi\in H^1(B_i,\R)$, $\|\phi\|_{H^1(B_i)}\leq C|\ln\v|$. We then have with \eqref{Grad-LinftyBiEstXi} and an integration by parts
\begin{equation}\label{DevEnergyChaton3}
\left|\int_{\p (B_i\cap\O)} f_i(w\wedge\n^\bot w)\cdot \nu\right|
\leq\|\n f_i\|_{L^2(B_i\cap\O)}\|\n\phi\|_{L^2(B_i\cap\O)}\leq C|\ln\v|^2 r_i.
\end{equation}
We now assume $i\in\J$ is s.t. $|u|\not\geq1/2$ in $B_i\cap\O$. By smoothness of  $|u_i|^2\in C^\infty(B_i\cap\O,\R)$, there exists $t_i\in]1/5,1/4[$, a regular value of $|u_i|^2$, s.t. $\o_i:=\{|u_i|^2<t_i\}\neq\emptyset$. We denote $D_i:=\O\cap[B_i\setminus\overline\o_i$]. Since $|u|^2\geq1/4$ on $\p B_i\cap\O$ we have $\p D_i=(\p B_i\cap\O)\cup\p\o_i\cup(\p\O\cap \overline{D_i})$.

Letting $W:=\dfrac{u}{|u|}\wedge\n^\bot\left(\dfrac{u}{|u|}\right)$ we then get
\begin{eqnarray}\label{DevEnergyChaton4}
\int_{\p D_i} f_iW\cdot \nu
=\int_{D_i} f_i\Div (W)+\n f_i\cdot W.
\end{eqnarray}
It is standard to check that $\Div\left(W\right)=0$ in $D_i$. Moreover:
\begin{equation}\nonumber
\left|\int_{D_i}\n f_i\cdot W\right|\leq2\|\n\xi\|_{L^2(B_i\cap\O)}\|\n u\|_{L^2(B_i\cap\O)}\leq C|\ln\v|^2 r_i.
\end{equation}
Consequently using \eqref{DevEnergyChaton4} we may deduce
\begin{equation}\label{DevEnergyChaton5}
\left|\int_{\p D_i} f_i\,W\cdot \nu\right|\leq C|\ln\v|^2 r_i.
\end{equation}
On the other hand, from \eqref{Grad-LinftyBiEstXi}, $\xi=0$ on $\p\O$ and $\Div\left(u\wedge\n^\bot u\right)=-2\p_1 u\wedge\p_2 u$ in $B_i\cap\O$, we get
\begin{eqnarray}\nonumber
&&\left|\int_{\p D_i} f_i\,W\cdot \nu-\int_{\p B_i\cap\O} f_i(w\wedge\n^\bot w)\cdot \nu\right|
\\\nonumber&=&\left|\int_{\p \o_i} f_i\,W\cdot \nu\right|
\\\label{DevEnergyChaton6}&=&\dfrac{1}{t_i}\left|\int_{\o_i} -2f_i\p_1 u\wedge\p_2 u+\n f_i\cdot\left(u\wedge\n^\bot u\right)\right|\leq C|\ln\v|^3 r_i.
\end{eqnarray}
We may conclude by using \eqref{DevEnergyChaton1}, \eqref{DevEnergyChaton2}, \eqref{DevEnergyChaton3}, \eqref{DevEnergyChaton5} and \eqref{DevEnergyChaton6}:
\begin{equation}\nonumber
-\int_\O v\wedge\n v\cdot A=2\pi\sum_{B_i\subset\O}d_i\xi(a_i)+o(1).
\end{equation}
The rest of the proof is exactly the same than in  \cite{S1}.

\section{Proof of some results of Section \ref{UnimSection}}
\subsection{Proof of Proposition \ref{MinimalMapHomo}}\label{PreuvePropUniModComp}
We use the same notation than in Proposition \ref{MinimalMapHomo}. In this proof, the letter $C$ is a quantity which depends only on  $\O$, $N$ and $\sum_i|d_i|$, its value may change from one line to another.

We argue as in \cite{LR1}. We let $\Pstar\in \cap_{0<p<2}W^{1,p}(\O,\R)\cap H_{\rm loc}^1(\O\setminus\{z_1,...,z_N\},\R)$ be the unique solution of 
\[
\begin{cases}
\Delta\Pstar=2\pi\sum_{i=1}^Nd_i\delta_{z_i}&\text{in }\O
\\
\Pstar=0&\text{on }\p\O
\end{cases}.
\]
and let $\Phi_\Rad\in H^1(\O_\Rad,\R)$  be the unique solution of
\begin{equation}\label{ConjaHarmmojqsdhfhfh}
\begin{cases}
\Delta\Phi_\Rad=0&\text{in }\O_\Rad
\\
\Phi_\Rad=0&\text{on }\p\O
\\
\Phi_\Rad={\rm Cte}_i&\text{on }\p B(z_i,\Rad)
\\
\di\int_{\p B(z_i,\Rad)}\p_\nu\Phi_\Rad=2\pi d_i&\text{for all }i\in\{1,...,N\}
\end{cases}.
\end{equation}

We then have $\n^\bot\Pstar=\wstar\wedge\n \wstar$ and $\n^\bot\Phi^\zd_\Rad=w^\zd_\Rad\wedge\n w^\zd_\Rad$. It is important to note that if  $w\in H^1(\O_\Rad,\S^1)$, then $|\n w|=|w\wedge\n w|$.

We may decompose $\Pstar$ as $\Pstar=\sum_i d_i\Phi_{z_i}$ where, for $z\in\O$, $\Phi_{z}$ is the unique solution of
\[
\begin{cases}
\Delta\Phi_{z}=2\pi\delta_{z}&\text{in }\O
\\
\Phi_{z}=0&\text{on }\p\O
\end{cases}.
\]
With a standard pointwise bound for the gradient of an harmonic function [see (2.31) in \cite{GT}] we have 
$\|\n\Phi_{z_i}\|_{L^\infty(\O\setminus\overline{B(z_i,\Rad)})}\leq C\dfrac{\|\Phi_{z_i}\|_{L^\infty(\O\setminus\overline{B(z_i,\Rad/4)})}}{\Rad}$. 
Thus
\begin{equation}\label{PremereBornePutaindeGrad}
\|\n\Pstar\|_{L^\infty(\O_{\Rad})}\leq C\dfrac{\sum_i|d_i|\|\Phi_{z_i}\|_{L^\infty(\O_{\Rad/4})}}{\Rad}.
\end{equation}
Moreover, it is easy to check that $\Phi_{z_i}=\ln|x-z_i|+R_{z_i}$ where $R_{z_i}$ is the harmonic extension of $-\ln|x-z_i|_{|\p\O}$. From \eqref{PremereBornePutaindeGrad} and by the maximum principle we get 
for $\Rad<\min\left\{[{\rm diam}(\O)]^{-1};1/4\right\}$
 \begin{equation}\label{AjoutNUMA}
 |\n\Pstar|\leq \dfrac{C(1+|\ln\Rad|)}{\Rad}\text{ in }{\O_\Rad}
 \end{equation}  which proves \eqref{BorneGradWstar}.

If there is $\eta>0$ s.t. $\tae>\eta$, then  $\|R_{z_i}\|_{C^1(\O)}\leq C_\eta$ where $C_\eta$ which depends only on  $\eta$ and $\O$. We thus get $\|\n \Pstar\|_{L^\infty(\O_\Rad)}\leq\dfrac{\tilde{C}_\eta}{\Rad}$ [where $\tilde{C}_\eta$ depends only on $\eta$, $N$, $\sum|d_i|$ and $\O$] and this estimates implies \eqref{BorneGradWstarSpeciale}.

We now define 
$R_\zd:=\sum_i d_iR_{z_i}$ in order to have $\Pstar=\sum_id_i\ln|x-z_i|+R_\zd$.
 
From Lemma I.4 in \cite{BBH} we have
\begin{eqnarray}\nonumber
\|\Phi_\Rad-\Pstar\|_{L^\infty(\O_\Rad)}
\leq\sum_i\left[\sup_{\p B(z_i,\Rad)}\sum_j\ln|x-z_j|-\inf_{\p B(z_i,\Rad)}\sum_j\ln|x-z_j|\right]+\\\label{lkjbblkjn}\phantom{luljhsdgslkjdfghndfg}+\sum_i\left[\sup_{\p B(z_i,\Rad)}R_\zd-\inf_{\p B(z_i,\Rad)}R_\zd\right].
\end{eqnarray} 
If $N=1$, then the first term of the RHS in \eqref{lkjbblkjn} is $0$. Otherwise, as in \cite{S1} [Proposition 5.1], we have
\begin{equation}\label{lkjbblkjn1}
\sum_i\left[\sup_{\p B(z_i,\Rad)}\sum_j\ln|x-z_j|-\inf_{\p B(z_i,\Rad)}\sum_j\ln|x-z_j|\right]\leq\dfrac{C\Rad}{\min_{i\neq j}|z_i-z_j|}.
\end{equation}
And for $i\in\{1,...,N\}$, by harmonicity of  $R_\zd$, for $0<\rho<\dfrac{\tae}{2}$ we get
\begin{equation}\label{lkjbblkjn0}
\|\n R_\zd\|_{L^\infty(B(z_i,\rho))}\leq\dfrac{C\|R_\zd\|_{L^\infty(\O)}}{\dist(z_i,\p\O)-\rho}\leq C\dfrac{1+|\ln(\tae)|}{\tae}.
\end{equation}
Then
\begin{equation}\label{lkjbblkjn2}
\sum_i\left[\sup_{\p B(z_i,\Rad)}R_\zd-\inf_{\p B(z_i,\Rad)}R_\zd\right]\leq C\dfrac{\Rad(1+|\ln(\tae)|)}{\tae}.
\end{equation}
We let
\begin{equation}\label{DefY}
Y:=\begin{cases}
\dfrac{\Rad(1+|\ln(\tae)|)}{\tae}&\text{if $N=1$}\\\dfrac{\Rad}{\min_{i\neq j}|z_i-z_j|}+\dfrac{\Rad(1+|\ln(\tae)|)}{\tae}&\text{if }N\geq2
\end{cases}.
\end{equation}
By combining \eqref{lkjbblkjn}, \eqref{lkjbblkjn1} and \eqref{lkjbblkjn2} we get
\begin{equation}\label{lkjbblkjn3}
\|\Phi_\Rad-\Pstar\|_{L^\infty(\O_\Rad)}\leq CY.
\end{equation}
From \eqref{AjoutNUMA} and \eqref{lkjbblkjn3} we immediately get
\begin{eqnarray}\nonumber
0&\leq&\int_{\O_\Rad}|\n\Pstar|^2-|\n\Phi_\Rad|^2+|\n(\Pstar-\Phi_\Rad)|^2
\\\label{Borjhpresquefinimlh}&\leq& C\,Y\Rad \max_i\|\p_\nu\Pstar\|_{L^\infty(\p B(z_i,\Rad))}.
\end{eqnarray}
On the other hand, for  $i\in\{1,...,N\}$, we have [with \eqref{lkjbblkjn0}]
\begin{equation}\label{BorneDerNormPstarCercle}
\|\p_\nu\Pstar\|_{L^\infty(B(z_i,\Rad))}
\leq C\left(\dfrac{1}{\Rad}+ \dfrac{1+|\ln(\tae)|}{\tae}\right).
\end{equation}
Using $X$ defined in  \eqref{DefX}, from  \eqref{Borjhpresquefinimlh} and \eqref{BorneDerNormPstarCercle}, we get
\begin{equation}\label{ConjaHarmmojqsdhfhfh222}
0\leq\int_{\O_\Rad}|\n\Pstar|^2-|\n\Phi_\Rad|^2+|\n(\Pstar-\Phi_\Rad)|^2\leq CX.
\end{equation}
From \eqref{ConjaHarmmojqsdhfhfh222} we deduce \eqref{ConvergenceShrHolklgBorne} and since $\int_{\p\O}(\varphi_\star-\varphi_\Rad)=0$, with a Poincaré inequality we obtain \eqref{ConvergenceH1ShrHolklgBorne}.
\subsection{Proof of Proposition \ref{Prop.EnergieRenDef}}\label{PreuvelammeShrinkSerfaty}

Let $\zd=\zd^{(n)}\in\Ostar\times\Z^N$ and denote $\tae:=\min_i\dist(z_i,\p\O)>0$. Assume that $d_1,...,d_N$ are independent of $n$. Let $\Rad=\Rad_n\to0$ be s.t 
\eqref{HypRayClass} holds.

In this proof the letter $C$ stands for a quantity which depends only on  $\O$, $N$, $C_1$ and $\sum_i|d_i|$, its value may change from one line to another.

By Remark \ref{Remark.DefConjuHarmPhase} and an integration by parts we have
\begin{equation}\label{PremiereEstPourlkjhkluklu1}
\dfrac{1}{2}\int_{\O_\Rad}|\n \wstar|^2=\dfrac{1}{2}\int_{\O_\Rad}|\n \Pstar|^2=-\dfrac{1}{2}\sum_i\int_{\p B(z_i,\Rad)}\Pstar\p_\nu\Pstar.
\end{equation}
For $i_0\in\{1,...,N\}$, we fix $x_{i_0}\in\p B(z_{i_0},\Rad)$. Then [with $\n^\bot\Pstar=\wstar\wedge\n\wstar$]
\begin{eqnarray}\nonumber
&&\int_{\p B(z_{i_0},\Rad)}\Pstar\p_\nu\Pstar
\\\label{DecompParInciohj}&=&\int_{\p B(z_{i_0},\Rad)}\left[\Pstar-\Pstar(x_{i_0})\right]\p_\nu\Pstar+2\pi d_{i_0}\Pstar(x_{i_0}).
\end{eqnarray}
On the one hand, arguing as in the proof of \eqref{lkjbblkjn3}, we get for $z\in \p B(z_{i_0},\Rad)$ :
\[
|\Pstar(z)-\Pstar(x_{i_0})|\leq\sup_{\p B(z_{i_0},\Rad)}\Pstar-\inf_{\p B(z_{i_0},\Rad)}\Pstar\leq CY.
\]
Then, using \eqref{BorneDerNormPstarCercle}, we obtain
\begin{equation}\label{Premirelkjhqsdjdjdjd1}
\sum_i\left|\int_{\p B(z_i,\Rad)}\left[\Pstar-\Pstar(x_i)\right]\p_\nu\Pstar\right|\leq CX.
\end{equation}
On the other hand, for  $i_0\in\{1,...,N\}$
\begin{eqnarray*}
\Pstar(x_{i_0})-R_\zd(z_{i_0})
=-d_{i_0}|\ln\Rad|+\sum_{j\neq i_0} d_j\ln|x_{i_0}-z_j|+\left[R_\zd(x_{i_0})-R_\zd(z_{i_0})\right],
\end{eqnarray*}
and with \eqref{lkjbblkjn0} we get $\di\left|R_\zd(x_{i_0})-R_\zd(z_{i_0})\right|\leq\dfrac{C(1+|\ln\tae|)\Rad}{\tae}$. We then immediately get:
\begin{equation}\label{NUMBBBBGBGBGB}
\Pstar(x_{i_0})=R_\zd(z_{i_0})-d_{i_0}|\ln\Rad|+\sum_{j\neq i_0} d_j\ln|z_{i_0}-z_j|+\mathcal{O}(X).
\end{equation}

With \eqref{DecompParInciohj}, \eqref{Premirelkjhqsdjdjdjd1} and \eqref{NUMBBBBGBGBGB} we may prove that \eqref{PremiereEstPourlkjhkluklu1} may be rewritten into
\begin{equation*}
\dfrac{1}{2}\int_{\O_\Rad}|\n \wstar|^2
=\pi\sum_i\left[d_{i}^2|\ln\Rad|-d_iR_\zd(z_{i})\right]-\pi\sum_{j\neq i}d_jd_j\ln|z_{i}-z_j|+\mathcal{O}(X)
\end{equation*}
where  "$\mathcal{O}(X)$" is quantity bounded by $CX$ with $C$ depending only on $N,\O$ and $\sum|d_i|$.


\subsection{Proof of Proposition \ref{Prop.ConditionDirEnergieRen}}\label{Sec.PreuvelammeShrinkSerfatyDir}

Let $\zd=\zd^{(n)}\in\Ostar\times\Z^N$, $\Rad\downarrow0$ and $\eta>0$ be as in the proposition.

In this proof the letter $C$ stands for a quantity which depends only on  $\O$, $N$ and $\sum_i|d_i|$, its value may change from one line to another.

We first claim that, for  $i\neq j$, $B(z_i,\eta)\cap B(z_j,\eta)\neq\emptyset$, $B(z_i,\eta)\subset\O$ and $\eta=\chi\Rad$ with $\chi\to\infty$. In particular we assume $n$ sufficiently large to have $\eta>\Rad$.

Since $\n^\bot\Pstar=\wstar\wedge\n\wstar$, for $i_0\in\{1,...,N\}$ and $z\in\O\setminus\{z_1,...,z_N\}$, we have
\[
\wstar\wedge\n\wstar(z)=d_{i_0}\n^\bot(\ln|z-z_{i_0}|)+\n^\bot\left[R_\zd(z)+\sum_{j\neq i_0} d_j\ln|z-z_j|\right].
\]
For $j\in\{1,...,N\}$, let $\theta_j$ be the main determination of the argument of  $\dfrac{z-z_j}{|z-z_j|}$ and let $\mathcal{R}$ be an harmonic conjugate of  $R_\zd$. In $\O\setminus\{z_1,...,z_N\}$ we have
\[
\wstar\wedge\n\wstar-d_{i_0}\n \theta_{i_0}=\n\left[\sum_{j\neq i_0} d_j\theta_j+\mathcal{R}\right].
\]
Then for $z\in B(z_{i_0},\eta)\setminus\{z_{i_0}\}$ we have $\wstar(z)=\left(\dfrac{z-z_{i_0}}{|z-z_{i_0}|}\right)^{d_{i_0}}\e^{\imath \varphi_{i_0}(z)}$ with $\varphi_{i_0}=\sum_{j\neq i_0} d_j\tilde\theta_j+\mathcal{R}+{\rm Cte}_{i_0}$ where, for $j\neq i_0$, $\tilde{\theta}_j$ is a determination of the argument of $\dfrac{z-z_i}{|z-z_i|}$ which is globally defined in  $B(z_{i_0},\eta)$. Note that $\varphi_{i_0}\in H^1(B(z_{i_0},\eta),\R)$.

On the other hand, by direct calculations, we have
$\left\|\sum_{j\neq i_0} d_j\n \tilde\theta_j\right\|_{L^\infty(B(z_{i_0},\eta))}\leq \dfrac{C}{\eta}$
and, since $R_\zd$ is harmonic, we also have from the definition of $\mathcal{R}$
\[
\|\n\mathcal{R}\|_{L^\infty(B(z_{i_0},\eta))}=\|\n R_\zd\|_{L^\infty(B(z_{i_0},\eta))}\leq C\dfrac{\| R_{\zd}\|_{L^\infty(\O)}}{\dist(B(z_{i_0},\eta),\p\O)}\leq C\dfrac{|\ln(\tae)|+1}{\tae}.
\]
We thus deduce
\begin{equation}\label{BorneDephWstarBouleEta}
\|\n\varphi_{i_0}\|_{L^\infty(B(z_{i_0},\eta))}\leq C\left(\dfrac{1+|\ln(\tae)|}{\tae}+\dfrac{1}{\eta}\right).
\end{equation}
We switch to polar coordinates by letting  for  $i\in\{1,...,N\}$ and $\rho\in]\Rad,\eta[$, $\tilde{\varphi}_{i}(\rho,\theta):=\varphi_{i}(z_i+\rho\e^{\imath\theta})$. We then get, by \eqref{BorneDephWstarBouleEta} and a mean value argument,
 the existence of  $\rho_n\in]\sqrt{\chi}\Rad,\eta[$ s.t.
\[
\sum_i\int_0^{2\pi}|\p_\theta\tilde\varphi_i(\rho_n,\theta)|^2\,{\rm d}\theta\leq\dfrac{C}{\ln\chi}\left[\dfrac{\eta(|\ln(\tae)|+1)}{\tae}+1\right]^2.
\]
We let $Z:=\dfrac{1}{\ln\chi}\left[\dfrac{\eta(|\ln(\tae)|+1)}{\tae}+1\right]^2$ and by assumption we have $Z\to0$.

We denote, for $i\in\{1,...,N\}$, $m_i=\di\dfrac{1}{2\pi}\int_0^{2\pi}\tilde\varphi_i(\rho_n,\theta)\,{\rm d}\theta$ in order to have
\[
\int_0^{2\pi}|\tilde\varphi_i(\rho_n,\theta)-m_i|^2\,{\rm d}\theta\leq CZ.
\]
We then define $\phi_i\in H^1(B(z_i,\rho_n)\setminus\overline{B(z_i,\Rad)},\R)$ using polar coordinates:
\[
\tilde\phi_i(s,\theta)=\dfrac{s-\rho_n}{\Rad-\rho_n}m_i+\dfrac{s-\Rad}{\rho_n-\Rad}\tilde\varphi(\rho_n,\theta)\text{ with }s\in(\Rad,\rho_n).
\]
For  $z_i+s\e^{\imath\theta}\in B(z_i,\rho_n)\setminus\overline{B(z_i,\Rad)}$, we let $\phi_i(z_i+s\e^{\imath\theta}):=\tilde\phi_i(s,\theta)$. By standard calculations we get $\displaystyle\int_{B(z_i,\rho_n)\setminus\overline{B(z_i,\Rad)}}|\n \phi_i|^2\leq CZ$. 

We conclude by defining $v=\begin{cases}
\wstar&\text{in }\O\setminus\cup\overline{B(z_i,\rho_n)}\\u_i\e^{\imath\phi_i}&\text{in }B(z_i,\rho_n)\setminus\overline{B(z_i,\Rad)}
\end{cases}$ with $u_i(z)=\left(\dfrac{z-z_i}{|z-z_i|}\right)^{d_i}$. It is clear that $v\in H^1(\O_\Rad,\S^1)$ and that  for $i\in\{1,...,N\}$ we have $v(z_i+\Rad\e^{\imath\theta})={\rm Cte}_iu_i$ [with ${\rm Cte}_i=\e^{\imath m_i}$]. Note that since $\deg_{\p B(z_i,\Rad)}(\wstar)=d_i$ we have 
\[
 \dfrac{1}{2} \int_{B(z_i,\rho_n)\setminus B(z_i,\Rad)}|\n u_i|^2\leq  \dfrac{1}{2}\int_{B(z_i,\rho_n)\setminus B(z_i,\Rad)}|\n \wstar|^2
\]
and 
\[
 \dfrac{1}{2} \int_{B(z_i,\rho_n)\setminus B(z_i,\Rad)}|\n (u_i\e^{\imath\phi_i})|^2=  \dfrac{1}{2}\int_{B(z_i,\rho_n)\setminus B(z_i,\Rad)}|\n u_i|^2+\dfrac{1}{2}\int_{B(z_i,\rho_n)\setminus B(z_i,\Rad)}|\n \phi_i|^2.
\]
Consequently using \eqref{BorneDephWstarBouleEta} and $\rho_n<\eta$ we obtain
\begin{eqnarray*}
\sum_i \dfrac{1}{2}\int_{B(z_i,\rho_n)\setminus B(z_i,\Rad)}|\n v|^2
\leq\sum_i \dfrac{1}{2}\int_{B(z_i,\rho_n)\setminus B(z_i,\Rad)}|\n \wstar|^2+CZ.
\end{eqnarray*}
Thus $\displaystyle\dfrac{1}{2}\int_{\O_\Rad}|\n v|^2\leq\dfrac{1}{2}\int_{\O_\Rad}|\n \wstar|^2+CZ$. The last estimate and  \eqref{ConvergenceShrHolklgBorne} end the proof.
\section{Proof of Proposition \ref{Prop.BorneSupSimple} }\label{AppProofUpBound}
\begin{proof}
{\bf Step 1. Selection of "good" points}

Let $d\in\N^*$ and  consider ${\bf D}\in\LamN$ which minimizes \eqref{CouplageEnergieRen}.

For $k\in\{1,...,N_0\}$, if $D_k\geq1$ we let $(\tilde{z}^{(k)}_1,...,\tilde{z}^{(k)}_{D_k})\in [B(p_k,\h^{-1/4})^{D_k}]^*$ which minimizes the infimum in the left hand side of \eqref{DevMesoscopicDef} with $R=\h^{-1/4}$, $p=p_k$ and $D=D_k$.

We then have the existence of $C$ [depending only on $\O$ and $d$] s.t. $|p_k-\tilde z^{(k)}_i|\leq C \h^{-1/2}$ and if $D_k\geq2$ then $|\tilde z^{(k)}_i-\tilde z^{(k)}_j|\geq \h^{-1/2}/C$ for $i\neq j$.

We may choose [in an arbitrary way] $z_i^{(k)}\in B(\tilde z_i^{(k)},\delta)\cap[\delta(\Z\times\Z)]$. Since $\delta\sqrt \h\to0$, we still have [up to change the value $C$]  $|p_k- z^{(k)}_i|\leq C \h^{-1/2}$ and if $D_k\geq2$ then $| z^{(k)}_i-\tilde z^{(k)}_j|\geq \h^{-1/2}/C$ for $i\neq j$.

For $i\in\{1,...,D_k\}$ we let $x_i^{(k)}:=z_i^{(k)}+\lambda\delta x_0$ where $x_0\in\o$ is an arbitrary point of minimum of $W^{\rm micro}$ [defined in \eqref{DefRenMicroEn3}].\\
{\bf Step 2. Construction of the test function}

We construct test functions in subdomains of $\O$ and then we glue the test functions.
\begin{itemize}
\item We let $w^{\rm macro}_{\h}\in H^1(\O_{\h^{-1}}({\bf z}),\S^1)$ be a minimizer of $I_{\h^{-1}}^{\rm Dir}\zd$ [defined in \eqref{MinPropDir}] with ${\bf d}=(1,...,1)\in\Z^d$ and ${\bf z}\in(\O^d)^*$ is a $d$-tuple s.t. $\{z_1,...,z_d\}=\{z_i^{(k)}\,|\,k\in\{1,...,N_0\} \text{ s.t. }D_k\geq1\text{ and }i\in\{1,...,D_k\}\}$. 
\item For $k\in\{1,...,N_0\}$ s.t. $D_k\geq1$ and $i\in\{1,...,D_k\}$, we let  $w_{k,i}^{\rm micro}\in H^1[B(z_i^{(k)},\h^{-1})\setminus\overline{B(x_i^{(k)},\lambda\delta^2)},\S^1]$ be a minimizer of the right hand side of \eqref{MicroRenoExpressionNonRescalDir} with $z_\v=z_i^{(k)}$, $x_\v=x_i^{(k)}$, $R=\h^{-1}$ and $r=\lambda\delta^2$ [from \eqref{HypLambdaDeltaConstrFoncTest} we have $R/r\to\infty$].

We let also $u_{k,i}\in H^1[B(x_i^{(k)},\lambda\delta^2),\C]$ be a minimizer of
 \[
u\mapsto\dfrac{1}{2}\int_{B(x_i^{(k)},\lambda\delta^2)}|\n u|^2+\dfrac{1}{2\v^2}(1-|u|^2)^2
\] with the Dirichlet boundary condition $u(x_i^{(k)}+\lambda\delta^2\e^{\imath\theta})=\e^{\imath\theta}$.
\end{itemize}
By considering well chosen constants ${\rm Cte}_{k,i}^{(1)}$, ${\rm Cte}_{k,i}^{(2)}$ and ${\rm Cte}_{k}$, we may glue the above test functions and we  define  $v\in H^1(\O,\C)$ :
\[
v=
\begin{cases}
w^{\rm macro}_{\h}&\text{in }\O_{\h^{-1}}(\bf z)
\\
{\rm Cte}_{k}&\text{in } B(z_i^{(k)},\h^{-1})\text{ if }D_k=0
\\
{\rm Cte}_{k,i}^{(1)}w_{k,i}^{\rm micro}&\text{in }B(z_i^{(k)},\h^{-1})\setminus\overline{B(x_i^{(k)},\lambda\delta^2)}\,\left|\begin{array}{c}k\in\{1,...,N_0\}\text{ s.t. }D_k\geq1\\\text{and }i\in\{1,...,D_k\}\end{array}\right.
\\
{\rm Cte}_{k,i}^{(2)}u_{k,i}&\text{in }B(x_i^{(k)},\lambda\delta^2)\left|\begin{array}{c}k\in\{1,...,N_0\}\text{ s.t. }D_k\geq1\\\text{and }i\in\{1,...,D_k\}\end{array}\right.
\end{cases}.
\]
{\bf Step 3. The energy of the test function}

We first note that the configuration $\zd$ is s.t. $\tae({\bf z})>\dfrac{1}{2}\dist(\Lambda,\p\O)$ and for $i\neq j$ we have $\dfrac{\h^{-1}}{|z_i-z_j|}\to0$, then we may apply Propositions \ref{MinimalMapHomo}, \ref{Prop.EnergieRenDef} and \ref{Prop.ConditionDirEnergieRen}. We may also use Proposition \ref{Prop.RenEnergieCluster}. From these propositions we get
\begin{eqnarray}\nonumber
&&\dfrac{1}{2}\int_{\O_{\h^{-1}}(\bf z)}|\n v|^2
\\\label{BorneSupFoncTestShaarp1}&=&\pi d\ln\h+W_{N_0}^{\rm macro}\pD-\pi\sum_{\substack{k=1\\\text{s.t. }D_k\geq2}}^{N_0}\sum_{i\neq j}\ln|z^{(k)}_i-z^{(k)}_j|+o(1).
\end{eqnarray}
For $k\in\{1,...,N_0\}$ s.t. $D_k\geq1$ and $i\in\{1,...,D_k\}$ with \eqref{MicroRenoExpressionNonRescalDir}, \eqref{DefRenMicroEn1} and \eqref{DefRenMicroEn2} we get:
\begin{eqnarray}\nonumber
&&\dfrac{1}{2}\int_{B(z_i^{(k)},\h^{-1})\setminus\overline{B(x_i^{(k)},\lambda\delta^2)}}\alpha|\n v|^2
\\\label{BorneSupFoncTestShaarp2}&=&\pi|\ln(\lambda\delta\h)|+b^2\pi|\ln( \delta)|+ W^{\rm micro}(x_0)+o(1).
\end{eqnarray}
From Lemma IX.1 in \cite{BBH} and \eqref{EstLoinInterfaceU} [with $|\n v|\leq C\v^{-1}$], for $k\in\{1,...,N_0\}$ s.t. $D_k\geq1$ we have
\begin{equation}\label{BorneSupFoncTestShaarp3}
\dfrac{1}{2}\int_{{B(x_i^{(k)},\lambda\delta^2)}}\alpha|\n v|^2+\dfrac{\alpha^2}{2\v^2}(1-|v|^2)^2=b^2\pi\ln(b\lambda\delta^2/\v)+b^2\gamma+o(1)
\end{equation}
where $\gamma\in\R$ is a universal constant.

In conclusion, by combining \eqref{BorneSupFoncTestShaarp1}, \eqref{BorneSupFoncTestShaarp2} and \eqref{BorneSupFoncTestShaarp3} [note $\lambda\delta\h\to0$]:
\begin{eqnarray}\nonumber
F(v)&\leq&
d\pi\left[b^2|\ln\v|+(1-b^2)|\ln(\lambda\delta)|\right]+d\left[W^{\rm micro}(x_0)+b^2\gamma+b^2\pi\ln b\right]+
\\\label{BorneSupFoncTestShaarpGlobal}&&\phantom{lglkhgklhjjkh}+W_{N_0}^{\rm macro}\pD-\pi\sum_{\substack{k=1\\\text{s.t. }D_k\geq2}}^{N_0}\sum_{i\neq j}\ln|z^{(k)}_i-z^{(k)}_j|+o(1)
.
\end{eqnarray}
{\bf Step 4. Definition of the magnetic potential and conclusion}

Let  $A_{({\bf z},{\bf 1})}$ be given by Definition \ref{DefA_ad} with $\ad=({\bf z},{\bf 1})$. It is clear that we have
\[
-\pi\sum_{\substack{k=1\\\text{s.t. }D_k\geq2}}^{N_0}\sum_{i\neq j}\ln|z^{(k)}_i-z^{(k)}_j|\leq C|\ln\delta|
\]
where $C$ depends only on $d$ and $\O$. 

Consequently, for $\v>0$ sufficiently small and $\borneh>\pi d$ we have $F(v)\leq\borneh|\ln\v|$. Therefore, with Remark \ref{Prop.BornéPourPotenteifjfjf}, the configuration $(v,A_{({\bf z},{\bf 1})})\in\H$ is s.t. $\F(v,A_{({\bf z},{\bf 1})})\leq\F(v,0)+o(1)\leq\borneh|\ln\v|^2+\Haus^2(\O)\h^2$.


Using Proposition\ref{Docmpen} and Lemma \ref{Rk.RegularityLondonModified}  we get \begin{eqnarray*}
\F(v,A_{({\bf z},{\bf 1})})&=&\h^2\Jo+2\pi\h\sum_{i=1}^d\xi_0(z_i)+F(v)+\tilde{V}[\zeta_{({\bf z},{\bf 1})}]+o(1)
\end{eqnarray*}
where $\zeta_{({\bf z},{\bf 1})}$ is the unique solution of \eqref{LondonEqModifie} with $\ad={({\bf z},{\bf 1})}$.

We now use Assertion \ref{PropClusterI3} of Proposition \ref{PropClusterI} in order to get $\tilde{V}[\zeta_{({\bf z},{\bf 1})}]=\tilde{V}[\zeta_\pD]+o(1)$ and then
\begin{equation}\label{EstEnConstructSplitBorneSup}
\F(v,A_{({\bf z},{\bf 1})})=\h^2\Jo+2\pi\h\sum_{i=1}^d\xi_0(z_i)+F(v)+\tilde{V}_{({\bf z},{\bf 1})}[\zeta_\pD]+o(1).
\end{equation}

We claim that, from the choice of the points $z_i^{(k)},\tilde z_i^{(k)}$ we have $\xi_0(z_i^{(k)})-\xi_0(\tilde z_i^{(k)})=\mathcal{O}(\delta/\sqrt{\h})$. Thus with Proposition \ref{EnergieRenMeso} we have 
\begin{eqnarray*}
&&-\pi\sum_{\substack{k=1\\\text{s.t. }D_k\geq2}}^{N_0}\sum_{i\neq j}\ln|z^{(k)}_i-z^{(k)}_j|+2\pi\h\sum_{\substack{k=1}}^{N_0}\sum_{i}\xi_0(z_i^{(k)})-2\pi d\h\min_\O\xi_0
\\&=&\sum_{\substack{k=1\\\text{s.t. }D_k\geq1}}^{N_0}\left[-\pi\sum_{\substack{i,j\in\{1,...,D_k\}\\i\neq j}}\ln| \tilde z^{(k)}_i-\tilde z^{(k)}_j|+2\pi\h\sum_{i=1}^{D_k}\left[\xi_0(\tilde z^{(k)}_i)-\min_\O\xi_0\right]\right]+o(1)
\\&=&\sum_{\substack{k=1\\\text{s.t. }D_k\geq1}}^{N_0}\left[\dfrac{\pi}{2}(D_k^2-D_k)\ln\left(\dfrac{\h}{D_k}\right)+C_{p_k,D_k}\right]+o(1).
\end{eqnarray*}
We may now conclude:
\begin{eqnarray*}
\F(v,B)&=&\h^2\Jo+d\Pic\left[-\h+\HoC\right]+\dfrac{\pi}{2}\ln\h\sum_{\substack{k=1\\\text{s.t. }D_k\geq1}}^{N_0}(D_k^2-D_k)+
\\&&\phantom{jdjdjdjdjdkdkdkdkd}+\Wmin_{d}+\dfrac{\pi}{2}\sum_{\substack{k=1\\\text{s.t. }D_k\geq1}}^{N_0}(D_k-D_k^2)\ln D_k+o(1).
\end{eqnarray*}
This estimate ends the proof of the proposition.
\end{proof}
\section{Proof of Proposition \ref{Prop.EtaEllpProp}}\label{AppendixPruveEtaEllipt}

Let $\h$ and $(v_\v,A_\v)$ be as in Proposition \ref{Prop.EtaEllpProp}. Note that we may assume that $A_\v=A_{v_\v}$ given by Lemma \ref{LemAuxConstructMagnPot} and then $\|A_\v\|_{L^\infty(\O)}=\mathcal{O}(\h)$. We drop the subscript $\v$. We first note that, by smoothness of $\O$, there is $t_0>0$, s.t. letting $\O_{t_0}:=\{x\in\R^2\,|\,\dist(x,\O)<t_0\}$, we may extend by reflexion $v\in H^1(\O,\C)$ into $u\in H^1(\O_{t_0},\C)$ letting $u=v$ in $\O$ and $u=v\circ\SO$ in $\O_{t_0}\setminus\overline{\O}$ where
\[
\begin{array}{cccc}
\SO:&\O_{t_0}\setminus\overline{\O}&\to&\O\\&x&\mapsto&\Pi(x)-\dist(x,\p\O)\nu_{\Pi(x)}
\end{array}.
\]
Here $\Pi:\O_{t_0}\setminus\overline{\O}\to\p\O$ is the orthogonal projection on $\p\O$ and, for $\sigma\in\p\O$, $\nu_\sigma$ is the normal outward at $\sigma$.
\begin{lem}\label{EtaEllpProp}
Let $\borneh\geq1$ and let $\{(v_\v,A_\v)\,|\,0<\v<1\}$ be a family in the Coulomb gauge of quasi-minimizers of $\F$ in $\H$ for an intensity of the applied field $\h=\h(\v)\geq0$  s.t. $\|\n |v|\|_{L^\infty(\O)}\leq\borneh\v^{-1}$. 

Under these hypotheses, for $\eta\in(0,1)$ there exists $\v_\eta,C_\eta>0$ [depending on $\borneh$] s.t. for $0<\v<\v_\eta$, if $z\in\O$ is s.t.
\[
b^2\int_{B(z,\sqrt{\v}/2)}|\n u|^2+\dfrac{b^2}{\v^2}(1-|u|^2)^2\leq \dfrac{C_\eta}{3}|\ln\v|
\]
with $u=\begin{cases}v&\text{in }\O\\v\circ\SO&\text{in }\O_{t_0}\setminus\overline{\O}\end{cases}$, then $|v(z)|>\eta$.
\end{lem}

In order to prove  Proposition \ref{Prop.EtaEllpProp} we need the following lemma.

\begin{lem}\label{lem.Doublerecouvrement}
There exists $\v_\O>0$ depending only on  $\O$ s.t. for $0<\v<\v_\O$, $z\in\O$ and $v\in H^1(\O,\C)$, by defining $u$ as in Lemma \ref{EtaEllpProp}, the following inequality holds:
\[
\int_{B(z,\sqrt{\v}/2)}|\n u|^2+\dfrac{b^2}{\v^2}(1-|u|^2)^2\leq 3\int_{B(z,\sqrt\v)\cap\O}|\n v|^2+\dfrac{b^2}{\v^2}(1-|v|^2)^2.
\]
\end{lem}

\begin{proof}[Proof of Lemma \ref{lem.Doublerecouvrement}]
In order to prove the lemma it suffices to check that  by smoothness of $\O$ we have $\|\n(\SO^{-1})\|_{L^\infty(\O)},\|{\rm jac}\,(\SO^{-1})\|_{L^\infty(\O)}=1+o(1)$. We then immediately obtain
\[
\int_{B(z,\sqrt{\v}/2)\setminus\O}|\n u|^2+\dfrac{b^2}{\v^2}(1-|u|^2)^2\leq[1+o(1)]\int_{\SO[B(z,\sqrt{\v}/2)\setminus\O]}|\n v|^2+\dfrac{b^2}{\v^2}(1-|v|^2)^2.
\]
On the other hand, if $x\in B(z,\sqrt{\v}/2)\setminus\O$ then $|\SO(x)-z|\leq [1+o(1)]\sqrt\v/2\leq\sqrt\v$ for sufficiently small $\v>0$ [depending only on $\O$]. Then $\SO[B(z,\sqrt{\v}/2)\setminus\O]\subset B(z,\sqrt\v)\cap\O$. The lemma follows from the monotonicity of the integral.
\end{proof}

By combining both lemmas we get Proposition \ref{Prop.EtaEllpProp}.

\begin{proof}[Proof of Lemma \ref{EtaEllpProp}]

We argue by contradiction and we assume the existence of  $\eta\in(0,1)$, $\v=\v_n\downarrow0$ s.t.  for all $n\geq1$ there are $(v,A)=(v_n,A_n)\in\H$, $z=z_n\in\O$ and $\h=\h^{(n)}\geq0$ s.t. $(v,A)$ is a quasi-minimizers of $\F$ in $\H$ satisfying:
\begin{equation}\label{ContraTardRERTard}
\int_{B(z,\sqrt{\v}/2)}|\n u|^2+\dfrac{b^2}{\v^2}(1-|u|^2)^2\leq \dfrac{|\ln\v|}{n}
\end{equation}
with $u=u_n=\begin{cases}v&\text{in }\O\\v\circ\SO&\text{in }\O_{t_0}\setminus\overline{\O}\end{cases}$ and $|v(z)|\leq\eta$. Up to replace $v$ by $\underline{v}$ we may assume $|v|\leq1$ in $\O$.

We are going to prove that \eqref{ContraTardRERTard} implies  
\begin{equation}\label{ContraIpadKlldl}
\displaystyle\dfrac{1}{\v^2} \int_{B(z,\v^{3/4})\cap\O}(1-|v|^2)^2=o(1).
\end{equation}
On the other hand, $\|\n |v|\|_{L^\infty(\O)}=\mathcal{O}(\v^{-1})$ and then, from an argument in  \cite{BBH} [Theorem III.3], we will get, for sufficiently large $n$, $|v(z)|>\eta$. Clearly this contradiction will end the proof.



Since for $n\geq1$ we have
$\displaystyle\int_{\v^{3/4}/2}^{\sqrt{\v}/2}\dfrac{{\rm d}\rho}{\rho}\,\rho\int_{\p B(z,\rho)}|\n u|^2+\dfrac{b^2}{\v^2}(1-|u|^2)^2
\leq \dfrac{|\ln\v|}{n}$, 
there exists $\displaystyle\rho_n\in(\v^{3/4},\sqrt{\v}/2)$ s.t. $\displaystyle\rho_n\int_{\p B(z,\rho_n)}|\n u|^2+\dfrac{b^2}{\v^2}(1-|u|^2)^2\leq\dfrac{4}{n}$. Then we get :
\begin{equation}\label{Eq.BorneEnCoordPolbis}
\rho_n\int_{\p B(z,\rho_n)}|\p_\tau {u}|^2+\dfrac{b^2}{\v^2}(1-|u|^2)^2\leq\dfrac{4}{n}.
\end{equation}
We switch in polar coordinate and we denote $\tilde u(\theta):=u(z+\rho_n\e^{\imath\theta})$. Estimate \eqref{Eq.BorneEnCoordPolbis} becomes
\begin{equation}\label{Eq.BorneEnCoordPol}
\int_0^{2\pi}|\p_\theta \tilde{u}|^2+\dfrac{b^2\rho_n^2}{\v^2}(1-|\tilde u|^2)^2\leq\dfrac{4}{n}.
\end{equation}
On the one hand,  $|\p_\theta |\tilde {u}||^2\leq|\p_\theta \tilde {u}|^2$ and then $\di\int_0^{2\pi}|\p_\theta|\tilde {u}||\leq\dfrac{2\sqrt{2\pi}}{\sqrt n}$. Consequently in $[0,2\pi]$ we get $(1-|\tilde {u}|^2)^2\geq\max_{[0,2\pi]}(1-|\tilde {u}|^2)^2-\dfrac{2\sqrt{2\pi}}{\sqrt n}$. From \eqref{Eq.BorneEnCoordPol} we deduce
\[
\dfrac{4\v^2}{nb^2\rho_n^2}\geq\int_0^{2\pi}(1-|\tilde  u|^2)^2\geq 2\pi\left[\max_{[0,2\pi]}(1-|\tilde {u}|^2)^2-\dfrac{2\sqrt{2\pi}}{\sqrt n}\right]
\]
and thus for sufficiently large $n$ we get $0\leq\max_{[0,2\pi]}(1- |\tilde {u}|^2)^2\leq\dfrac{100}{\sqrt n}$.

For a further use we define 
\[
\begin{array}{cccc}\mod_n:&B(z,\rho_n)&\to&[0,1]
\\& z+\rho\e^{\imath\theta}&\mapsto&\left(|\tilde u(\theta)|-1\right)\dfrac{\rho}{\rho_n}+1
\end{array}.
\]
By direct calculations we have
\begin{equation}\label{EstModulAlADemande}
\int_{B(z,\rho_n)}|\n\mod_n|^2+\dfrac{1}{2\v^2}(1-\mod_n^2)^2
= \mathcal{O}\left(\dfrac{1}{ n}\right).
\end{equation}
On the other hand, for $n$ sufficiently large,  $|{u}|^2\geq\dfrac{1}{2}$ in $\p B(z,\rho_n)$. We thus may compute the degree of  $u$ on $\p B(z,\rho_n)$ and we find
$\left|\deg_{\p B(z,\rho_n)}(u)\right|
= \mathcal{O}\left(\dfrac{1}{ n}\right)$ which implies, for sufficiently large $n$, $\deg_{\p B(z,\rho_n)}(u)=0$. Consequently, we may write  $u=|u|\e^{\imath\varphi}$ with $\varphi=\varphi_n\in H^1(\p B(z,\rho_n),\R)$. Moreover, up to multiply $u$ by a constant in $\S^1$, we may assume $\int_{\p B(z,\rho_n)}\varphi=0$. 

We then consider $\tilde\varphi:[0,2\pi]\to\R$ defined by $\tilde\varphi(\theta)=\varphi(z+\rho_n\e^{\imath\theta})$, and thus
\[
\mathcal{O}\left(\dfrac{1}{ n}\right)=\rho_n\int_{\p B(z,\rho_n)}|\n \varphi|^2\geq\int_0^{2\pi}|\p_\theta\tilde\varphi|^2.
\]
Since $\displaystyle\int_0^{2\pi}\tilde\varphi=0$, this estimate implies $\displaystyle\int_0^{2\pi}\tilde\varphi^2=\mathcal{O}\left(\dfrac{1}{ n}\right)$.

Letting $\psi=\psi_n:B(z,{\rho_n})\to\R$, $z+\rho\e^{\imath\theta}\mapsto\dfrac{\rho}{\rho_n}\tilde\varphi(\theta)$, it is direct to check $\displaystyle\int_{B(z,\rho_n)}|\n \psi|^2=\mathcal{O}\left(\dfrac{1}{ n}\right)$.

We are now in position to end the proof by considering $V=V_n=\mod_n\e^{\imath\psi}\in H^1(B(z,\rho_n),\C)$ in order to have $V=v$ on $\p B(z,\rho_n)\cap\O$, 
 \[
 \dfrac{1}{2}\int_{\O\cap B(z,\rho_n)}|\n V|^2+\dfrac{1}{2\v^2}(1-|V|^2)^2=\mathcal{O}\left(\dfrac{1}{ n}\right).
 \]
and [with $\|A\|_{L^\infty(\O)}=\mathcal{O}(\h)$]
 \[
\left| \int_{\O\cap B(z,\rho_n)}\alpha(V\wedge\n V)\cdot A\right|\leq C\dfrac{\h^{}\rho_n}{\sqrt n}=o(1).
 \]
 Since $V=v$ on $\p B(z,\rho_n)\cap\O$ we have $w:=\begin{cases}v&\text{in }\O\setminus B(z,\rho_n)\\V&\text{in }B(z,\rho_n)\cap\O\end{cases}\in H^1(\O,\C)$. Considering the comparison configuration $(\tilde w,A)$, from the quasi-minimality of $(v,A)$ and the above estimates we get 
  \[
\int_{\O\cap B(z,\rho_n)}|\n v|^2+\dfrac{1}{2\v^2}(1-|v|^2)^2\leq b^{-4}\int_{\O\cap B(z,\rho_n)}|\n V|^2+\dfrac{1}{2\v^2}(1-|V|^2)^2+o(1)=o(1).
 \]
Since $\rho_n>\v^{3/4}$  we get \eqref{ContraIpadKlldl} and thus this estimate ends the proof.
\end{proof}
\section{Proof of Proposition \ref{Prop.ConstrEpsMauvDisk}}\label{SectAppenPreuveConstructionPetitDisque}
The proof of the proposition is an adaptation of the arguments presented in \cite{AB1} [Section V] and also used in \cite{S1} [Proposition 3.2]. It is also inspired of the bad disk construction in \cite{BBH}. Let $\mu$, $\lambda$, $\delta$,  $(v,A)$ and $\h$ be as in the proposition. 

{\bf Step 1. A first covering of $ \{|v|\leq1/2\}$}\\

For $0<\v<\v_{1/2}$ [$\v_{1/2}>0$ is given by Proposition \ref{Prop.EtaEllpProp} with $\eta=1/2$] we consider a covering of  $\O$ by disks $\{B(x^\v_1,4\sqrt\v),...,B(x^\v_{N_\v},4\sqrt\v)\}$ s.t., for $i\neq j$, $B(x^\v_i,\sqrt\v)\cap B(x^\v_j,\sqrt\v)=\emptyset$ and $x^\v_i\in\O$.

For the simplicity of the presentation we omit the dependance in $\v$.\\

We say that $B(x_i,4\sqrt\v)$ is a {\it bad disk} if  $\tilde E_\v[v,B(x_i,8\sqrt\v)\cap\O]> C_{1/2}|\ln\v|$ where for a disk $B$ we denoted
\[
\tilde E_\v(v,B\cap\O):=\int_{B\cap\O}|\n v|^2+\dfrac{1}{\v^2}(1-|v|^2)^2
\]
and $C_{1/2}>0$ is given by Proposition \ref{Prop.EtaEllpProp} with $\eta=1/2$.  Let 
\[
J'=J'_\v:=\{i\in\{1,...,N_\v\}\,|\, B(x_i,4\sqrt\v)\text{ is a bad disk}\}.
\]
We make two fundamental claims:
\begin{enumerate}
\item There exists $M_0\geq1$ [independent of  $\v$] s.t. ${\rm Card}(J')\leq M_0$.
\item If $B(x_i,4\sqrt\v)$ is not a bad disk then $|v|\geq1/2$ in $B(x_i,4\sqrt\v)$.
\end{enumerate}
The first claim is a direct consequence of \eqref{CrucialBoundedkjqbsdfbn} and $B(x^\v_i,\sqrt\v)\cap B(x^\v_j,\sqrt\v)=\emptyset$ for $i\neq j$.

The second claim is given by Proposition \ref{Prop.EtaEllpProp}. 
Then $\cup_{i\in J'} B(x_i,4\sqrt\v)$ is covering of $\{|v|\leq1/2\}$ and ${\rm Card}(J')\leq M_0$.

Up to drop some disks, we may always assume that for $i\in J'$ we have $B(x_i,4\sqrt\v)\cap\{|v|\leq1/2\}\neq\emptyset$. Consequently using Corollary \ref{Cor.DegNonNul}, for $i\in J'$ and $0<\v<\min\{\v_0,\v_{1/2}\}$ [$\v_0$ given by Corollary \ref{Cor.DegNonNul}] we have
$\dist(x_i,\Lambda)= \mathcal{O}(|\ln\v|^{-s_0})$.\\

If $|v|>1/2$ in $\O$ then there is nothing to prove. We then assume  $J'\neq \emptyset$.\\

{\bf Step 2. Separation process}\\

We replace the above bad disks with disks having same centers and with a radius  $\v^\mu$. Let $\v_\mu^{(1)}>0$ be s.t. $\min\{\v_0,\v_{1/2}\}\geq \v_\mu^{(1)}$, for $0<\v<\v_\mu^{(1)}$ we have $4\sqrt\v<\v^\mu$ and
\begin{equation}\nonumber
\max_{i\in J'}\,\dist(B(x_i,\v^\mu),\Lambda)\leq \dfrac{1}{\ln|\ln\v|}.
\end{equation} In particular $\cup_{i\in J'} B(x_i,\v^\mu)$ is a covering of $\{|v|\leq1/2\}$.\\

The goal of this step is to get a covering of  $\{|v|\leq1/2\}$ with disks $B(x_i,\v^s)$ where $i\in\tilde {J}=\tilde J_\v\subset J'$, $s=s_\v=2^{-K}\mu$, $K=K_\v\in\{0,...,M_0-1\}$ and s.t. for  $i,j\in\tilde J$, $i\neq j$, we have
\begin{equation}\label{ConditiondeSepartionMauvDisk}
|x_i-x_j|\geq\v^{s/2}.
\end{equation}
If ${\rm Card}(J')=1$ or \eqref{ConditiondeSepartionMauvDisk} is satisfied with $s=\mu$ [i.e. $K=0$] then we let $\tilde{J}=J'$ and we obtained the desired result of this step. Otherwise, there are $i_0,j_0\in J'$ [with $i_0<j_0$] s.t.  $|x_{i_0}-x_{j_0}|<\v^{\mu/2}$. In this case we let  $J^{(1)}:=J'\setminus\{i_0\}$ and we claim that ${\rm Card}(J^{(1)})={\rm Card}(J')-1$.

If ${\rm Card}(J^{(1)})=1$ or ${\rm Card}(J^{(1)})>1$ with \eqref{ConditiondeSepartionMauvDisk} holds with $s=2^{-1}\mu$ [i.e. $K=1$]  for all $i,j\in J^{(1)}$ [$i\neq j$] then  the goal of this step is done with $\tilde J= J^{(1)}$ and $s=2^{-1}\mu$. 

Otherwise, there exits  $i_0,j_0\in J^{(1)}$ [with $i_0<j_0$] s.t.  $|x_{i_0}-x_{j_0}|<\v^{s/2}$. We then let $J^{(2)}:=J^{(1)}\setminus\{i_0\}$ and thus ${\rm Card}(J^{(2)})={\rm Card}(J^{(1)})-1$.

By noting that ${\rm Card}(J')\leq M_0$, the above process stops after at most  $M_0-1$ iteration. We thus get the existence of $K=K_\v\in\{0,...,M_0-1\}$  and $\emptyset\neq J^{(K)}=J^{(K)}_\v\subset J'$ s.t. ${\rm Card} (J^{(K)})=1$ or \eqref{ConditiondeSepartionMauvDisk} is satisfied with $s=s_\v=2^{-K}\mu$ and $i,j\in J^{(K)}$ [$i\neq j$].

We then denote  $\tilde J:=J^{(K)}$, $s=2^{-K}\mu$ and we fix $0<\v_\mu^{(2)}\leq\v_\mu^{(1)}$ s.t. for $0<\v<\v_\mu^{(2)}$ we have
\begin{equation}\nonumber
\max_{i\in \tilde J}\,\dist(B(x_i,\v^{s/4}),\Lambda)\leq \dfrac{1}{\ln|\ln\v|}<10^{-1}\dist(\Lambda,\p\O).
\end{equation}
In particular $B(x_i,\v^{s/4})\subset\O$ for $i\in\tilde J$.
\\

{\bf  Step 3. Definition of $\rad$}

With Corollary 5.2 in \cite{bourgain2010morse}, for a.e. $t\in{\rm Image}(|v|)$ the set  $V(t):=\{x\in\O\,|\,|v(x)|=t\}$ is a finite union of curve. Moreover if a such curve is included in  $\O$ then it is a Jordan curve.

Following the same strategy as in \cite{AB1} [Lemma V.1], we have the existence of $t_\v\in[1-2|\ln\v|^{-2},1-|\ln\v|^{-2}]$ s.t. $V(t_\v)$ is a finite union of Jordan curves s.t.
\begin{equation}\label{BorneValeurMesure}
\Haus^1[V(t_\v)]\leq C\v|\ln\v|^5\text{ with $C$ is independent of $\v$.}
\end{equation}
We fix $0<\v_\mu^{(3)}\leq\v_\mu^{(2)}$ s.t. for $0<\v<\v_\mu^{(3)}$ we have $C\v|\ln\v|^5\leq10^{-2}\v^{s}$.

We denote for $i\in\tilde{J}$
\begin{equation}\label{DefAiEps}
\A_i=\A_i^\v:=\{\rho\in[\v^s,\v^{2s/3}]\,|\,|v|\geq t_\v\text{ on }\p B(x_i,\rho)\}.
\end{equation}
From the continuity of $|v|$, it is clear that $[\v^s,\v^{2s/3}]=\A_i\cup\B_i\cup\mathcal{C}_i$ where
\[
\B_i=\B_i^\v:=\{\rho\in[\v^s,\v^{2s/3}]\,|\,\exists\,x\in\p B(x_i,\rho)\text{ s.t. }|v(x)|=t_\v\}
\]
and
\[
\mathcal{C}_i=\mathcal{C}_i^\v:=\{\rho\in[\v^s,\v^{2s/3}]\,|\,|v|< t_\v\text{ on }\p B(x_i,\rho)\}.
\]
We first claim that, since the function  $\rho\mapsto\rho$ is increasing, we have
\begin{eqnarray*}
\mathcal{O}(\v^2|\ln\v|)&=&\int_{\mathcal{C}_i}{\rm d}\rho\int_{\p B(x_i,\rho)}(1-|v|^2)^2
\\&\geq&2\pi(1-t_\v^2)^2\int_{\mathcal{C}_i}\rho{\rm d}\rho
\\&\geq&2\pi(1-t_\v^2)^2\int_0^{\Haus^1(\mathcal{C}_i)}\rho{\rm d}\rho=\pi(1-t_\v^2)^2\Haus^1(\mathcal{C}_i)^2.
\end{eqnarray*}
Then $\Haus^1(\mathcal{C}_i)=\mathcal{O}(\v|\ln\v|^{5/2})$.

On the other hand one may prove that if  $I$ is a connected components of  $\B_i$, then there is $\rho_1,\rho_2$ s.t. $ I=[\rho_1,\rho_2]$. Since straight lines are geodesics, we obviously get 
\[
\Haus^1(I)=\rho_2-\rho_1\leq\Haus^1[V(t_\v)\cap\overline{B(x_i,\rho_2)}\setminus B(x_i,\rho_1)].
\]
Moreover one may prove
that if $[\rho_1,\rho_2]$ and $[\rho_1',\rho_2']$ are distinct connected component of $\B_i$ and if $\Gamma$ is a connected component of $V(t_\v)$ s.t. $ \Gamma\cap\overline{B(x_i,\rho_2)}\setminus B(x_i,\rho_1)\neq \emptyset$ then $ \Gamma\cap\overline{B(x_i,\rho_2')}\setminus B(x_i,\rho_1')= \emptyset$ [here we used \eqref{BorneValeurMesure}]. One may conclude: $\Haus^1(\B_i)\leq\Haus^1(V(t_\v))\leq C\v|\ln\v|^5$.

Consequently
\[
\Haus^1(\A_i)\geq\Haus^1([\v^s,\v^{2s/3}])-\Haus^1(\B_i)-\Haus^1(\mathcal{C}_i)\geq\v^{2s/3}-\v^s-\Haus^1(V(t_\v))-\mathcal{O}(\v|\ln\v|^{5/2}).
\]
Fix $0<\v_\mu^{(4)}\leq\v_\mu^{(3)}$ s.t. for $0<\v<\v_\mu^{(4)}$ we have $\Haus^1(\A_i)\geq\v^{2s/3}-\v^s-\sqrt\v$.

Define
\begin{equation}\label{DefAEps}
\A=\A_{\mu,\v}:=\cap_{i\in \tilde J}\A_i.
\end{equation}
It is clear that $\Haus^1(\A)\geq\v^{2s/3}-\v^s-M_0\sqrt\v$

Since $\rho\mapsto1/\rho$ is decreasing we have
\begin{eqnarray*}
\mathcal{O}(|\ln\v|)
&\geq&\int_\A\dfrac{{\rm d}\rho}{\rho}\,\sum_{i\in\tilde J}\rho\int_{\p B(x_i,\rho)}|\n v|^2+\dfrac{1}{\v^2}(1-|v|^2)^2
\\&\geq&\int_{\v^{2s/3}-\Haus^1(\A)}^{\v^{2s/3}}\dfrac{{\rm d}\rho}{\rho}\,\times\,\inf_{\rho\in\A}\sum_{i\in\tilde J}\rho\int_{\p B(x_i,\rho)}|\n v|^2+\dfrac{1}{\v^2}(1-|v|^2)^2.
\end{eqnarray*}
Consequently, there exist $\rad=\rad_{\mu,\v}\in\A$, $C_\mu\geq1$ [$C_\mu$ is independent of $\v$] and $0<\v_\mu^{(5)}\leq\v_\mu^{(4)}$ s.t. for $0<\v<\v_\mu^{(5)}$ we have

\begin{equation}\label{BorneBordConditiondeSepartionMauvDisk}
\sum_{i\in\tilde J}\rad\int_{\p B(x_i,\rad)}|\n v|^2+\dfrac{1}{\v^2}(1-|v|^2)^2\leq C_\mu.
\end{equation}
We finally let $J_\mu:=\tilde J $, with \eqref{ConditiondeSepartionMauvDisk} and \eqref{BorneBordConditiondeSepartionMauvDisk} the result is proved.

\section{Proof of Proposition \ref{VeryNiceCor}}\label{Sec.PreuveVeryNiceCor}
The proof is an adaptation of the proof of (VI.21) in \cite{AB1}.

Let $\tilde\alpha=\tilde\alpha_n\in L^\infty(\O,[\beta^2;1])$,  $\zd=\zd^{(n)}\in\Ostar\times\Z^N$ and $u=u_n\in H^1(\O,\C)$ be as in the proposition.

We first claim that up to consider $\underline{u}$ instead of $u$ we may assume $|u|\leq1$ in $\O$. Note also that if $\int_{\O_\Rad}|\n u|^2\geq \beta^{-2}\int_{\O_\Rad}|\n \wstar|^2$, then there is nothing to prove. We thus may assume
\begin{equation}\nonumber
\int_{\O_\Rad}|\n u|^2< \beta^{-2}\int_{\O_\Rad}|\n \wstar|^2.
\end{equation}

Let $w:={u}/{|u|}\in H^1(\O_\Rad,\S^1)$. From Lemma I.1 in \cite{BBH} we have $w\wedge\n w=\n^{\bot}\Pstar+\n H$ with $H=H_\v\in H^1(\O_\Rad,\R)$ and 
\begin{equation}\label{NaturalHyp}
\int_{\O_\Rad}|\n H|^2\leq (\beta^{-1}+1)^2\int_{\O_\Rad}|\n \Pstar|^2.
\end{equation}
Let $\Phi_\Rad$ be the unique solution of \eqref{ConjaHarmmojqsdhfhfh}. 
We have $\di\int_{\O_\Rad}\n H\cdot\n^\bot\Phi_\Rad=0$. Then letting $\rho=|u|$:
\begin{eqnarray*}
\int_{\O_\Rad}\tilde \alpha\rho^2\n H\cdot\n^\bot\Pstar=\int_{\O_\Rad}(\tilde \alpha\rho^2-1)\n H\cdot\n^\bot\Pstar+\int_{\O_\Rad}\n H\cdot(\n^\bot\Pstar-\n^\bot\Phi_\Rad).
\end{eqnarray*}
But, from \eqref{ConjaHarmmojqsdhfhfh222}, there exists $C\geq1$ s.t. $\displaystyle
\left|\int_{\O_\Rad}\n H\cdot(\n^\bot\Pstar-\n^\bot\Phi_\Rad)\right|\leq C\|\n H\|_{L^2(\O_\Rad)}\sqrt X$ where $X$ is defined in \eqref{DefX}.

Consequently, letting $\tilde C:= 4C^2/\beta^2$ we get
\begin{eqnarray*}
2\int_{\O_\Rad}\n H\cdot\n^\bot\Pstar+\int_{\O_\Rad}\tilde \alpha\rho^2|\n H|^2&=&2\int_{\O_\Rad}\n H\cdot(\n^\bot\Pstar-\n^\bot\Phi_\Rad)+\int_{\O_\Rad}\tilde \alpha\rho^2|\n H|^2
\\&\geq& \|\n H\|_{L^2(\O_\Rad)}\left(\dfrac{\beta^2}{4}\|\n H\|_{L^2(\O_\Rad)}-2C\sqrt X\right)\\&\geq&- \tilde CX.
\end{eqnarray*}
Therefore
\begin{equation}\nonumber
\int_{\O_\Rad}\tilde \alpha\rho^2|\n w|^2
\geq\int_{\O_\Rad}|\n \Pstar|^2-\int_{\O_\Rad}(1-\tilde \alpha\rho^2)|\n\Pstar|^2-2\int_{\O_\Rad}(1-\tilde \alpha\rho^2)|\n H||\n\Pstar|-\mathcal{O}(X).
\end{equation}
On the other hand, using  \eqref{BorneGradWstar} and Corollary \ref{CorBorneGrossEneStar}, we get
\begin{eqnarray*}
\left|\int_{\O_\Rad}(1-\tilde \alpha\rho^2)|\n\Pstar|^2\right|&\leq&\left|\int_{\O_\Rad}(1-\rho^2)|\n\Pstar|^2\right|+\left|\int_{\O_\Rad}(1-\tilde \alpha)|\n\Pstar|^2\right|
\\&\leq&\|\n\Pstar\|_{L^\infty(\O_\Rad)}\|\n\Pstar\|_{L^2(\O_\Rad)}\left(K+L\right)
\end{eqnarray*}
and with \eqref{NaturalHyp}:
\begin{eqnarray*}
\left|\int_{\O_\Rad}(1-\tilde \alpha\rho^2)|\n H||\n\Pstar|\right|&\leq&\left|\int_{\O_\Rad}(1-\rho^2)|\n H||\n\Pstar|\right|+\left|\int_{\O_\Rad}(1-\tilde \alpha)|\n H||\n\Pstar|\right|
\\&\leq&\|\n\Pstar\|_{L^\infty(\O_\Rad)}\|\n\Pstar\|_{L^2(\O_\Rad)}\left(K+L\right)(2\beta^{-1}+1).
\end{eqnarray*}
The proposition is thus proved.

\section{Proof of Proposition \ref{Prop.BonEcartement}}\label{Proof.Prop.BonEcartement}

We prove the first assertion and we assume ${\rm Card}(J_\mu)\geq2$. We let $\chi_1:=2\h^{-1}\ln\h$, $\chi_2:=2\h^{-1/2}\ln\h$ and $\O_{\chi_2}=\O\setminus\cup_{p\in\Lambda}\overline{B(p,\chi_2)}$.
 
In order to get sufficiently sharp estimates to prove the proposition, we decompose  $\O_\rad$  in several subdomains.  To this aim, we distinguish two cases for $p\in\Lambda$ : either ${\rm Card}(J_p^{(y)})\geq 2$ or ${\rm Card}(J_p^{(y)})\in\{0,1\}$  where $J_p^{(y)}:=\{k\in J^{(y)}\,|\,y_k\in B(p,\chi_2)\}$ [the $y_k$'s are introduced in Definition \ref{DefiSousEnsJ}].

If $p\in\Lambda$ is s.t. ${\rm Card}(J_p^{(y)})\geq2$, then with Lemma \ref{Lem.Separation} [with $\num=17$ and $\eta=\chi_1/2$], there are $\kappa_p=\kappa_{p,\v}\in\{17^0,...,17^{N_0-1}\}$ and $\tilde J_p^{(y)}\subset J_p^{(y)}$ s.t.
\[
\bigcup_{k\in J_p^{(y)}} B( y_k,\chi_1/2)\subset\bigcup_{k\in \tilde J_p^{(y)}} B( y_k,\kappa_p\chi_1/2)\text{ and }| y_k- y_l|\geq8\kappa_p\chi_1\text{ for }k,l\in\tilde J_p^{(y)},\,k\neq l.
\]
We then let $\dom_p:=B(p,\chi_2)\setminus\cup_{k\in  \tilde J_p^{(y)}}\overline{B(y_k,\kappa_p\chi_1)}$ and, for $k\in \tilde J_p^{(y)}$,  we write $\underline d_k:=\deg_{\p B(y_k,\kappa_p\chi_1)}( v)$. We denote also $ D_p:=\sum_{k\in \tilde J_p^{(y)}}\underline d_k$


If $p\in\Lambda$ is s.t. $J_p^{(y)}=\{k\}$, then we let $\dom_p=B(p,\chi_2)\setminus\overline{B(y_k,\kappa\delta)}$ with $\kappa$ given by Definition \ref{DefiSousEnsJ}. We let also $ D_p:=\underline d_k:=\deg_{\p B(y_k,\kappa\delta)}(v)$. 

Recall that we denoted (see Definition \ref{DefiSousEnsJ}), for $k\in J^{(y)}$, $\tilde d_k:=\deg_{\p B(y_k,\kappa\delta)}(v)$. Consequently, if $J_p^{(y)}=\{k\}$, then $ D_p=\underline d_k=\tilde d_k$.

If $J_p^{(y)}=\emptyset$ then we denote $D_p=0$ and $\dom_p=B(p,\chi_2)$.

The heart of the proof consists in proving that  $\underline d_k=1$ for all $k$. Indeed, we know that if  $i\in J_\mu$ then $\deg_{\p B(z_i,\rad)}(v)=1$. Consequently  $\underline d_k$ is the number of points $z_i$ contained in a disk of radius  at least  $\chi_1$.

We let:
\begin{itemize}
\item  $\mathcal{R}:=\bigcup_{k\in  J^{(y)}}B(y_k,\kappa\delta)\setminus\bigcup_{i\in  J_\mu}\overline{B(z_i,\rad)}$, $\kappa$ given in Definition \ref{DefiSousEnsJ}.
\item For $p\in\Lambda$ s.t. ${\rm Card}(J_p^{(y)})\geq2$ and for $k\in \tilde J_p^{(y)}$ we denote 
\[
\mathcal{Q}_{k,p}:=B(y_k,\kappa_p\chi_1)\setminus\bigcup_{\substack{l\in J^{(y)}\\y_l\in B(y_k,\kappa_p\chi_1)}}\overline{B(y_l,\kappa\delta)}.
\]
Moreover, by construction, we have  [for sufficiently small  $\v$]
\begin{equation}\label{VoilaPourquoiYaDesSur2}
\bigcup_{\substack{l\in J^{(y)}\\y_l\in B(y_k,\kappa_p\chi_1)}}{B(y_l,\kappa\delta)}\subset\bigcup_{\substack{l\in J^{(y)}\\y_l\in B(y_k,\kappa_p\chi_1)}}{B(y_l,\chi_1/2)}\subset  B(y_k,\kappa_p\chi_1/2).
\end{equation}
\end{itemize}
Thus
\begin{eqnarray}\nonumber
\dfrac{1}{2}\int_{\O_\rad}\alpha|\n v|^2&\geq&\dfrac{1}{2}\int_{\mathcal{R}}\alpha|\n v|^2+
\sum_{\substack{p\in\Lambda}}\dfrac{1}{2}\int_{\dom_{p}}\alpha|\n v|^2+
\\\label{EstVendre0}&&+\sum_{\substack{p\in\Lambda\\{\rm Card}(J_p^{(y)})\geq2}}\sum_{k\in\tilde J_p^{(y)}}\dfrac{1}{2}\int_{\mathcal{Q}_{k,p}}\alpha|\n v|^2+\dfrac{1}{2}\int_{\O_{\chi_2}}\alpha|\n v|^2.
\end{eqnarray}
From \eqref{OnCompteRelFin1} and \eqref{OnCompteRelFin2} we have
\begin{equation}\label{EstVendre1}
\dfrac{1}{2}\int_{\mathcal{R}}\alpha|\n v|^2\geq d\pi\left[b^2|\ln\rad|+(1-b^2)|\ln\lambda|-b^2|\ln\delta|\right]+\mathcal{O}(1).
\end{equation}
If $J_p^{(y)}=\{k\}$, then with Corollary \ref{Cor.BorneInfProcheIncl}.\ref{Cor.BorneInfProcheIncl1} we get
\begin{equation}\label{EstVendre2}
\dfrac{1}{2}\int_{\dom_{p}}\alpha|\n v|^2\geq\pi\underline{d}_k^2\ln\left(\dfrac{\chi_2}{\delta}\right)+\mathcal{O}(1).
\end{equation}
And if ${\rm Card}(J_p^{(y)})\geq2$, still with Corollary \ref{Cor.BorneInfProcheIncl}.\ref{Cor.BorneInfProcheIncl1}:
\begin{equation}\label{EstVendre4}
\dfrac{1}{2}\int_{\dom_p}\alpha|\n v|^2\geq
\pi\sum_{k\in \tilde J_p^{(y)}}\underline d_k^2\ln \left(\dfrac{\chi_2}{\chi_1}\right)+\mathcal{O}(1).
\end{equation}

We continue by dealing with the case ${\rm Card}(J_p^{(y)})\geq2$. From Corollary \ref{Cor.BorneInfProcheIncl}.\ref{Cor.BorneInfProcheIncl1} applied in $\mathcal{Q}_{k,p}$ for $k\in \tilde J_p^{(y)}$ [with \eqref{VoilaPourquoiYaDesSur2}] we get
\begin{equation}\label{EstVendre3}
\sum_{k\in\tilde J_p^{(y)}}\dfrac{1}{2}\int_{\mathcal{Q}_{k,p}}\alpha|\n v|^2
\geq\pi\sum_{k\in \tilde J_p^{(y)}}\sum_{\substack{l\in J^{(y)}\\y_l\in B(y_k,\kappa_p\chi_1)}}\tilde d_l^2\ln\left(\dfrac{\chi_1}{\delta}\right)+\mathcal{O}(1)
\end{equation}
In order to end the proof, using  Propositions \ref{MinimalMapHomo} $\&$ \ref{Prop.EnergieRenDef} $\&$ \ref{VeryNiceCor}, we get
\begin{equation}\label{EstVendre5}
\dfrac{1}{2}\int_{\O_{\chi_2}}\alpha|\n v|^2\geq\pi\sum_{p\in\Lambda} D_p^2|\ln\chi_2|+\mathcal{O}(1).
\end{equation}
We let
\[
\Delta:=\sum_{\substack{p\in\Lambda\text{ s.t.}\\{\rm Card} (J_p^{(y)})\geq2}}\sum_{k\in \tilde J_p^{(y)}}\underline d_k^2+\sum_{\substack{p\in\Lambda\text{ s.t.}\\J_p^{(y)}=\{k\}}}\underline{d}_k^2\text{ and }\tilde\Delta:=\sum_{k\in J^{(y)}}\tilde d_k^2.
\]

From \eqref{EstVendre0}, \eqref{EstVendre1}, \eqref{EstVendre2}, \eqref{EstVendre4}, \eqref{EstVendre3} and \eqref{EstVendre5} we get
\begin{eqnarray}\nonumber
&&\dfrac{1}{2}\int_{\O_\rad}\alpha|\n v|^2
\\\nonumber&&\geq \mathcal{O}(1)+d\pi\left[b^2|\ln\rad|+(1-b^2)|\ln\lambda|-b^2|\ln\delta|\right]+\pi\sum_{\substack{p\in\Lambda\text{ s.t.}\\J_p^{(y)}=\{k\}}}\underline{d}_k^2\ln\left(\dfrac{\chi_2}{\delta}\right)+
\\\nonumber&&+\pi\sum_{\substack{p\in\Lambda\\{\rm Card} (J_p^{(y)})\geq2}}\left[\sum_{k\in \tilde J_p^{(y)}}\underline d_k^2\ln \left(\dfrac{\chi_2}{\chi_1}\right)+\sum_{\substack{l\in J^{(y)}\\y_l\in B(p,\chi_2+\lambda\delta)}}\tilde d_l^2\ln\left(\dfrac{\chi_1}{\delta}\right)\right] +\pi\sum_{p\in\Lambda} D_p^2|\ln\chi_2|
\\\nonumber&&\geq d\pi\left[b^2|\ln\rad|+(1-b^2)|\ln(\lambda\delta)|\right]+\pi|\ln\chi_2|\left(\sum_{p\in\Lambda} D_p^2-\Delta\right)+\pi|\ln\delta|(\tilde\Delta-d)+
\\\nonumber&&+\pi|\ln\chi_1|\sum_{\substack{p\in\Lambda\\{\rm Card} (J_p^{(y)})\geq2}}\left[\sum_{k\in \tilde J_p^{(y)}}\underline d_k^2-\sum_{\substack{l\in J^{(y)}\\y_l\in B(p,\chi_2+\lambda\delta)}}\tilde d_l^2\right]+\mathcal{O}(1).
\end{eqnarray}

Since  $\underline d_k,\tilde d_l\geq1$ for all $k,l$, from Lemma \ref{LemSommeDegCarréDec}.\ref{LemSommeDegCarréDec1} we have $\sum_{p\in\Lambda} D_p^2\geq \Delta\geq\tilde\Delta\geq d$ and moreover
\[
\Delta=d\Leftrightarrow (\text{$\underline d_k=1$ for all $k$})
\]
and
\[
\tilde\Delta=d\Leftrightarrow (\text{$\tilde d_l=1$ for all $l$}).
\]
On the other hand since for $p\in\Lambda$ s.t. $J_p^{(y)}=\{k\}$ we have $\underline d_k=\tilde d_k$, we get
\[
\Delta-\tilde\Delta=\sum_{\substack{p\in\Lambda\\{\rm Card} (J_p^{(y)})\geq2}}\left[\sum_{k\in \tilde J_p^{(y)}}\underline d_k^2-\sum_{\substack{l\in J^{(y)}\\y_l\in B(p,\chi_2+\lambda\delta)}}\tilde d_l^2\right].
\]

Then  \eqref{BrneSup-ApplicDirEn} gives
\[
\dfrac{\L_1(d)}{\pi}\ln\h\geq \left(\sum_{p\in\Lambda} D_p^2-\Delta\right)|\ln\chi_2|+(\tilde\Delta-d)|\ln\delta|+(\Delta-\tilde\Delta)|\ln\chi_1|+\mathcal{O}(1).
\]
Since $|\ln\chi_1|=\ln({\h})+\mathcal{O}[\ln(\ln\h)]$ and $|\ln\chi_2|=\ln\sqrt{\h}+\mathcal{O}[\ln(\ln\h)]$ we obtain
\begin{eqnarray}\nonumber
&&\left(\dfrac{\L_1(d)}{\pi}+\dfrac{d-\sum_{p\in\Lambda} D_p^2}{2}\right)\ln{\h}
\\\label{topContra}&\geq&(\Delta-\tilde\Delta)\ln\sqrt{\h}+(\tilde\Delta-d)|\ln(\delta\sqrt{\h})|+\mathcal{O}[\ln(\ln\h)].
\end{eqnarray}

From Lemma \ref{LemSommeDegCarréDec}.\ref{LemSommeDegCarréDec2} and the definition of $\L_1(d)$ [see Lemma \ref{LemLaisseLectTrucSimple}], we have 
\begin{equation}\label{topBissosososo}
\dfrac{\L_1(d)}{\pi}+\dfrac{d-\sum_{p\in\Lambda} D_p^2}{2}\leq0.
\end{equation}

Using \eqref{topBissosososo} in \eqref{topContra}, \eqref{PutaindHypTech} and $\tilde\Delta-d\geq0\&\Delta-\tilde\Delta\geq0$ we get $\tilde\Delta-d=\Delta-\tilde\Delta=0$ and then $\Delta=d$, {\it i.e.} $\underline d_k=1$ for all $k$.

On the other hand, with the help of  \eqref{topContra} we may write
\[
0\geq\left(\dfrac{\L_1(d)}{\pi}+\dfrac{d-\sum_{p\in\Lambda} D_p^2}{2}\right)\ln{\h}\geq\mathcal{O}[\ln(\ln\h)].
\]

We may thus deduce $\dfrac{\L_1(d)}{\pi}+\dfrac{d-\sum_{p\in\Lambda} D_p^2}{2}=0$ and then, with Lemma \ref{LemSommeDegCarréDec}.\ref{LemSommeDegCarréDec2}, for $p\in\Lambda$ we have $D_p\in\{\lfloor d/N_0\rfloor;\lceil d/N_0\rceil\}$.

\bibliographystyle{abbrv}
\bibliography{MagneticBib}

\begin{thebibliography}{10}

\bibitem{ASS1}
A.~Aftalion, E.~Sandier, and S.~Serfaty.
\newblock {P}inning phenomena in the {Ginzburg-Landau} model of
  superconductivity.
\newblock {\em J. Math. Pures Appl.}, 80(3):339--372, 2001.

\bibitem{AB1}
L.~Almeida and F.~Bethuel.
\newblock {Topological methods for the Ginzburg-Landau equations}.
\newblock {\em J. Math. Pures Appl.}, 77(1):1 -- 49, 1998.

\bibitem{BBH1}
F.~Bethuel, H.~Brezis, and F.~H\'{e}lein.
\newblock Asymptotics for the minimization of a {Ginzburg-Landau} functional.
\newblock {\em Calc. Var. Partial Differential Equations}, 1(2):123--148, 1993.

\bibitem{BBH}
F.~Bethuel, H.~Brezis, and F.~H\'{e}lein.
\newblock {\em {Ginzburg-Landau Vortices}}.
\newblock Progress in Nonlinear Differential Equations and their Applications,
  13. Birkh\"auser Boston Inc., Boston, MA, 1994.

\bibitem{bourgain2010morse}
J.~Bourgain, M.~Korobkov, and J.~Kristensen.
\newblock On the {Morse-Sard} property and level sets of {S}obolev and {BV}
  functions.
\newblock {\em {Rev. Mat. Iberoam.}}, 1(29):1--23, 2013.

\bibitem{Publi4}
M.~{Dos Santos}.
\newblock The {Ginzburg-Landau} functional with a discontinuous and rapidly
  oscillating pinning term. {Part II}: the non-zero degree case.
\newblock {\em Indiana Univ. Math. J.}, 62(2):551--641, 2013.

\bibitem{dos2015microscopic}
M.~{Dos Santos}.
\newblock Microscopic renormalized energy for a pinned {Ginzburg-Landau}
  functional.
\newblock {\em Calc. Var. Partial Differential Equations}, 53(1-2):65--89,
  2015.

\bibitem{Dos-MicroRenoEN}
M.~{Dos Santos}.
\newblock {E}xplicit expression of the microscopic renormalized energy for a
  pinned {G}inzburg-{L}andau functional.
\newblock preprint - https://hal.archives-ouvertes.fr/hal-01684216, Jan. 2018.

\bibitem{Publi3}
M.~{Dos Santos} and O.~Misiats.
\newblock {Ginzburg-Landau model with small pinning domains}.
\newblock {\em Netw. Heterog. Media}, 6(4):715--753, 2011.

\bibitem{GT}
D.~Gilbarg and N.~Trudinger.
\newblock {\em Elliptic partial differential equations of second order}.
\newblock Springer, 2015.

\bibitem{K1}
A.~Kachmar.
\newblock {M}agnetic vortices for a {G}inzburg-{L}andau type energy with
  discontinuous constraint.
\newblock {\em ESAIM Control Optim. Calc. Var.}, 16(3):545--580, 2009.

\bibitem{LM1}
L.~Lassoued and P.~Mironescu.
\newblock {Ginzburg-Landau type energy with discontinuous constraint}.
\newblock {\em J. Anal. Math.}, 77:1--26, 1999.

\bibitem{LR1}
C.~Lefter and V.~R{\u{a}}dulescu.
\newblock Minimization problems and corresponding renormalized energies.
\newblock {\em Diff. Int. Eq}, 9(5):903--917, 1996.

\bibitem{Rub}
J.~Rubinstein.
\newblock {On the equilibrium position of Ginzburg-Landau vortices}.
\newblock {\em {Z. Angew. Math. Phys.}}, 46(5):739--751, 1995.

\bibitem{SS2}
E.~Sandier and S.~Serfaty.
\newblock {G}inzburg-{L}andau minimizers near the first critical field have
  bounded vorticity.
\newblock {\em Calc. Var. Partial Differential Equations}, 17(1):17--28, May
  2003.

\bibitem{SS1}
E.~Sandier and S.~Serfaty.
\newblock {\em {Vortices in the magnetic Ginzburg-Landau model}}.
\newblock Birkh\"auser Boston Inc., Boston, MA, 2007.

\bibitem{S1}
S.~Serfaty.
\newblock Local minimizers for the {G}inzburg--{L}andau energy near critical
  magnetic field: Part {I}.
\newblock {\em Comm. Contemp. Math.}, 01(02):213--254, 1999.

\end{thebibliography}

\end{document}